\documentclass[twoside,12pt]{article}
\usepackage{amsmath,amssymb,amsthm,amsfonts}
\usepackage{paralist}
\usepackage{subfigure}
\usepackage[title]{appendix}

\usepackage{graphics} 
\usepackage{epsfig} 
\usepackage{graphicx}
\usepackage{caption}
\usepackage{epstopdf}
\usepackage[colorlinks=true]{hyperref}
\usepackage{float}
\usepackage{pdfpages}
\usepackage{amsfonts}
\usepackage{amsmath,amsthm}
\usepackage{multirow}
\usepackage{cancel}
\usepackage{amsmath}
\usepackage{txfonts}
\usepackage{enumerate}
\usepackage[numbers,sort&compress]{natbib}
\usepackage[font=scriptsize,labelfont=bf,width=0.95\textwidth]{caption}

\numberwithin{equation}{section}
\newcommand{\R}{{\mathbb R}}

\newcommand{\be}{\begin{equation}}
\newcommand{\ee}{\end{equation}}

\newcommand{\ben}{\begin{eqnarray*}}
\newcommand{\enn}{\end{eqnarray*}}

\newtheorem{proposition}{Proposition}[section]
\newtheorem{theorem}{\textbf Theorem}[section]
\newtheorem{lemma}{\textbf Lemma}[section]

 \numberwithin{equation}{section}
\newtheorem{remark}{Remark}[section]
\renewcommand{\theequation}{\arabic{section}.\arabic{equation}}

\usepackage{microtype}

\begin{document}

\title{\textbf{Mountain-Pass Solutions for Second-Order Ergodic Mean-Field Game Systems}}
\author{
	Fanze Kong \thanks{Department of Applied Mathematics, University of Washington, Seattle, WA 98195, USA; fzkong@uw.edu}, 
	Yonghui Tong\thanks{Center for Mathematical Sciences, Wuhan University of Technology, Wuhan 430070, China; myyhtong@whut.edu.cn} and
	Xiaoyu Zeng\thanks{Center for Mathematical Sciences, Wuhan University of Technology, Wuhan 430070, China; xyzeng@whut.edu.cn}
}
\date{\today}
\maketitle
 \abstract{We study the existence of mountain-pass solutions to a potential-free mean-field game system in the whole space $\mathbb R^n$ under the mass-supercritical regime, assuming an aggregating local coupling and a $C^2$ Hamiltonian that is $\gamma$-homogeneous with $\gamma > 1$. Due to the lack of smoothness of the underlying variational structure, the standard deformation lemma and the classical mountain-pass theorem are not directly applicable. To overcome this difficulty, we constrain the nonlinear term and employ a {\em two-stage linearization} argument to establish the existence of least-energy solutions to an auxiliary mean-field game problem with general coercive potentials. In the vanishing coercive potential limit, we recover compactness by using maximal regularity for Hamilton–Jacobi equations together with Pohozaev-type identities, and show that the potential-free mean-field game system admits a classical solution, which is also an optimizer of a Gagliardo–Nirenberg-type inequality. Finally, we analyze the mountain-pass geometry of the variational structure, which yields that the solution obtained above corresponds to a mountain-pass type solution of the original mean-field game system. These results provide an affirmative answer to the longstanding problem concerning the existence of mountain-pass solutions for mean-field game systems. Furthermore, as a byproduct, we relax the admissible set and provide a unified framework for establishing the optimal Gagliardo–Nirenberg inequality below the Sobolev critical exponent. } 



\medskip


\maketitle

  \hypersetup{linkcolor=black}

 \tableofcontents

\section{Introduction}
In this paper, we consider a class of viscous stationary Mean-Field Game (MFG) systems, which are used to describe the long-time behavior of games involving a large number of homogeneous rational players. A typical mathematical model is given as follows:
\begin{equation}\label{MFG-SS}
\left\{
\begin{array}{ll}
-\Delta u+H(\nabla u)+\lambda=f(m)+V(x) , &x \in \mathbb R^n,     \\
 \Delta m+\nabla\cdot (m\nabla H(\nabla u))=0,&x \in \mathbb R^n,\\
\int_{\mathbb R^n} mdx=M>0,
\end{array}
\right.
\end{equation}
where $(m,u,\lambda)$
 denotes a solution triple and $\lambda\in\mathbb R$ is a Lagrange multiplier.  Here the function $m$ represents the population density of players, $u$ is the value function of a typical player, and $M>0$ is the total mass of the population.  In particular, $H:\mathbb R^n\rightarrow \mathbb R$ is the Hamiltonian, $V(x)$ is the potential and $f(m)$ denotes the coupling cost function.  

MFG theories and systems were independently proposed by Lasry et al. \cite{Lasry} and Huang et al. \cite{Huang} to describe the dynamics of complex systems with a large number of indistinguishable agents. Drawing on ideas from statistical physics and kinetic theory, they introduced a class of coupled backward-forward parabolic equations, consisting of a Hamilton-Jacobi equation and a Fokker-Planck equation, which are now referred to as MFG systems. In the setting described below, their mathematical formulation is given by
 \begin{equation}\label{MFG-time}
\left\{
\begin{array}{ll}
u_t= -\Delta u+H(\nabla u)-V(x)-f(m), &x \in \mathbb R^n,t>0,\\
m_t=\Delta m +\nabla\cdot (m\nabla H(\nabla u)),&x \in \mathbb R^n, t>0,\\
u\vert_{t=T}=u_T, m|_{t=0}=m_0,&x\in \mathbb R^n.
\end{array}
\right.
\end{equation}
Here $m_0$ is the initial data of density and $u_T$ is the terminal data of value function $u$.  We remark that (\ref{MFG-SS}) captures the long-time behaviors of the solutions to (\ref{MFG-time}).  We next give a brief overview of the derivation of the models (\ref{MFG-SS}) and (\ref{MFG-time}) from the associated particle system.


Suppose that for $i=1,\cdots,N$, the position of the 
$i$-th player evolves according to the stochastic differential equation (SDE):
\begin{align}\label{game-process-dXti}
dX_t^i=-y^i_t dt+\sqrt{2}dB_t^i, \ \ X_0^i=x^i\in\mathbb R^n,
\end{align}
where $x^i$ denotes the initial state, $y^i_t$ is the control applied by the $i$-th player, and $B_t^i$ are independent Brownian motions.  For simplicity, we consider all agents as indistinguishable and drop the index 
$i$ in \eqref{game-process-dXti}. Each player aims to minimize the average cost, which is typically given by
\begin{align}\label{longsenseexpectation}
J(y_t):=\mathbb E\int_0^T[L(y_t)+V(X_t)+f(m(X_t))] dt + u_T(X_T),
\end{align}
where $L$ is the Lagrangian, defined as the Legendre transform of $H$ satisfying $H(p)=\sup_{y\in\mathbb R^n}(py-L(y))$ given in (\ref{MFG-time}).  Applying the dynamic programming principle in (\ref{longsenseexpectation}), one takes the mean-field limit to formulate the time-dependent system (\ref{MFG-time}).  Similarly, if we consider the following ergodic type average cost minimization problem
\begin{align*}
\bar J:=\limsup_{T\rightarrow +\infty}\inf_{Y_t}\mathbb E\bigg[\frac{1}{T}\int_0^T[L(y_t)+V(X_t)+f(m(X_t))] dt\bigg] ,
\end{align*}
then, by the dynamic programming principle, the associated ergodic system \eqref{MFG-SS} is obtained.

One major focus in the study of \eqref{MFG-time} is the global well-posedness and long-time dynamics, see \cite{Car12,Car13,CGMT13,cirantgoffi2021,GPM12,GPM13,GPV13,ahuja2016wellposedness,gangbo2022global}.  Concerning the numerical analysis of system \eqref{MFG-time}, we refer the reader to the survey \cite{achdou2020mean} and, for instance, to \cite{ACD10, achdou2012mean, CS12, CS13}.  Another active research area focuses on the analysis of solutions to the stationary problem \eqref{MFG-SS}. The systems are classified as focusing or defocusing, depending on the coupling cost.  If the coupling cost is monotone increasing and satisfies the Lasry-Lions monotonicity condition, the MFG system is of the focusing type and admits a unique stationary solution  \cite{Lasry}.  In contrast, for focusing MFG systems, e.g., with a monotone decreasing coupling cost, \eqref{MFG-SS} may have multiple solutions, rendering the existence analysis more challenging. Under certain assumptions on $f$, $V$, and 
$H$, several recent works have studied the existence of solutions to focusing MFG system (\ref{MFG-SS}), see \cite{cesaroni2018concentration, GM15,gomes2016regularity, meszaros2015variational, cirant2016stationary,cirant2025critical,kong2024blow,kong2025local}.
  
This paper addresses the existence of solutions to the focusing MFG system \eqref{MFG-SS} with aggregating coupling, assuming appropriate conditions on $H$, $f$  and $V$.   {In particular, we suppose that  Hamiltonian $H:\mathbb{R}^n\to \mathbb R$ satisfies the following conditions:
\begin{center}
\begin{itemize}
\item[(H1). ] 
$H$ is strictly convex, $H\in C^{2}\left(\mathbb R^n \setminus \{0\} \right)$ and  it is homogeneous of degree $\gamma>1$.   
\item[(H2). ] There exist $C_{H}>0$ such $\inf_{|\boldsymbol{p}|=1}H(\boldsymbol{p})\geq C_{H}>0$ .
\end{itemize}
\end{center}  
Given the above assumptions on $H$, it is straightforward to verify that there exists a constant $C'_H > C_H$ such that
\begin{equation}\label{MFG-H}
    0\leq C_{H}|\boldsymbol{p}|^{\gamma}\leq H(\boldsymbol{p})\leq C'_{H}|\boldsymbol{p}|^{\gamma}, \,\text{for all}\, \boldsymbol{p}\in \mathbb{R}^n.
\end{equation}
Moreover, the Legendre transform of $H$ is defined by $ L(\boldsymbol{q}) := \sup_{\boldsymbol{p}\in\mathbb{R}^n} [\boldsymbol{p}\cdot \boldsymbol{q} - H(\boldsymbol{p})] \, \text{for any } \boldsymbol{q}\in\mathbb{R}^n.$ 
In light of the duality between $H$ and $L$, we infer that $L \in C^{2}(\mathbb{R}^n \setminus \{0\})$, is convex, and homogeneous of degree $\gamma'$ and there exist constants $C'_L > C_L > 0$ such that 
 \begin{align}\label{MFG-L}
0\leq C_{L}|\boldsymbol{q}|^{\gamma'}\leq L(\boldsymbol{q})\leq C'_{L}|\boldsymbol{q}|^{\gamma},\,\text{for any}\,\boldsymbol{q}\in \mathbb{R}^n,
\end{align}
where $\gamma'=\frac{\gamma}{\gamma-1}$ is the conjugate number of $\gamma.$ }

\begin{remark}
    Several examples of Hamiltonian $H$ are given as follows:
\begin{itemize}
    \item[(i).]  $H({\boldsymbol p}) = C_H |{\boldsymbol p}|^{\gamma}$ with constants $\gamma > 1$ and $C_H > 0$;
       \item[(ii).]  $H(\boldsymbol{p}) =\sum_{i=1}^n C^i_{H} |p_i|^{\gamma}$ with $\boldsymbol{p}:=(p_1,\cdots,p_n)^T$  and constants $\gamma > 1$, $C^i_H > 0$.
\end{itemize}
  
\end{remark}

  As shown in \cite{cesaroni2018concentration,cirant2025critical}, nonlinear Schr\"{o}dinger equations are subsystems of (\ref{MFG-SS}).  
There is a rich literature (see e.g. \cite{guo2014massRobert,guo2016energy,bartsch2019multiple,zhang2022normalized}) devoted to the study of normalized solutions to Schr\"{o}dinger equations via the variational stuctures.  
The pioneering work on the existence of solutions to viscous ergodic MFG systems (\ref{MFG-SS}) with aggregating coupling, via the variational method, can be traced to \cite{cesaroni2018concentration}.  In detail,   
they point out that the solution of \eqref{MFG-SS} corresponds to the critical points of the following constrained minimization problem
\begin{align}\label{min_prob}
e_{\alpha,M}:=\inf_{(m,w)\in\mathcal{A}\cap \left\{m: \Vert m\Vert_{L^{1}(\mathbb R^n)}=M\right\}}\mathcal{J}(m,w),\,\,	
\end{align}
where $M>0$ and
\begin{align*}
	\mathcal{J}(m,w)=\int_{\mathbb R^n} mL\bigg(-\frac{w}{m}\bigg)\,dx +\int_{\mathbb R^n} V(x)m\,dx-\int_{\mathbb R^n}F(m)\,dx.
\end{align*}
Here Lagrangian $L$ is defined by
\begin{align}\label{general-Lagrangian}
	L\bigg(-\frac{w}{m}\bigg):=\left\{\begin{array}{ll}
		\sup\limits_{\boldsymbol{p}\in\mathbb R^n}\big(-\frac{\boldsymbol{p}\cdot w}{m}-H(\boldsymbol{p})\big),&m>0,\\
		0,&(m,w)=(0,0),\\
		+\infty,&\text{otherwise},
	\end{array}
	\right.
\end{align}
 $mL\big(-\frac{w}{m}\big)$ is the Legendre transform of $mH(p)$. The admissible set $\mathcal{A}$ is defined as
\begin{align}\label{constraint-set-A}
	\mathcal{A}:=\Big\{&(m,w)\in  (L^1(\mathbb R^n)\cap W^{1,\hat q}(\mathbb R^n))\times L^{1}(\mathbb R^n)\nonumber\\
	~&\text{s. t. }\int_{\mathbb R^n}\nabla m\cdot\nabla\varphi\,dx=\int_{\mathbb R^n}w\cdot\nabla\varphi\, dx,\forall \varphi\in C_c^{\infty}(\mathbb R^n),\int_{\mathbb R^n}V(x) m\,dx<+\infty,~
	m\geq,\not\equiv 0\text{~a.e.~in}\,\mathbb{R}^n\Big\},
\end{align}
where $\hat q$ is given by
\begin{equation}\label{hatqconstraint}
	\hat q:=\begin{cases}
		\frac{n}{n-\gamma'+1} &\text{ if }\gamma'<n,\\
		\in\left(\frac{2n}{n+2},n\right)&\text{ if }\gamma'=n,\\
		\gamma' &\text{ if }\gamma'> n.
	\end{cases}
\end{equation}
Assume that Hamiltonian $H$ and potential $V$ are perturbations of polynomial growth, and that the local coupling $f:[0,+\infty)\rightarrow \mathbb R$  is locally Lipschitz continuous and satisfies
\begin{align*}
-C_fm^{\alpha}-C_{0}\leq f(m)\leq -C_fm^{\alpha}+C_{0},~~\exists\, C_f,C_{0},\alpha>0,
\end{align*}
then, Cirant and Cesaroni \cite{cesaroni2018concentration} showed that (\ref{MFG-SS}) admits a least-energy solution under the condition $0<\alpha<\alpha_{*}:=\frac{\gamma'}{n}$, which is referred to as the mass-subcritical exponent case.  Recently, Cirant et al. \cite{cirant2025critical} analyzed the existence of a least-energy solution to (\ref{MFG-SS}) in the mass-critical case, i.e., $\alpha=\alpha_*.$  However, in the mass-supercritical exponent case, i.e., $\alpha>\alpha_{*}$, the functional $\mathcal J(m,w)$ is unbounded from below on $\mathcal{A}\cap \{m:\Vert m\Vert_{L^1}=M\}$ with $\mathcal A$ given in (\ref{constraint-set-A}), which renders the study of solutions to (\ref{MFG-SS}) via the variational approach particularly challenging.  We mention that there exists another critical exponent $\alpha^{*}:=\frac{\gamma'}{n-\gamma'}$ (if $n\leq \gamma'$, we define $\alpha^{*}:=+\infty$), which arises from the Sobolev embedding when estimating the density $m$ via Fokker-Planck equation.  For the viscous MFG system with Neumann boundary conditions in a bounded domain, Cirant et al. \cite{cirant2023ergodic} established the existence of local minimizers ranging from the mass-supercritical up to the Sobolev-critical case. By contrast, it is believed that local minimizers do not exist in whole space $\mathbb R^n$ and the existence of mountain-pass type solutions to (\ref{MFG-SS}) remains largely open, both in the whole space and in bounded domains. 
       
In this paper, we focus on the existence of {\textit{mountain-pass}} type solutions to \eqref{MFG-SS} in the {{mass-supercritical}} regime, i.e.,   $\alpha>\alpha_*=\frac{\gamma'}{n}$ with $n\geq 2$.  Moreover, we restrict our attention to the Sobolev-subcritical regime, i.e., $\alpha<\alpha^{*}$. As discussed above, in this setting, the functional $\mathcal{J}(m,w)$ is unbounded from below, and consequently the variational approach employed in \cite{cesaroni2018concentration,cirant2025critical} to solve \eqref{MFG-SS} directly via the constrained minimization problem \eqref{min_prob} is no longer applicable.  Furthermore, since the functional $\mathcal{J}(m,w)$ is non-smooth, many classical methods developed for the $L^2$ mass-supercritical Schr\"{o}dinger equation such as those in \cite{Jeanjean1997,wei2022wu,Bartsch2021Verzini} and the references therein, cannot be applied in the present framework. To overcome these difficulties, we adopt an $L^{1+\alpha}$-constrained minimization framework, rather than an $L^1$-constraint one; the auxiliary problem will be analyzed in Section~\ref{MFG_sys}. 
We next state our main results in the following subsection.

\subsection{Main results}\label{mainresults11}
Our main contribution is the existence of mountain-pass solutions to the MFG system \eqref{MFG-SS} in the mass-supercritical regime. These findings complement the results in \cite{cesaroni2018concentration} and \cite{cirant2025critical}, and provide a detailed discussion of the mass supercritical case. Furthermore, as a byproduct, we establish optimal Gagliardo-Nirenberg inequality in the mass-super critical regime.  More specifically, we consider \eqref{MFG-SS} with $V \equiv 0$, $f(m) := -m^{\alpha}$ for $\alpha > \alpha_{*}$, and $H$ satisfying conditions (H1) and (H2) stated above.   

The existence of mountain-pass solutions for MFG system \eqref{MFG-SS} can be transformed to find the formal critical points of the functional $\mathcal{J}_0$ constrained on $\mathcal{A}\cap \left\{m: \Vert m\Vert_{L^{1}(\mathbb R^n)}=M\right\}$, which are associated with the following mountain-pass level
\begin{align}\label{constrain_mountain_pass__problem}
	e_{MP}:=\inf_{h\in \Gamma}\max_{t\in[0,1]} \mathcal{J}_{0} (h(t)), 
\end{align} 
where
\begin{align}\label{L^1_constrain_functional}
	\mathcal{J}_{0} (m,w):=\int_{\R^n} mL\left( -\frac{w}{m}\right)\, dx-\frac{1}{1+\alpha}\int_{\R^n}m ^{1+\alpha}\, dx,
\end{align}
and the path set is defined as 
\begin{align}\label{mountain_pass_path_set}
	\Gamma:=\bigg\{  h\in C\left([0,1],\mathcal{A}\cap \left\{m: \Vert m\Vert_{L^{1}(\mathbb R^n)}=M \right\}\right): h(0)=(m_{1},w_{1})\,\,\text{and}\,\, h(1)=(m_{2},w_{2})\bigg\},
\end{align} 
with $$\int_{\R^n} m_{1}L\left( -\frac{w_{1}}{m_{1}}\right)\, dx\leq R_{0},\ \int_{\R^n} m_{2}L\left(- \frac{w_{2}}{m_{2}}\right)\, dx\geq 2R_{0}, \ \mathcal{J}_{0} (m_{1},w_{1})>0 \text{ and } \mathcal{J}_{0} (m_{2},w_{2})<0$$ for some constant $R_{0}>0$.    

In what follows, we assume that $V(x)$ is a locally H\"{o}lder continuous potential that satisfying  the following conditions:
\begin{center}
\begin{itemize}
\item[(V1). ] 
$V(x)\in L^\infty_{\rm loc}(\mathbb{R}^n)$ and $V(x)\rightarrow +\infty$ as $|x|\rightarrow+\infty$.   
\item[(V2). ] there exist constants $C_1,C_2, K>0$  such that 
\begin{subequations}\label{V2mainasumotiononv}
\begin{align}
    &0< V(x)\leq C_2|x|^b, \ \ |x|\geq K; \label{V2mainassumption_1}\\
   &0< C_1\leq \frac{V(x+y)}{V(x)}\leq C_2,\text{~for~all~}|x|\geq K \text{~with~}|y|<2; \label{V2mainassumption_2}\\
   &\sup_{\nu\in[0,1]}V(\nu x)\leq C_2V(x)\text{~for~}|x|\geq K.\label{V2mainassumption_3}
\end{align}
\end{subequations}
\end{itemize}
\end{center}  
\begin{remark}\label{rmark12mar22}
    Several examples of potential $V$ are given as follows:
\begin{itemize}
    \item[(i).]  $V = C_V |x|^{b}$ with constants $b > 0$ and $C_V > 0$;
       \item[(ii).]  $V=C_V\ln (1+|x|)$ with constant $C_V>0$.
\end{itemize}
\end{remark}

Concerning the min-max problem given in (\ref{constrain_mountain_pass__problem}) and the existence of mountain-pass solutions to (\ref{MFG-SS}), we obtain
\begin{theorem}\label{Existience_of_MPS_MFG}
For any $\gamma'>1$, letting $\alpha_{*}:=\frac{\gamma'}{n}<\alpha<\alpha^{*}$, where $\alpha^*:=\frac{\gamma'}{n-\gamma'}$ if $n> \gamma'$ and $\alpha^*:=+\infty$ if $n\leq \gamma'$, then there exists a classical solution $(\hat u,\hat m, \hat\lambda)\in C^{2}(\mathbb{R}^n)\times W^{1,p}(\mathbb{R}^n)\times \mathbb{R}$, $\forall p>1,$ to the following MFG system
		\begin{align}\label{L^1_constrain_equation}
		\left\{\begin{array}{ll}
			-\Delta u+H(\nabla u)+\lambda=-m^{\alpha},&x\in\mathbb R^n,\\
			\Delta m+\nabla\cdot (m\nabla H(\nabla u))=0,&x\in\mathbb R^n,\\
		 \int_{\mathbb R^n}m \,dx=M.
		\end{array}
		\right.
	\end{align}
     Moreover, the pair $(\hat m, \hat w)\in \mathcal{A}\cap \left\{m: \Vert m\Vert_{L^{1}(\mathbb R^n)}=M\right\}$ with $\hat w = -\hat m\nabla H(\nabla \hat u)$ and $\mathcal{J}_0(\hat m,\hat w)=e_{MP}$, corresponding to the mountain-pass level given in \eqref{constrain_mountain_pass__problem}. 
\end{theorem}  

To prove Theorem \ref{Existience_of_MPS_MFG}, instead of tackling the problem (\ref{constrain_mountain_pass__problem}) directly, we first consider the following constrained minimization problem in $L^{1+\alpha}$:
\begin{align}\label{delta_problem}
	e_{\delta}:=\inf_{(m,w)\in \mathcal{A}\cap \left\{m: \Vert m\Vert_{L^{1+\alpha}(\mathbb R^n)}=1 \right\}}\mathcal E_{\delta}(m,w),
\end{align}
with
\begin{align}\label{delta_energy_functional}
	\mathcal E_{\delta}(m,w):=
	\int_{\mathbb R^n}mL\left (-\frac{w}{m}\right)\,dx+\int_{\mathbb R^n}(\delta V(x)+1)m\,dx,
\end{align}
{where $\delta \in(0,1)$ is a parameter and locally H\"{o}lder continuous potential $V$ satisfies (V1)-(V2). }
{\noindent  We first prove that for any $\alpha\in(\alpha_{*},\alpha^{*})$,  there exists a minimizer $(m,w)$ for problem \eqref{delta_problem} via a regularization argument as shown in \cite{cesaroni2018concentration}.  More importantly,  in Lemma \ref{potential_MFG} below, we prove the existence of a Euler-Lagrange multiplier $\mu$ and a function $u\in C^2(\mathbb{R}^n)$ such that $(u,m,\mu)$ satisfies the following auxiliary MFG  system:
\begin{align}\label{MFG_mass_supercritical-0}
\left\{\begin{array}{ll}
-\Delta u+H(\nabla u)+\mu m^{\alpha}=\delta V(x)+1,&x\in\mathbb R^n,\\
\Delta m+\nabla\cdot (m\nabla H(\nabla u))=0,&x\in\mathbb R^n,\\
w=-m\nabla H(\nabla u),\ \int_{\mathbb R^n}m^{1+\alpha}\,dx=1.
\end{array}
\right.
\end{align}  
The main obstacle lies in the fact that the existence results for ergodic Hamilton–Jacobi equations established in \cite{GuyJoao2016} (see also Lemma~\ref{lemma22preliminary} below) cannot be directly applied owing to the presence of multiplier $\mu$, which forces us to modify the arguments employed in previous works \cite{cesaroni2018concentration,cirant2025critical} for tackling the minimization problem \eqref{min_prob} subject to mass-subcritical and mass-critical exponents.   {By a suitable scaling argument, the minimization problem~(\ref{min_prob}) can be reformulated as the $L^{1+\alpha}$-constrained problem~(\ref{delta_problem}).} After obtaining a minimizer  of (\ref{delta_problem}), the derivation of the associated value function $u$
 via Lemma~\ref{lemma22preliminary} would, however, require at least the uniform positivity condition, $\inf_{x\in\mathbb R^n}m>0$.  Such a condition cannot be satisfied in the present setting  since $m\in L^{1+\alpha}$ and $m\geq 0.$  To overcome this issue, we introduce a two-stage linearization argument on the energy functional to establish the $u$-equation in \eqref{MFG_mass_supercritical-0}.  However, an additional multiplier arises in this process; to show that it must vanish, we employ an appropriate scaling argument combined with variational characterizations of the multipliers.  All details of this procedure are provided in the proof of Lemma \ref{potential_MFG}. 
}

Once we obtain a solution of the auxiliary system  \eqref{MFG_mass_supercritical-0} by studying (\ref{min_prob}).  Then, by taking the limit  $\delta\rightarrow 0^{+}$ in (\ref{delta_problem}), we establish in Lemma \ref{Existience_of_MFG_free_potential} the existence of a classical solution to the auxiliary problem without the potential for MFG system \eqref{MFG_mass_supercritical-0} with $\delta=0$ and the $L^{1+\alpha}$ constraint under the mass-supercritical case.
 However, our main goal is to investigate the existence of a mountain pass solution to the MFG system \eqref{L^1_constrain_equation} under an $L^1$ constraint.  To fill in the gap, we apply a scaling argument and Pohozaev identities to study the relationship between the systems \eqref{L^1_constrain_equation} and \eqref{e_0_min} below. As a consequence, we establish an optimal Gagliardo–Nirenberg type inequality subject to the mass supercritical exponent and the result is summarized as
\begin{theorem}\label{GNinequalitythm}
      Let   $\alpha\in(0,\alpha^{*})$ and $\mathcal{A}$ be given by \eqref{constraint-set-A} with $V$ satisfying \textup{(V1)}-\textup{(V2)}  in (\ref{V2mainasumotiononv}). Then, the best constant
\begin{align}\label{Gn_Inequality}
	\Gamma_{\alpha}:=\inf_{(m,w)\in \mathcal{A}}\frac{\left(\int_{\mathbb R^n}m L\left(-\frac{w}{m}\right)\,dx\right)^{\frac{n\alpha}{\gamma'}}\left(\int_{\mathbb R^n}m\,dx\right)^{\frac{(\alpha+1)\gamma'-n\alpha}{\gamma'}}}{\int_{\mathbb R^n}m^{\alpha+1}\, dx}
\end{align}
can be attained. 
In particular,  {when $\alpha\in(\alpha_*,\alpha^*)$}, the pair $(\hat m, \hat w)$ given in Theorem \ref{Existience_of_MPS_MFG} is an optimizer of $\Gamma_{\alpha}$. 
\end{theorem}  
We recover the Gagliardo-Nirenberg type inequality \eqref{Gn_Inequality} established in \cite{cesaroni2018concentration,cirant2025critical} and further improve upon their results. More precisely, in those works the inequality holds under the assumption that $m$ exhibits polynomial decay. In contrast, the mild condition introduced in \eqref{V2mainasumotiononv} relaxes the restriction on $m$ in \eqref{constraint-set-A}, thereby allowing densities with significantly slower decay (see Remark~\ref{rmark12mar22}) while still ensuring the validity of \eqref{Gn_Inequality}.  We remark that 
problem \eqref{Gn_Inequality} is scaling invariant under $(s^{\beta}m(tx),s^{\beta+1}w(tx))$ for any $s>0$, $t>0$ and $\beta>0.$  As a consequence, it is straightforward to see that \eqref{Gn_Inequality} is equivalent to the following minimization problem:
\begin{align}\label{Gn_Inequality_constrain}
	\Gamma_{\alpha}:=\inf_{(m,w)\in\mathcal{A}\cap \left\{m: \Vert m\Vert_{L^{1+\alpha}(\mathbb R^n)}=1 \right\}}{\left(\int_{\mathbb R^n}m L\left(-\frac{w}{m}\right)\,dx\right)^{\frac{n\alpha}{\gamma'}}\left(\int_{\mathbb R^n}m\,dx\right)^{\frac{(\alpha+1)\gamma'-n\alpha}{\gamma'}}}, \ \ \alpha\in\left(0,\alpha^{*}\right).
\end{align}
Moreover, one can obtain { from Lemma \ref{sect3_lemm32}} that minimization problem \eqref{Gn_Inequality_constrain}
	is equivalent to the following problem
\begin{align}\label{e_0_min}
	e_{0}:=\inf_{(m,w)\in \mathcal{A}\cap \left\{m: \Vert m\Vert_{L^{1+\alpha}(\mathbb R^n)}=1 \right\}}\int_{\mathbb R^n}mL\left(-\frac{w}{m}\right)\,dx+\int_{\mathbb R^n}m\,dx.
\end{align}
 Actually, the minimization problem \eqref{Gn_Inequality_constrain}, and thus \eqref{Gn_Inequality}, is attained by a minimizer $(m_{0},w_{0})\in  \mathcal{A}\cap \left\{m: \Vert m\Vert_{L^{1+\alpha}(\mathbb R^n)}=1 \right\}$ of $e_{0}$.
Moreover, we show that $(m_{0},w_{0})$ corresponds to a solution of the auxiliary potential-free MFG system \eqref{mass_supercritical_problem_free_potential} below. On the other hand, we obtain that $(\hat m, \hat w)\in  \mathcal{A}\cap \left\{m: \Vert m\Vert_{L^{1}(\mathbb R^n)}=M\right\}$, derived by a scaling of $(m_{0},w_{0})$, is also a minimizer of $\Gamma_{\alpha}$, which corresponds to a solution of MFG system \eqref{L^1_constrain_equation}.  Finally,  we show that the solution to \eqref{L^1_constrain_equation} obtained above is indeed a mountain-pass solution by proving  that $\mathcal{J}_{0}( \hat m, \hat w )=e_{MP}$.



\medskip

{
With the existence of a solution
$(\hat u, \hat m, \hat \lambda)$ to \eqref{L^1_constrain_equation}, and  
 an optimizer  $(\hat m, \hat w)$ of
\eqref{Gn_Inequality}, we further classify the optimizer $(\hat m, \hat w)$,
as well as the associated solution $( \hat u, \hat m,\hat \lambda)$ to
\eqref{L^1_constrain_equation}, in terms of the parameter $\alpha$.
This classification recovers several results from
\cite{cesaroni2018concentration} and \cite{cirant2025critical}.
}

\begin{proposition}\label{cover_0_super} 
Let $\alpha\in(0,\alpha^{*})$,
then we have the following properties:
\begin{itemize}
    \item[(i).] When $\alpha\in(0,\alpha_{*})$, the functional $\mathcal{J}_{0}$ given in (\ref{L^1_constrain_functional})
has a global minimizer $(\hat m,\hat w)$ on  $\mathcal{A}\cap \left\{m: \Vert m\Vert_{L^{1}(\mathbb R^n)}=M \right\}$. Moreover, there exists $(\hat u, \hat \lambda)\in(C^2(\mathbb{R}^n)\times\mathbb{R})$ with $\hat w = -\hat m\nabla H(\nabla \hat u)$, such that 
 $(\hat u,\hat m,\hat \lambda)$ is a ground state of  MFG system  \eqref{L^1_constrain_equation}.
   
    \item[(ii).] When $\alpha=\alpha_{*}$, if $M<M^{*}$ with $M^{*}=\Big(\Gamma_{\alpha_{*}}(1+\alpha_{*})\Big)^{\frac{1}{\alpha_{*}}}$, then
    MFG system \eqref{L^1_constrain_equation} does not admit any classical solution; if $M=M^{*}$, then $\mathcal{J}_{0}$ has a global minimizer $(\hat m,\hat w)$ on $\mathcal{A}\cap \left\{m: \Vert m\Vert_{L^{1}(\mathbb R^n)}=M^{*} \right\}$, and $(\hat u,\hat m,\hat \lambda)$ with $\hat w = -\hat m\nabla H(\nabla \hat u)$ is a ground state  of MFG system \eqref{L^1_constrain_equation}; when $M>M^{*}$, then $\mathcal{J}_{0}$ does not have any global minimizer or  mountain-pass type optimizer on $\mathcal{A}\cap \left\{m: \Vert m\Vert_{L^{1}(\mathbb R^n)}=M \right\}$.
    \item [(iii).] When $\alpha\in(\alpha_{*},\alpha^{*})$,  the minimax value $e_{MP}$ defined by  (\ref{constrain_mountain_pass__problem}) of $\mathcal{J}_{0}$ on $\mathcal{A}\cap \left\{m: \Vert m\Vert_{L^{1}(\mathbb R^n)}=M \right\}$  has an  optimizer $(\hat m,\hat w)$, and  $(\hat u,\hat m,\hat \lambda)$ with $\hat w = -\hat m\nabla H(\nabla \hat u)$ is  a mountain-pass  solution of MFG system \eqref{L^1_constrain_equation}.
\end{itemize}
\end{proposition}  

Proposition \ref{cover_0_super} classifies the optimizer of $\Gamma_{\alpha}$ defined in (\ref{Gn_Inequality}) and the detailed argument is shown in Appendix \ref{cover_0_superappen}.

The remainder of this paper is organized as follows.  In Section \ref{preliminaries}, we present preliminary results on the existence and regularity of the stationary Hamilton-Jacobi and Fokker-Planck equations. In Section \ref{MFG_sys}, we consider the $L^{1+\alpha}$ constrained minimization problem as an auxiliary problem and discuss the existence of its solutions both with and without potentials. Section \ref{mountain_pass_solution_MFG} is devoted to the existence of a mountain-pass solution, as stated in Theorem \ref{Existience_of_MPS_MFG} for the MFG system \eqref{Existience_of_MPS_MFG}.  In particular, the optimal Gagliardo-Nirenberg type inequality below the Sobolev-critical exponent is established, which is shown in Theorem \ref{GNinequalitythm}. 
\medskip

\section{Some preliminary existence and regularity results}\label{preliminaries}
Noting that \eqref{MFG-SS} consists of a coupled system of Fokker–Planck and Hamilton–Jacobi equations, we begin by collecting several preliminary results concerning the existence and regularity of solutions to these two equations, respectively.  In subsection \ref{subsection1}, we are concerned with Hamilton-Jacobi equations.  Subsection \ref{subsection2} is devoted to the existence and regularity of Fokker-Planck equations.
\subsection{Hamilton-Jacobi equations}\label{subsection1}
 This subsection is dedicated to the investigation of the Hamilton-Jacobi equation associated with a Hamiltonian $H$ satisfying assumptions \textup{(H1)}--\textup{(H2)}.  First of all, we {establish} the local maximal regularity property for the Hamilton-Jacobi equation as follows.
\begin{lemma}\label{thmmaximalregularity}
 Let $p>\frac{n}{\gamma'}$,  $ \gamma\geq \frac{n}{n-1}$, $f\in L^p(\Omega)$ and assume $u\in W^{2,p}(\Omega)$ solves
    \begin{align}\label{ueqlemma21lpmaximal}
	-\Delta u+ H(\nabla u)= f\text{~in~}\Omega \subset \mathbb{R}^n,
    \end{align}
	in the strong sense.  Then for each $K>0$ and $\Omega'\subset\subset \Omega$, we have
	{$$ \Vert |\nabla u|^\gamma\Vert_{L^p(\Omega')} + \Vert D^2 u\Vert_{L^p(\Omega')}\leq C,$$}
	where $\Vert f\Vert_{L^p(\Omega)}\leq K $ and the constant $C=C(K,\mathrm{dist}(\Omega',\partial\Omega),n,p, C_H,\gamma)>0$.
\end{lemma}
\begin{proof}
{Since the argument is similar to that in the proofs of \cite[Theorem~1.1]{cirant2025critical} and \cite[Theorem~1.3]{cirant2022local}, we only provide a brief outline for completeness. First, under assumption (H1) and assumption (H2), one can adapt the arguments in \cite[Lemma~2.9]{cirant2025critical} (see also \cite[Lemma~2.5]{cirant2022local}), where the Hamiltonian is taken as $H(\boldsymbol{p}) = |\boldsymbol{p}|^{\gamma}$ for $\boldsymbol{p} \in \mathbb{R}^n$, to obtain the  following Liouville-type result: 
 
   Let $u \in W_{\mathrm{loc}}^{2,p}(\mathbb{R}^n)$ with $p > \frac{n}{\gamma'}$ be a solution to 
\[
-\Delta u + C_0 H(\nabla u) = 0 \quad \text{in } \mathbb{R}^n,
\]
where $C_0 \geq 0$ is a constant. If, in addition, $\nabla u \in L^{r',q}(\mathbb{R}^n)$ for some $q \in (0,n]$, then $u$ is constant, i.e., $u \equiv C$ for some $C \in \mathbb{R}$.

  In light of Liouville result stated above, Bernstein method yields weighted Morrey-type estimates for solutions to \eqref{ueqlemma21lpmaximal}. Specifically, if $u \in W^{2,p}(\Omega)$ with $p > \frac{n}{\gamma'}$ and $\gamma' \geq \frac{n}{n-1}$ is a strong solution and $\|f\|_{L^{p}(\Omega)} \leq M$, then there exists $C = C(M,n,p,\gamma,\Omega) > 0$ such that
\[
\sup_{B_{2R}(x_{0}) \subset \Omega} 
R^{q} \int_{B_{R}(x_{0})} |\nabla u|^{\gamma}\, dx \,
\bigl(\mathrm{dist}(B_{R}(x_{0}), \partial\Omega)\bigr)^{\gamma' - q}
\leq C.
\]
   
  Finally, the proof of Lemma~\ref{thmmaximalregularity} follows by combining the preceding estimates with \cite[Lemma~2.8]{cirant2025critical}.   } 
\end{proof}   


Now, we consider the following form of Hamilton-Jacobi equations:
\begin{align}\label{HJB-regularity}
	-\Delta u_k+H(\nabla u_k)+\lambda_{k}=f_k(x)+V_k(x),\ \ x\in\mathbb R^n,
\end{align}
{where $H$ satisfies (H1)-(H2)} and $(u_k,\lambda_k)$ denotes a sequence of solutions to \eqref{HJB-regularity}.
Focusing on the regularities and the lower bounds of $u_k$, we state them in the following lemmas.  
\begin{lemma}\label{sect2-lemma21-gradientu}
	Assume that $f_k\in L^{\infty}(\mathbb R^n)$ satisfies  $\Vert f_k\Vert_{L^\infty(\mathbb{R}^n)}\leq C_f$,  $\vert\lambda_{k}\vert\leq \lambda$ and the potential functions $V_k(x)\in C^{0,\theta}_{\rm loc}(\mathbb R^n)$ with $\theta\in(0,1)$ satisfy $0\leq V_k(x)\rightarrow +\infty$ as $|x|\rightarrow +\infty,$ and $\exists~ R>0$ sufficiently large such that 
	\begin{align*}
		0< C_1\leq \frac{V_k(x+y)}{V_k(x)}\leq C_2,\text{~for~all~}k\text{~and~all~}|x|\geq R \text{~with~}|y|<2,
	\end{align*}
	where the positive constants $C_f$, $\lambda$, $R$, $C_1$ and $C_2$ are independent of $k$.  Let $(u_k,\lambda_k)\in C^2(\mathbb R^n)\times \mathbb R$ be a sequence of solutions to \eqref{HJB-regularity}. Then, for all $k$,
	\begin{align*}
		|\nabla u_k(x)|\leq C(1+V_k(x))^{\frac{1}{\gamma}}, \text{ for all } x\in\mathbb{R}^n,
	\end{align*}
	where the constant $C$ depends on $C_H$, $C_1$, $C_2$, $\lambda$, $\gamma$, $n$ and $C_f.$
	
	In particular, if there exist   $b\geq 0$ and $C_{F}>0$  independent of $k,$ such that following conditions hold on $V_k$
	\begin{align}\label{cirant-Vk}
		C_F^{-1}(\max\{|x|-C_F,0\})^b\leq V_k(x)\leq C_F(1+|x|)^b,~~\text{for all }k\text{ and }x\in\mathbb R^n,
	\end{align}
	then  we have 
	\begin{align*}
		|\nabla u_k|\leq C(1+|x|)^{\frac{b}{\gamma}}, ~\text{for all }k\,\,\text{and}\,\, x\in\mathbb R^n,
	\end{align*}
	where the  constant $C$ depends on $C_H$, $C_{F}$, $b$, $\lambda$, $\gamma$, $n$ and $C_f.$
    \end{lemma}
\begin{proof}
   {
    To begin with, by applying Bernstein method \cite[Theorem~A.1]{Lasry1989Lions} together with assumptions (H1)-(H2), it is straightforward to verify that if $v \in C^{2}(B_{2}(0))$ satisfies
    \begin{align*}
        \lvert -\nabla v+H(\nabla v)\rvert \leq K\,\text{on}\,B_{2}(0)\,\text{with some constant}\, K>0,
    \end{align*}
   then for any $r\in [1,\infty)$ it holds
   \begin{align*}
        \lVert \nabla v\rVert_{L^{r}(B_{1}(0))} \leq \tilde {K},
    \end{align*}   
    where $\tilde{K}$ depends on $K$, $\gamma$, $n$ and $C'_{H}$ given in \eqref{MFG-L}.  Using classical elliptic regularity estimates, we then obtain
\[
\|\nabla v\|_{L^{\infty}(B_{1}(0))} \leq \tilde{K}.
\]
The lemma is then concluded by an argument similar to that shown in \cite[Theorem~2.5]{cesaroni2018concentration} and \cite[Lemma~3.1]{cirant2025critical}. }
\end{proof}
 
Besides the gradient estimates of $u_k$, we also have the following results for the lower bounds of $u_k$: 
\begin{lemma}\label{lowerboundVkgenerallemma22}
	Suppose all conditions in Lemma \ref{sect2-lemma21-gradientu} hold. 
	Let $u_k$ be a family of $C^2$ solutions and assume that $u_k(x)$ are bounded from below uniformly.  Then there exist positive constants $C_3$ and $C_4$ independent of $k$ such that 
	\begin{align}\label{29uklemma22}
		u_k(x)\geq C_3V^{\frac{1}{\gamma}}_k(x)-C_4,\text{~}\forall x\in\mathbb R^n,~\text{for all }k.
	\end{align}
	In particular, if the following conditions hold on $V_k$
	\begin{align}\label{cirant-Vk-1}
		C_F^{-1}(\max\{|x|-C_F,0\})^b\leq V_k(x)\leq C_F(1+|x|)^b,~~\text{for all }k\text{ and }x\in\mathbb R^n,
	\end{align}
	where constants $b> 0$ and $C_{F}$ are independent of $k,$ then we have 
	\begin{align}\label{usolutionlowerestimatepre-11}
		u_k(x)\geq C_3|x|^{1+\frac{b}{\gamma}}-C_4,\text{~for all }k, x\in\mathbb R^n.
	\end{align}
	If $b=0$ in (\ref{cirant-Vk-1}) and there exist $R>0$ and $\sigma>0 $ independent of $k$ such that 
	\begin{align*}
 V_k-\lambda_{k}+ f_k>\sigma>0,\text{~for~all }|x|>R,
	\end{align*}
	then \eqref{usolutionlowerestimatepre-11} also holds.
\end{lemma}
\begin{proof}
 The proof follows along the same lines as that of Lemma~3.2 in \cite{cirant2025critical}.  
\end{proof}

Concerning the existence of classical solutions to (\ref{HJB-regularity}), we have the following results:
\begin{lemma}[ C.f. \cite{cesaroni2018concentration,cirant2025critical} ]  \label{lemma22preliminary}
	Suppose $V_{k}$ and $f_{k}$ are locally H\"{o}lder continuous and $V_{k}+f_{k}$ is bounded below uniformly in $k$.  Define 
	\begin{align}\label{def_lambda}
		\bar \lambda_k:=\sup\left\{\lambda_k\in\mathbb R~|~(\ref{HJB-regularity})\text{ has a solution }u_k\in C^2(\mathbb R^n)\right\}.
	\end{align}
	Then 
	\begin{itemize}
		\item[(i).]  For each $k$, $\bar \lambda_k$ is finite and \eqref{HJB-regularity} admits a solution $(u_k,\lambda_k)\in C^2(\mathbb R^n)\times \mathbb R$ with $\lambda_k=\bar \lambda_k$, where $ u_k(x)$ is bounded from below.  Moreover,
		$$\bar \lambda_k=\sup\left\{\lambda_k\in\mathbb R~|~(\ref{HJB-regularity})\text{ has a subsolution }u_k\in C^2(\mathbb R^n)\right\}.$$
		\item[(ii).] If $V_k$ satisfies \eqref{cirant-Vk} with $b>0$, then $u_k$ is unique up to constants for fixed $k$ and there exists a positive constant $C$ independent of $k$ such that 
		\begin{align}\label{lowerboundusect2}
			u_k(x)\geq C|x|^{\frac{b}{\gamma}+1}-C^{-1},\, \forall x\in\mathbb R^n.
		\end{align}
		In particular, if $V_k\equiv 0$, $b=0$ in \eqref{cirant-Vk-1} and there exists $\xi>0$ independent of $k$ such that 
		\begin{align}\label{verifylemma22}
			f_k-\lambda_{k} \geq \xi>0,\ \ \text{for}~|x|>K_2,
		\end{align}
		where $K_2>0$ is a large constant independent of $k$, then \eqref{lowerboundusect2} also holds.
	\end{itemize}
	(iii). If $V_k$ satisfies (V1)-(V2) and positive constants $C_1$, $C_2$ and $K$ independent of $k,$ then there exist uniformly bounded from below classical solutions $u_k$ to problem \eqref{HJB-regularity} satisfying estimate \eqref{29uklemma22}.
\end{lemma} 
\begin{proof}

We first prove statement $(i)$. Note that $H$ is assumed to satisfy (H1)-(H2), which guarantees that $H(\boldsymbol{p}) \asymp |\boldsymbol{p}|^{\gamma}$.
Therefore, the conclusion of $(i)$ follows from a minor modification of the argument in the proof of Theorem 2.1 in \cite{GuyJoao2016}.
As noted in \cite{cesaroni2018concentration}, \cite{GuyJoao2016} assumes that $f_k + V_k \in W^{1,\infty}_{\mathrm{loc}}(\mathbb{R}^n)$ to obtain gradient estimates for the solutions. Lemma~\ref{sect2-lemma21-gradientu} shows that this requirement can be weakened to local H\"older continuity of $f_k + V_k$, while still ensuring classical elliptic regularity.


 To establish $(ii)$, we invoke the results of \cite{Ichihara2011siam} (see also \cite{GuyJoao2016}). Specifically, since the bounds on the coefficients are independent of $n$, by following the approaches in \cite{Ichihara2011siam}, one establishes that the sequence $(u_k)$ admits a uniform lower bound. Consequently, Lemma~\ref{lowerboundVkgenerallemma22} directly yields the desired gradient estimate.
 

Finally, we establish conclusion $(iii)$ and present a proof sketch following the argument in \cite[Theorem~2.6]{GuyJoao2016}.
 To begin with, we consider the ergodic problem
\begin{align}\label{ergodic_pro}
-\Delta u_{k}^{R}+H(\nabla u_{k}^{R})+\lambda_{k}^{R}=f_{k}+V_{k}, \,\, x\in B_{R}(0),
\end{align}
with state constraint boundary conditions given by
\begin{align*}
\begin{cases}
u_{k}^{R}(x)\to +\infty, &\text{as }x\to \partial B_{R}(0),\,\, \text{if } \gamma\in(1,2]\; \,(\text{subquadratic and quadratic cases}),\\[4pt]
-\Delta u_{k}^{R}+H(\nabla u_{k}^{R})+\lambda_{k}^{R}\geq f_{k}+V_{k}, &\text{on }\partial B_{R}(0),\,\,\text{if } \gamma>2\;\, (\text{superquadratic case}),
\end{cases}
\end{align*}
where $\gamma$ is the degree of   Hamiltonian $H$ shown in (H1). Following the proofs in \cite{Guy2010sub} for the subquadratic and quadratic cases, and \cite{Tchamba2010super} for the superquadratic case, we have that for any given $R>0$, there exist a unique $\lambda_k^R$ and a unique (up to an additive constant) $u_k^R \in C^2(B_R)$ satisfying \eqref{ergodic_pro} (see also \cite{Lasry1989Lions}).

We now claim that $\lim_{R \to \infty} \lambda_k^R = \bar{\lambda}_k.$ Indeed, the sequence $(\lambda_k^R)$ is monotone nonincreasing in $R$ and bounded from below by $\bar{\lambda}_k$, so that it converges to some $\lambda_k^* \ge \bar{\lambda}_k$.  Following the argument of Lemma~\ref{sect2-lemma21-gradientu}, for every compact set $\Omega \subset \mathbb{R}^n$, there exists a constant $C>0$ such that $|\nabla u_k^R| \le C$ on  $\Omega$ 
for all sufficiently large $R$ and for every $k$.  Without loss of generality, we may assume that $u_k^R(0) = 0$. The gradient estimate, together with classical elliptic regularity, implies that $u_k^R$ is uniformly bounded in $C^2(\Omega)$ with respect to $R$. By the Ascoli-Arzel\`a theorem and a standard diagonalization argument, there exists $u_k \in C^2(\mathbb{R}^n)$ such that  $u_k^R \to u_k \,\text{in } C^2_{\mathrm{loc}}(\mathbb{R}^n) \text{ as } R \to +\infty.$
Moreover, $u_k$ satisfies \eqref{HJB-regularity} with $\lambda = \lambda_k^*$. In view of the definition of $\bar{\lambda}_k$ in \eqref{def_lambda}, it follows that $\lambda_k^* = \bar{\lambda}_k$.


Now let $x_{k}^R$ be a point where $u_{k}^R$ attains its minimum in $B_R(0)$. Noting that $u_{k}^{R}$ solves \eqref{ergodic_pro} and $H(0)=0$,  a direct computation at $x_{k}^R$ yields $\lambda_{k}^R-V_{k}(x_{k}^R)- f_{k}(x_{k}^R) = \Delta u_{k}^R(x_{k}^R) \geq 0$. Using the coercivity of $V_{k}$ together with the convergence $\lambda_k^R \to \bar{\lambda}_k$, we can find a compact set $\Omega$ (independent of $k$ and $R$) and some $R_0>0$ such that $x_k^R \in \Omega$ for all $R> R_0$. Since $u_k^R(0)=0$ and the gradient of $u_{k}^R$ is bounded on $\Omega$, we obtain that $u_k^R(x_{k}^R) \geq -C$ for some constant $C>0$ that does not depend on $k$ and $R$. As $x_{k}^R$ is a minimal point of $u_{k}^R$, we have $u_{k}^R(x) \geq u_{k}^R(x_{k}^R) \geq -C$ for any $x \in B_R(0)$. Passing to the limit as $ R\to\infty$ yields $u_{k}(x) \geq -C$ for all $x\in \mathbb{R}^n$ with a constant $C$ independent of $k$.
\end{proof}

With the a priori estimates and existence results for solutions to the Hamilton–Jacobi equation \eqref{HJB-regularity}, we next discuss the regularities of solutions to Fokker-Planck equations, which is presented in Subsection \ref{subsection2}.

\subsection{Fokker-Planck equations}\label{subsection2}
Before stating the gradient estimates satisfied by solutions to the Fokker–Planck equations, we first recall the following key lemma, which holds for any function $m\in L^p(\mathbb R^n)$ 
\begin{lemma}\label{proposition-lemma21-FP}
 Suppose $p>1$ and $m\in L^p(\mathbb R^n)$ such that
 \begin{align*}
 \Big|\int_{\mathbb R^n}m\Delta\varphi\,dx\Big|\leq N\Vert\nabla \varphi\Vert_{L^{p'}(\mathbb R^n)}\ \ \text{~for~all~}\varphi\in C_c^{\infty}(\mathbb R^n),
 \end{align*}
 where $N>0$ is a positive constant.  Then we have $m\in W^{1,p}(\mathbb R^n)$ and $\Vert\nabla m\Vert_{L^p(\mathbb R^n)}\leq C_pN,
 $
 where $C_p$ is a positive constant depending only on $p$.
\end{lemma}
\begin{proof}
The proof can be found in Proposition 2.4 of \cite{cesaroni2018concentration}.
\end{proof}
Now, we are concerned with the following Fokker-Planck equations:
\begin{align}\label{sect2-FP-eq}
-\Delta m+\nabla\cdot w=0,\ \ x\in\mathbb R^n,
\end{align}
where $w$ is given and $m$ denotes the solution.  With the aid of Lemma \ref{proposition-lemma21-FP}, we can obtain the crucial a-priori estimates satisfied by $m$.  To begin with, we recall that $\hat q$ is defined as \eqref{hatqconstraint}, and set 
${\hat q}^*=\frac{n\hat q}{n-\hat q}$ if $\hat q<n$, and ${\hat q}^*=+\infty$ if $\hat q\geq n$. Choose
 $\beta\in [\hat q,{\hat q^*}]$ such that 
$ \frac{1}{\hat q}=\frac{1}{\gamma'}+\frac{1}{\gamma \beta}$.
Then one can deduce from \eqref{hatqconstraint} that 
\begin{equation*}
\beta=\begin{cases}{\hat q}^*, \ &\text{ if }\gamma'<n,\\
\in(\hat q,{\hat q}^*), \ &\text{ if }\gamma'=n,\\
\infty,&\text{ if }\gamma'>n.\end{cases}
\end{equation*}
Set 
\begin{equation*}
0<\mathcal{S}_{\hat q,\gamma'}^{-1}:=\inf_{m\in W^{1, \hat q}(\mathbb R^n)} \frac{\|\nabla m\|^\theta_{L^{\hat q}(\mathbb R^n)}\|\nabla m\|^{1-\theta}_{L^{\hat q}(\mathbb R^n)}}{\|m\|_{L^{\beta}(\mathbb R^n)}}<\infty,\text{ where }\theta\in[0,1] \text{ satisfying }\frac{1}{\beta}=\theta(\frac{1}{{\hat q}}-\frac{1}{n})+1-\theta.
\end{equation*}
We then state the following lemma, which addresses the regularity of solutions to equation \eqref{sect2-FP-eq}.
\begin{lemma}[ C.f. Lemma 3.4 in \cite{cirant2025critical} ]\label{lemma21-crucial}    
Assume that $(m,w)\in \left(L^1(\mathbb R^n)\cap W^{1, \hat q}(\mathbb R^n)\right)\times L^1(\mathbb R^n)$ is a  solution to (\ref{sect2-FP-eq}) and
\begin{equation*}
\Lambda_{\gamma'}:=\int_{\mathbb R^n} m\Big|\frac{w}{m}\Big|^{\gamma'}\,dx<\infty.
\end{equation*}
Then, we have $w\in L^{1}(\mathbb R^n)\cap L^{\hat q}(\mathbb R^n)$ and there exists constant $\mathcal{C}=\mathcal{C}(\Lambda_{\gamma'},\|m\|_{L^1(\mathbb R^n)})>0$ such that
$$\|m\|_{W^{1,\hat q}(\mathbb R^n)}, \|w\|_{L^1(\mathbb R^n)},\|w\|_{L^{\hat q}(\mathbb R^n)}\leq \mathcal{C}.$$ More precisely, we have
\begin{equation*}
\|\nabla m\|_{L^{\hat q}(\mathbb R^n)}\leq \mathcal{S}_{\hat q,\gamma'}^{\frac{1}{\gamma-\theta}}\left(C_{\hat q} \Lambda_{\gamma'}^\frac{1}{\gamma'}\right)^{\frac{\gamma}{\gamma-\theta}}\|m\|_{L^1(\mathbb R^n)}^\frac{1-\theta} {\gamma-\theta},~
~
\|m\|_{L^{\hat q}(\mathbb R^n)}\leq  \mathcal{S}_{\hat q,\gamma'}^{\frac{1}{\gamma-\theta}} \left(C_{\hat q} \Lambda_{\gamma'}^\frac{1}{\gamma'}\right)^{\frac{\theta}{\gamma-\theta}}\|m\|_{L^1(\mathbb R^n)}^{\frac{1-\theta} {\gamma-\theta}+\frac{1}{\gamma'}},
\end{equation*}
and 
\begin{equation*}
\|w\|_{L^1(\mathbb R^n)}\leq  \Lambda_{\gamma}^\frac{1}{\gamma'}\|m\|_{L^1(\mathbb R^n)}^\frac{\gamma'-1}{\gamma'},~~\|w\|_{L^{\hat q}(\mathbb R^n)}\leq \Lambda_{\gamma'}^\frac{1}{\gamma'} \left(\mathcal{S}_{\hat q,\gamma'}\right)^{\frac{1}{\gamma-\theta}}\left(C_{\hat q} \Lambda_{\gamma'}^\frac{1}{\gamma'}\right)^{\frac{\theta}{\gamma-\theta}}\|m\|_{L^1(\mathbb R^n)}^\frac{1-\theta} {\gamma-\theta},
\end{equation*}
where 
\begin{equation*}
\theta=\frac{n\gamma'(\hat q-1)}{(r-1)(nq-n+q)}=\begin{cases} 1, \ &\text{ if }\gamma'<n,\\
\frac{n^2(\hat q-1)}{(n-1)(n\hat q-n+\hat q)} \ &\text{ if }\gamma'=n,\\
\frac{n\gamma'}{n\gamma'-n+\gamma'},&\text{ if }\gamma'>n.
\end{cases}
\end{equation*}
\end{lemma}

Next, we turn our attention to the system \eqref{MFG-SS}, which is a coupled system consisting of a Hamilton-Jacobi equation and a Fokker-Planck equation. Under suitable assumptions on the population density 
$m$ and the Lagrange multiplier $\lambda,$ concerning the decay properties of $m$, we obtain the following lemma 
\begin{lemma}\label{mdecaylemma}
Assume that $(u,m,\mu)\in C^2(\mathbb R^n)\times \left(W^{1,p}(\mathbb R^n)\cap L^1(\mathbb R^n)\right)\times \mathbb R$ with $u$ bounded from below, $ p>n$ and $\mu>0$, is the solution of the following potential-free problem  
\begin{align}\label{26preliminaryfinal}
\left\{\begin{array}{ll}
-\Delta u+H(\nabla u)+\mu m^{\alpha}=1, &x\in\mathbb R^n,\\
\Delta m+\nabla\cdot(m\nabla H(\nabla u))=0, &x\in\mathbb R^n.
\end{array}
\right.
\end{align}
 Then, there exist $\kappa_1,\kappa_2>0$ such that 
\begin{align*}
m(x)\leq \kappa_1 e^{-\kappa_2|x|}  ~\text{ for all } x\in \mathbb R^n.
\end{align*}
\end{lemma}
\begin{proof}
By slightly modifying the proof of Lemma 3.6 in \cite{cirant2025critical}, we obtain the desired conclusion.
\end{proof}
We next collect the Pohozaev identities satisfied by the solution to (\ref{26preliminaryfinal}) in the following lemma  
\begin{lemma}\label{poholemma}
Let $(u,m,\mu)$ satisfy the assumptions of Lemma \ref{mdecaylemma} and denote $w=-m\nabla H(\nabla u)$. Then the following identities hold:
\begin{align}\label{eq2.49}
\left\{\begin{array}{ll}
\frac{1}{\mu}\int_{\mathbb R^n}m\, dx=\frac{(\alpha+1)\gamma'-n\alpha}{(\alpha+1)\gamma'}\int_{\mathbb R^n}m^{\alpha+1}\,dx,\\
\int_{\mathbb R^n}m L\left(-\frac{w}{m}\right)\, dx=\frac{n\mu \alpha}{(\alpha+1)\gamma'}\int_{\mathbb R^n}m^{1+\alpha}\, dx=(\gamma-1)\int_{\mathbb R^n} m H(\nabla u)\, dx.
\end{array}
\right.
\end{align}
\end{lemma}
\begin{proof}
Similarly as the argument shown in Lemma 3.7  of  \cite{cirant2025critical},
using the exponential decay property of $m$ obtained in Lemma \ref{mdecaylemma}, we test the first equation in \eqref{26preliminaryfinal} against $\nabla m \cdot x$ and the second equation against $\nabla u \cdot x$, take the sum of the resulting expressions, and perform integration by parts to deduce
\begin{align}\label{mar18poho1}
 n\int_{\mathbb R^n}m\,dx-\frac{n\mu }{\alpha+1}\int_{\mathbb R^n}m^{\alpha+1}\,dx+ (\gamma-{n})\int_{\mathbb R^n} m H(\nabla u)\,dx+(2-n)\int_{\mathbb R^n}\nabla u\cdot \nabla m\,dx=0.
\end{align}
We multiply the $m$-equation by $u$ in (\ref{26preliminaryfinal}) and integrate it by parts to get
\begin{align}\label{225pohobefore18}
\int_{\mathbb R^n} \nabla m\cdot \nabla udx=-\int_{\mathbb R^n}m\nabla H(\nabla u)\cdot\nabla u\,dx=-\gamma\int_{\mathbb R^n}m H(\nabla u)\,dx.
\end{align}
where in the second equality, we used assumption (H1) satisfied by $H.$  By using (\ref{mar18poho1}) and (\ref{225pohobefore18}), one has
\begin{align}\label{combine1mar18}
 n\int_{\mathbb R^n}m\,dx-\frac{n\mu }{\alpha+1}\int_{\mathbb R^n}m^{\alpha+1}\,dx+\frac{n-\gamma'}{\gamma'-1}\int_{\mathbb R^n} mH(\nabla u)\,dx=0.
\end{align}

In addition, we multiply the $u$-equation and the $m$-equation in (\ref{26preliminaryfinal}) by $m$ and $u$, respectively, then subtract them and integrate by parts to obtain
\begin{align}\label{mar18poho2}
(1-\gamma)\int_{\mathbb R^n}mH(\nabla u)\,dx-\int_{\mathbb R^n}m\,dx =-\mu\int_{\mathbb R^n}m^{\alpha+1}\, dx.
\end{align}
Combining (\ref{combine1mar18}) with (\ref{mar18poho2}), we find the first identity in (\ref{eq2.49}) holds.  Moreover, we invoke (\ref{mar18poho2}) and the first identity in (\ref{eq2.49}) to get
\begin{align}\label{togetmar18plug}
\frac{n\mu \alpha}{(\alpha+1)\gamma'}\int_{\mathbb R^n}m^{1+\alpha}\, dx=(\gamma-1)\int_{\mathbb R^n} m H(\nabla u)\, dx.
\end{align}
In addition, by using the facts that $H$ is strictly convex, $w=-m\nabla H(\nabla u)$ and assumption (H1), we have 
\begin{align}\label{finalpohomar18}
(\gamma-1)H(\nabla u)=L\left(-\frac{w}{m}\right).
\end{align}
Substituting (\ref{finalpohomar18}) into (\ref{togetmar18plug}), we obtain the second identity in (\ref{eq2.49}) holds.

\end{proof}

Now, we are ready to study the existence of solutions to the auxiliary problem.  More specifically, we shall prove Lemma \ref{potential_MFG} and Lemma \ref{Existience_of_MFG_free_potential} in Section \ref{MFG_sys}.


\section{Existence of Solutions to Auxiliary MFG Systems}\label{MFG_sys}
This section is devoted to the investigation of minimization problem (\ref{delta_problem}) and the corresponding limit of $\delta \rightarrow 0^{+}.$  We shall show that the minimizer of (\ref{delta_problem}) is associated with a solution to an auxiliary MFG system with the coercive potential, which is shown in Lemma \ref{potential_MFG}.   Then by taking the limit of $\delta\rightarrow 0^{+},$ we show the potential-free auxiliary system admits a classical solution, where the key ingredient is the introduction of a {\em {two-stage}} linearization argument for finding the value function 
 $u$, based on the variational relation associated with the multiplier  $\mu.$ 

\subsection{Auxiliary systems with the coercive potential}\label{potential_MFG_sys}
In this subsection, we shall prove that the minimization problem \eqref{delta_problem}  with energy $\mathcal{E}_{\delta}$ given by \eqref{delta_energy_functional} has a minimizer $(m_{\delta},w_{\delta})\in\mathcal{A}\cap \left\{m: \Vert m\Vert_{L^{1+\alpha}(\mathbb R^n)}=1 \right\}$ for any $\delta\in(0,1)$. Moreover, we show that there exists $(u_{\delta},\mu_{\delta})\in C^{2}(\mathbb{R}^{n})\times \mathbb{R}$ such that $(u_{\delta},m_{\delta},\mu_{\delta})\in  C^{2}(\mathbb{R}^{n})\times W^{1,p}(\mathbb{R}^{n})\times\mathbb{R}$ ($\forall p>1$) is a classical solution to the auxiliary MFG system \eqref{MFG_mass_supercritical}.   We remark that the regularity of $m$ established in Lemma \ref{lemma21-crucial} will be weak if $\gamma'$ is small. Consequently, a regularization and linearization procedure for the constraint set is necessary to obtain the solution of \eqref{MFG_mass_supercritical}.  Actually, by Lemma \ref{lemma21-crucial}, if $\gamma'>n$, one can directly obtain uniform boundedness of 
 $m_k$ in $L^{\infty}(\mathbb R^n)$ and $C^{0,\theta}(\mathbb R^n)$ for some $\theta\in(0,1)$, when $(m_k,w_k)$ is a minimizing sequence.  In contrast, when $\gamma'\leq n$,  such uniform $L^\infty$ bounds are no longer available, and regularization is required to ensure the {$L^\infty$}-boundedness of 
$m.$ 

Next, we first focus on the case of $\gamma'\leq n$. To this end, let $\eta_{k}\geq 0$ be the standard mollifier, satisfying
$$\int_{\mathbb R^n}\eta_{k}\,dx=1, \ \ \text{supp}\eta_{k}\in B_{\frac{1}{k}}(0),$$
for $k>0$,  we consider the following auxiliary minimization problem 
\begin{align}\label{delta_problem_tilde}
	\tilde{e}_{\delta, k}:=\inf_{(m,w)\in\mathcal{A}\cap \left\{m: \Vert m\ast \eta_{k}\Vert_{L^{1+\alpha}(\mathbb R^n)}=1\right\}}\mathcal E_{\delta}(m,w),
\end{align}
where {$\mathcal{E}_{\delta}(m,w)$ is given by \eqref{delta_energy_functional}. Here and in the following of this section, $\alpha\in(0,\alpha^{*})$, $\delta\in(0,1)$ is a parameter,} $\mathcal{A}$ denotes the constraint set defined in \eqref{constraint-set-A} and $V$ satisfies the conditions (V1) and (V2).


By applying the mollification and then taking the limit, we can establish the existence of a solution to the auxiliary MFG system. The corresponding results are summarized as follows.
{
\begin{lemma}\label{potential_MFG}
Assume that $V$ satisfies (V1)-(V2). Let $\gamma'>1$ and $\alpha\in(0,\alpha^{*})$. Then, for each $\delta\in(0,1)$, the minimization problem $e_{\delta}$ in \eqref{delta_problem} admits a minimizer $(m_{\delta},w_{\delta}) \in \mathcal{A}\cap \left\{m: \Vert m\Vert_{L^{1+\alpha}(\mathbb R^n)}=1 \right\}$. Correspondingly, there exists a classical solution ${\big(u_{\delta}, m_{\delta},\mu_{\delta}\big)\in  C^2(\mathbb R^n)\times W^{1,p}(\mathbb R^n)\times \mathbb R}$, $\forall p>1,$ with $\mu_{\delta}=e_{\delta}$, to the following auxiliary MFG system:
\begin{align}\label{MFG_mass_supercritical}
\left\{\begin{array}{ll}
-\Delta u_{\delta}+H(\nabla u_{\delta})+\mu_{\delta} m_{\delta}^{\alpha}=\delta V(x)+1,&x\in\mathbb R^n,\\
\Delta m_{\delta}+\nabla\cdot (m_{\delta} \nabla H(\nabla u_{\delta}))=0,&x\in\mathbb R^n,\\
w_{\delta}=-m_{\delta} \nabla H(\nabla u_{\delta}),\ \int_{\mathbb R^n}m_{\delta}^{1+\alpha}\,dx=1.
\end{array}
\right.
\end{align}
\end{lemma}   }

Before proving Lemma \ref{potential_MFG}, we collect several key lemmas and propositions that will be used in the subsequent arguments.
\begin{lemma}\label{lem4.1}
	Let $V(x)\in L^\infty_{\rm loc}(\mathbb R^n)$ satisfy $\lim\limits_{|x|\to+\infty}V(x)=+\infty$, and define
$$\mathcal{W}_{p,V}:=\bigg\{m \big|\ m\in W^{1,p}(\mathbb R^n)\cap L^1(\mathbb R^n) \text{ and }\int_{\mathbb R^n}V(x)|m|\,dx<\infty\bigg\}.$$
	Then, the embedding $\mathcal{W}_{p,V}\hookrightarrow L^q(\mathbb R^n)$ is compact for any $1\leq q<p^*$, where $p^*=\frac{np}{n-p}$ if $1\leq p<n$ and $p^*=\infty$ if $p\geq n$.
\end{lemma}
\begin{proof}
	See Theorem XIII.67 in \cite{reed2012methods} or Theorem 2.1 in \cite{bartsch1995} for details.
\end{proof}

{
\begin{lemma}\label{blowupanalysismlinfboundcritical}
	Suppose that V is locally Hölder continuous and satisfies (V1) and (V2) given in \eqref{V2mainasumotiononv}. Let $(u_{k},m_{k},\mu_{k})\in C^2(\mathbb R^n)\times \left(L^1(\mathbb R^n)\cap L^{1+\alpha}(\mathbb R^n)\right)\times \mathbb R$, $\alpha\in(0,\alpha^{*}) $, be solutions to the following systems
	\begin{align}\label{eq-potential-newest}
		\left\{\begin{array}{ll}
			-\Delta u+H(\nabla u)+\mu_k g_k[m]=\delta_k V(x+y_{k})+C, &x\in\mathbb R^n,\\
			\Delta m+\nabla\cdot(m\nabla H(\nabla u))=0, &x\in\mathbb R^n,\\
			\int_{\mathbb R^n}G_{k}[m]\,dx=c,
		\end{array}
		\right.
	\end{align}  
 {where $y_k$ is a sequence, $\delta_k$ is a bounded sequence,  $c,C>0$ are constants independent of $k$},
 $G_{k}[m]:=\int_{0}^ {m}g_{k}[s]\,ds$ and $g_{k}[\cdot]$ satisfies for all $m\in L^p(\mathbb R^n)$, $p\in[1,\infty]$ and $k\in \mathbb N$,
	\begin{equation}\label{gkm-newest-estimate-crucial2024207}
		\|g_k[m]\|_{L^p(\mathbb R^n)}\leq K\bigg(\|m^{\alpha}\|_{L^p(\mathbb R^n)}+1\bigg) \ \text{ for some }K>0,
	\end{equation}
	and 
	\begin{equation}\label{gkm-newest-estimate-crucial202420711}
		\|g_k[m]\|_{L^p(B_R(x_0))}\leq K\bigg(\|m^{\alpha}\|_{L^p(B_{2R}(x_0))}+1\bigg) \ \text{ for any }R>0\text{ and }x_0\in \mathbb R^n.
	\end{equation}
	Assume that
	\begin{align}\label{assumptionsincrucialemmanewest}
		\sup_k\|m_k\|_{L^1(\mathbb R^n)}<\infty,~~\sup_k\|m_k\|_{L^{1+\alpha}(\mathbb R^n)}<\infty,~~\sup_k\int_{\mathbb R^n}V(x)m_k\,dx<\infty,~~\sup_k|\mu_k|<\infty,
	\end{align}
	and for all $k$, $u_k$ is bounded from below uniformly.  Then we have 
	\begin{equation}\label{uniformlyboundmkinlinfty}
		\limsup_{k\to\infty}\|m_k\|_{L^\infty(\mathbb \R^n)}<\infty.
	\end{equation}
    \end{lemma}   }

\begin{proof}
The proof is based on a modification of the argument for Lemma 5.2 in \cite{cirant2025critical}. For completeness, we provide a brief proof. To obtain local uniform estimates of $m_k$	
 , we observe from \eqref{assumptionsincrucialemmanewest} and \eqref{gkm-newest-estimate-crucial2024207} that
\begin{align*}
\sup_{k}\|g_k[m]\|_{L^{1+\frac{1}{\alpha}}(\mathbb R^n)}\leq K\bigg(\|m^{\alpha}\|_{L^{1+\frac{1}{\alpha}}(\mathbb R^n)}+1\bigg) <\infty.
\end{align*}
Then, we use the local H\"{o}lder continuity of $V$ and define
$$D_R:=\left\{x:|\delta_kV(x+y_k)|\leq R\right\},\forall R>0,$$ then apply the maximal regularities shown in Lemma \ref{thmmaximalregularity} to obtain
$
\sup_k\Vert |\nabla u_k|^{\gamma}\Vert_{L^{1+\frac{1}{\alpha }}(D_{{R}/{2}})}<\infty,
$
which implies
\begin{align*}
\sup_{k}\Vert |\nabla u_k|^{\gamma-1}\Vert_{L^{{(1+\frac{1}{\alpha})\gamma'}}(D_{R/2})}<\infty\text{~with~}\bigg(1+\frac{1}{\alpha}\bigg)\gamma'>n.
\end{align*}
Focusing on Fokker-Planck equations, we have from Theorem 1.6.5 in \cite{bogachev2022fokker} that 
\begin{align*}
\sup_{k}\Vert m_k\Vert_{L^\infty(D_{R/4})}<\infty.
\end{align*}

Next, we claim that 
\begin{align}\label{crucialandnewclaim}
\lim_{R\rightarrow +\infty}\sup_{k}\bigg\Vert \frac{m_k^{\alpha}(\cdot)}{\delta_k V(\cdot+y_k)}\bigg\Vert_{L^\infty(D^c_R)}=0.
\end{align}
To show (\ref{crucialandnewclaim}), we argue by contradiction and assume there exist $\varepsilon>0,$ $x_l$ and $k_l\rightarrow +\infty$ such that
\begin{align}\label{411contradictionkeystuff}
\frac{m_{k_l}^{\alpha}(x_l)}{\delta_{k_l}V(x_l+y_{k_l})}\geq \varepsilon,
\end{align}
where $|\delta_{k_l}V(x_l+y_{k_l})|\rightarrow+\infty.$  Then we define
\begin{align}\label{rescaling2024206}
v_l(x)=a_l^{\gamma'-2}u_{k_l}(x_l+a_lx),~~\eta_l(x)=a_l^{\frac{\gamma'}{\alpha}} m_{k_l}(x_l+a_lx),
\end{align}
where $a_l^{\gamma'}=\frac{1}{\delta_{k_l}V(x_l+y_{k_l})}\rightarrow 0$.  Upon substituting (\ref{rescaling2024206}) into (\ref{eq-potential-newest}), one obtains 
\begin{align}\label{413rescale2024206}
\left\{\begin{array}{ll}
-\Delta v_l+H(\nabla v_l)=a_l^{\gamma'}C+a_l^{\gamma'}{\delta_{k_l}}V(x_l+y_{k_l}+a_lx)-a^{\gamma'}_l\mu_{k_l}g_l[ a_l^{-\frac{\gamma'}{\alpha}}\eta_l], &x\in\mathbb R^n,\\
\Delta \eta_l+\nabla\cdot(\eta_l \nabla H(\nabla v_l))=0,&x\in\mathbb R^n.
\end{array}
\right.
\end{align}
 In light of the assumption (\ref{V2mainassumption_2}), we have 
\begin{align}\label{mildassumptiononVinlemmacritical}
\Vert a_l^{\gamma'}{\delta_{k_l}}V(x_l+a_l x+y_{k_l})\Vert_{L^\infty(B_1(0))}\leq C_2.
\end{align}
Combining (\ref{gkm-newest-estimate-crucial2024207}) with (\ref{gkm-newest-estimate-crucial202420711}), one obtains for $l$ large,
\begin{align}\label{418combinecrucialnewest1}
\Vert a^{\gamma'}_l g_l[\eta_l a_l^{-\frac{\gamma'}{\alpha}}]\Vert_{L^{{1+\frac{1}{\alpha}}}(B_1(0))}= &{a^{\gamma'}_l}\Vert g_l[\eta_l a_l^{-\frac{\gamma'}{\alpha}}]\Vert_{L^{{1+\frac{1}{\alpha}}}(B_1(0))}\nonumber\\
\leq & a^{\gamma'}_lK(\Vert \eta_l^{\alpha} a_l^{-\gamma'}\Vert_{L^{1+\frac{1}{\alpha}}(B_2(0))}+1)\leq K\left( \Vert \eta_l^{\alpha} \Vert_{L^{1+\frac{1}{\alpha}}(B_2(0))}+1\right).
\end{align}
On the other hand, one has from \eqref{assumptionsincrucialemmanewest} and $a_l\rightarrow 0$ that
\begin{align}\label{419combinecrucialnewest2}
\mu_{k_l}^{1+\frac{1}{\alpha}}\Vert \eta_l^{\alpha}\Vert^{1+\frac{1}{\alpha}}_{L^{1+\frac{1}{\alpha}}(B_2(0))}=\mu_{k_l}^{1+\frac{1}{\alpha}}a_l^{\gamma'+\frac{\gamma'}{\alpha}-n}\Vert m_{k_l}\Vert^{1+\alpha}_{L^{1+\alpha }(B_{2a_l}(x_l))}\rightarrow 0\text{~as~}l\rightarrow +\infty,
\end{align}
where we used $\frac{\gamma'}{n}<\alpha<\frac{\gamma'}{n-\gamma'}$ if $\gamma'<n$ and $\frac{\gamma'}{n}<\alpha<+\infty$ if $\gamma'\geq n.$
Then we combine \eqref{mildassumptiononVinlemmacritical}, (\ref{418combinecrucialnewest1}) with (\ref{419combinecrucialnewest2}) and similarly use the maximal regularities shown in \cite{cirant2021problem} to get 
\begin{align*}
\Vert |\nabla v_l|^{\gamma}\Vert_{L^{1+\frac{1}{\alpha}}(B_{{1}/{2}}(0))}\leq C, \text{~for~}l~\text{large},
\end{align*}
where constant $C>0$ is independent of $l.$  Then focusing on the second equation of (\ref{413rescale2024206}), we similarly apply the standard elliptic regularity estimates (See Theorem 1.6.5 in \cite{bogachev2022fokker}) to obtain $\eta_l\in C^{0,\theta}(B_{1/4}(0))$ with $\theta\in(0,1)$.  By using the local H\"{o}lder's continuity of $\eta_l$, we have from (\ref{411contradictionkeystuff}) that 
\begin{align*}
\eta^{\alpha}_l(0)=m_{k_l}^{\alpha}(x_l)a_l^{\gamma'}=\frac{m_{k_l}^{\alpha}(x_l)}{\delta_{k_l}V(x_l+y_{k_l})}\geq \varepsilon,
\end{align*}
which implies there exist $\bar\xi, R>0$ such that 
$\eta_l(x)\geq\bar \xi \text{ for }x\in B_R(0)$. Then we find
\begin{align*}
\int_{\mathbb R^n}m_{k_l} \delta_{k_l}V(x+y_{k_l})\,dx&\geq \bar \xi a_l^{-n}\delta_{k_l}\int_{B_{a_lR}(0)}V(x_l+y+y_{k_l})\,dy
\geq \frac{\bar\xi}{2}a_l^{-n}\delta_{k_l}V(x_l+y_{k_l})|B_{a_lR}(0)|\\
&\geq C\bar \xi\delta_{k_l}V(x_l+y_{k_l})\rightarrow +\infty,
\end{align*}
  which is contradicted to (\ref{assumptionsincrucialemmanewest}) then proves claim (\ref{crucialandnewclaim}).

  With (\ref{crucialandnewclaim}), we next construct a Lyapunov function and use it to establish uniform bounds for $\|m_k\|_{W^{1,q}}$ with $q>n$. By the Sobolev embedding theorem, this immediately yields the desired conclusion \eqref{uniformlyboundmkinlinfty}. The  argument follows similarly to Lemma 5.2 in \cite{cirant2025critical}, and is therefore omitted.

\end{proof}

With the compactness and regularity results established in Lemma \ref{lem4.1} and Lemma \ref{blowupanalysismlinfboundcritical}, we now turn to the minimization problem \eqref{delta_problem_tilde}. The following result concerns the existence of a minimizer for \eqref{delta_problem_tilde}:

\begin{proposition}\label{exi_min_original}
	For every $\delta\in(0,1)$ and $\alpha\in(0,\alpha^{*})$, there exists constant $C>0$ independent of $k$ and $\delta$  such that $0\leq\tilde e_{\delta,k}\leq C$ and $\tilde e_{\delta,k}$ has a minimizer $(\tilde{m}_{\delta,k},\tilde{w}_{\delta,k})\in \mathcal{A}\cap \left\{m: \Vert m\ast \eta_{k}\Vert_{L^{1+\alpha}(\mathbb R^n)}=1 \right\}$, 
	where $\tilde{e}_{\delta, k}$ is defined in \eqref{delta_problem_tilde}.  In addition,
	\begin{align}\label{min_delta_pro_2}
		\tilde{m}_{\delta,k}(1+V)\in L^{1}(\mathbb{R}^{n}),\,\, \tilde{w}_{\delta,k}(1+V)^{\frac{1}{\gamma}}\in L^{1}(\mathbb{R}^{n}).
	\end{align} 
\end{proposition}
\begin{proof}
	First, we deduce from the definition of $\mathcal{E}_{\delta}$ that $\tilde{e}_{\delta, k}\geq 0$. Now, let $\tilde{m}=c_{k}e^{-|x|}$ and $\tilde{w}=\nabla \tilde{m}$, where $c_{k}$ is chosen such that $(\tilde{m},\tilde{w})\in \mathcal{A}\cap \left\{m: \Vert m\ast \eta_{k}\Vert_{L^{1+\alpha}(\mathbb R^n)}=1\right\}$. One can check that $c_{k}$ is uniformly bounded in $k$. Hence, we obtain $0\leq \tilde{e}_{\delta,k}\leq C$ for some $C>0$ independent of $k>0$ and $\delta\in(0,1)$.
    Now, let $(m_l,w_l)\in \mathcal{A}\cap \left\{m: \Vert m\ast \eta_{k}\Vert_{L^{1+\alpha}(\mathbb R^n)}=1 \right\}$ be a minimizing sequence of $\tilde{e}_{\delta,k}$, namely, $\mathcal{E}_{0}(m_l,w_l)\to \tilde{e}_{\delta, k}$ as $l\to\infty$. Then, one gets 
	\begin{align}\label{energy_est}
		C_{L}\int_{\mathbb R^n} m_{l} L\left(-\frac{w_{l}}{m_{l}}\right) \,dx +\int_{\mathbb R^n} \left(\delta V(x) +1\right)m_{l}\,dx\leq \tilde{e}_{\delta, k}+1\leq C+1.
	\end{align}
	 Therefore, by Lemma \ref{lemma21-crucial} and (\ref{MFG-L}), we obtain that 
	\begin{align*}
		\|m_l\|_{W^{1,\hat{q}}(\mathbb{R}^n)},\,\, \|w_l\|_{L^{\hat{q}}(\mathbb{R}^n)}\leq C, 
	\end{align*}
	for some $\hat q$, $C>0$ independent of $k$.
	As a consequence, there exists $(\tilde{m}_{\delta,k},\tilde{w}_{\delta,k})\in W^{1,\hat{q}}(\mathbb{R}^n)\times L^{\hat{q}}(\mathbb{R}^n)$ such that 
	\begin{align*}
		(m_l,w_l) \rightharpoonup (\tilde{m}_{\delta,k},\tilde{w}_{\delta,k})\,\, \text{weakly in}\,\, W^{1,\hat{q}}(\mathbb{R}^n)\times L^{\hat{q}}(\mathbb{R}^n)\,\, \text{as}\,\,l\to+\infty. 
	\end{align*}
 Applying the Sobolev compact embedding results stated in Lemma \ref{lem4.1}, we obtain that $m_l\to \tilde{m}_{\delta,k}$ strongly in $L^{1}(\mathbb{R}^n)\cap L^{1+\alpha}(\mathbb{R}^n)$, and thus $\int_{\mathbb{R}^n}(\tilde{m}_{\delta,k}\ast\eta_{k})^{1+\alpha}\,dx=1$. As a result,  $(\tilde{m}_{\delta,k},\tilde{w}_{\delta,k})\in\mathcal{A}\cap \left\{m: \Vert m\ast \eta_{k}\Vert_{L^{1+\alpha}(\mathbb R^n)}=1 \right\}$.  Then we conclude that $(\tilde{m}_{\delta,k},\tilde{w}_{\delta,k})$ is a minimizer of $\tilde{e}_{\delta, k}$  by the weak lower semicontinuity of the functional $\mathcal{E}_{\delta}(m,w)$. Finally, by using \eqref{energy_est} and Fatou's lemma, we deduce that 
	\begin{align}\label{min_est_2}
		\int_{\mathbb{R}^n}\tilde{m}_{\delta,k}L\left(-\frac{\tilde{w}_{\delta,k}}{\tilde{m}_{\delta,k}}\right)\,dx\leq C \,\,\text{and}\,\,\int_{\mathbb R^n} (V(x) +1)\tilde{m}_{\delta,k}\,dx\leq C_{\delta},
	\end{align}
    where $C>0$ is independent of $k>0$ and $\delta\in (0,1)$, and $C_{\delta}>0$ depends only on $\delta\in(0,1)$.
	Therefore, we get
\begin{equation*}
\int_{\mathbb{R}^n}|\tilde{w}_{\delta,k}|\,dx\leq \int_{\mathbb{R}^n}|\tilde{w}_{\delta,k}|\left(V(x)+1\right)^{\frac{1}{\gamma}}\,dx\leq  \left(\int_{\mathbb{R}^n}\tilde{m}_{\delta,k}\left|\frac{\tilde{w}_{\delta,k}}{\tilde{m}_{\delta,k}}\right|^{\gamma'}\,dx\right)^{\frac{1}{\gamma'}} \left(\int_{\mathbb R^n}\tilde{m}_{\delta,k} \left(V(x)+1\right)\,dx\right)^{\frac{1}{\gamma}}\leq C_{\delta},
\end{equation*}
	which, combining with \eqref{MFG-L}, indicates that $\tilde{w}_{\delta,k}\in L^{1}(\mathbb{R}^n)$ and \eqref{min_delta_pro_2} holds.  
  This completes the proof of this lemma.
\end{proof}

In Proposition \ref{exi_min_original}, we proved the existence of a minimizer $(\tilde{m}_{\delta,k},\tilde{w}_{\delta,k})$ of $\tilde e_{\delta,k}$. However, due to the presence of nonlinear constraints in $\mathcal{A}\cap \left\{m: \Vert m\ast \eta_{k}\Vert_{L^{1+\alpha}(\mathbb R^n)}=1 \right\}$, we cannot obtain solutions of the auxiliary MFG system \eqref{MFG_mass_supercritical} by utilizing the minimizer $(\tilde{m}_{\delta,k},\tilde{w}_{\delta,k})$.  The key difficulty here lies in exchanging the supremum and infimum in the augmented Lagrangian functional, as discussed in \cite{cesaroni2018concentration}. To overcome this issue, we consider the minimization problem with linearized constraints (first-stage linearization), as follows
\begin{align}\label{delta_problem_bar}
	\bar{e}_{\delta, k}:=\inf_{(m,w)\in \mathcal{A}\cap \left\{m: \left\Vert \left[\left(\tilde{m}_{\delta,k}\ast\eta_{k}\right)^{\alpha}\ast\eta_{k}\right]m\right\Vert_{L^{1}(\mathbb R^n)}=1\right\}}\mathcal E_{\delta}(m,w),
\end{align}
where $\mathcal{E}_{\delta}$ is defined in \eqref{delta_energy_functional}
and $\tilde{m}_{\delta,k}$ is obtained in Proposition \ref{exi_min_original}. The following result concerns the attainability of $\bar{e}_{\delta, k}$ and the relationship between the minimizers of $\tilde{e}_{\delta, k}$ and $\bar{e}_{\delta, k}$.
\begin{proposition}\label{exi_min_appro_re}
	For any $\delta\in(0,1)$ and  $\alpha\in(0,\alpha^{*})$, minimization problem \eqref{delta_problem_bar} is attained by some $(\bar{m}_{\delta,k},\bar{w}_{\delta,k})\in \mathcal{A}\cap \left\{m: \left\Vert \left[\left(\tilde{m}_{\delta,k}\ast\eta_{k}\right)^{\alpha}\ast\eta_{k}\right]m\right\Vert_{L^{1}(\mathbb R^n)}=1 \right\}$.
	Moreover, the minimizers satisfy
	\begin{align}\label{min_delta_pro_4}
		\bar{m}_{\delta,k}(1+V)\in L^{1}(\mathbb{R}^{n}),\,\, \bar{w}_{\delta,k}(1+V)^{\frac{1}{\gamma}}\in L^{1}(\mathbb{R}^{n})\,\,\text{and}\,\,\bar{m}_{\delta,k}\ast\eta_{k}(x)=\tilde{m}_{\delta,k}\ast\eta_{k}(x)\,\,\text{a.e. in}\,\,\mathbb{R}^{n},
	\end{align} 
where $\tilde{m}_{\delta,k}$ is obtained in Proposition \ref{exi_min_original}. Furthermore,  $(\bar{m}_{\delta,k},\bar{w}_{\delta,k}) \in  \mathcal{A}\cap \big\{m: \Vert m\ast\eta_{k}\Vert_{L^{1+\alpha}(\mathbb R^n)}=1 \big\}$  is the minimizer of $\tilde{e}_{\delta, k}=\bar{e}_{\delta, k}$. 
\end{proposition}
\begin{proof}
  The proof of the existence of the minimizer for $\bar e_{\delta, k}$
 is similar to that of Proposition \ref{exi_min_original}, and we therefore omit it. Recall that  $(\tilde{m}_{\delta,k},\tilde{w}_{\delta,k}) \in \mathcal{A}\cap \left\{m: \Vert m\ast \eta_{k}\Vert_{L^{1+\alpha}(\mathbb R^n)}=1 \right\}$ given in Proposition \ref{exi_min_original} is a minimizer of $\tilde{e}_{\delta,k}$. Then we have $$(\tilde{m}_{\delta,k},\tilde{w}_{\delta,k})\in \mathcal{A}\cap \left\{m: \left\Vert \left[\left(\tilde{m}_{\delta,k}\ast\eta_{k}\right)^{\alpha}\ast\eta_{k}\right]m\right\Vert_{L^{1}(\mathbb R^n)}=1 \right\},$$
and thus
\begin{equation}\label{rela_energy}
	\bar{e}_{\delta, k}=\inf_{(m,w)\in \mathcal{A}\cap \left\{m: \left\Vert \left[\left(\tilde{m}_{\delta,k}\ast\eta_{k}\right)^{\alpha}\ast\eta_{k}\right]m\right\Vert_{L^{1}(\mathbb R^n)}=1 \right\}}\mathcal E_{\delta}(m,w)\leq \mathcal E_{\delta}(\tilde{m}_{\delta,k},\tilde{w}_{\delta,k})=\tilde{e}_{\delta, k}.
\end{equation}
On the other hand, let $(\bar{m}_{\delta,k},\bar{w}_{\delta,k})\in  \mathcal{A}\cap \left\{m: \left\Vert \left[\left(\tilde{m}_{\delta,k}\ast\eta_{k}\right)^{\alpha}\ast\eta_{k}\right]m\right\Vert_{L^{1}(\mathbb R^n)}=1 \right\}$ be a minimizer of $\bar{e}_{\delta, k}$. Define
\[
\tilde{m}_{\delta,k}^{*}=\frac{\bar{m}_{\delta,k}}{\|\bar{m}_{\delta,k}\ast\eta_{k}\|_{L^{1+\alpha}(\mathbb{R}^n)}}\,\,\text{and}\,\, \tilde{w}_{\delta,k}^{*}=\frac{\bar{w}_{\delta,k}}{\|\bar{m}_{\delta,k}\ast\eta_{k}\|_{L^{1+\alpha}(\mathbb{R}^n)}}.
\]
It is easily seen that $(\tilde{m}_{\delta,k}^{*},\tilde{w}_{\delta,k}^{*})\in \mathcal{A}\cap \left\{m: \Vert m\ast \eta_{k}\Vert_{L^{1+\alpha}(\mathbb R^n)}=1 \right\}$, then  
\begin{equation}\label{rela_energy_1}
	\begin{aligned}
		\tilde{e}_{\delta, k}&=\inf_{(m,w)\in\mathcal{A}\cap \left\{m: \Vert m\ast \eta_{k}\Vert_{L^{1+\alpha}(\mathbb R^n)}=M \right\}}\mathcal E_{\delta}(m,w)\leq \mathcal E_{\delta}(\tilde{m}_{\delta,k}^{*},\tilde{w}_{\delta,k}^{*})\\
		&=\int_{\mathbb R^n}\tilde{m}_{\delta,k}^{*} L\left (-\frac{\tilde{w}_{\delta,k}^{*}}{\tilde{m}_{\delta,k}^{*}}\right )\,dx +\int_{\mathbb R^n} \left(\delta V(x) +1\right)\tilde{m}_{\delta,k}^{*}\,dx\\
		&=\frac{1}{\|\bar{m}_{\delta,k}\ast\eta_{k}\|_{L^{1+\alpha}(\mathbb{R}^n)}} \left\{ \int_{\mathbb R^n}\bar{m}_{\delta,k}L\left (-\frac{\bar{w}_{\delta,k}}{\bar{m}_{\delta,k}}\right )\,dx +\int_{\mathbb R^n}\left(\delta V(x) +1\right)\bar{m}_{\delta,k}\,dx\right\}\\
		&=\frac{1}{\|\bar{m}_{\delta,k}\ast\eta_{k}\|_{L^{1+\alpha}(\mathbb{R}^n)}} \bar{e}_{\delta, k}.	
	\end{aligned}
\end{equation}
Combining \eqref{rela_energy} with \eqref{rela_energy_1}, one finds  
\begin{equation}\label{rela_energy_2}
	\|\bar{m}_{\delta,k}\ast\eta_{k}\|_{L^{1+\alpha}(\mathbb{R}^n)}\leq 1.
\end{equation}
Moreover, we apply H\"older's inequality and \eqref{rela_energy_2} to get 
\begin{equation}\label{rela_energy_3}
	\begin{aligned}
		1&=\int_{\mathbb R^n}\left[\left(\tilde{m}_{\delta,k}\ast\eta_{k}\right)^{\alpha}\ast\eta_{k}\right]\bar{m}_{\delta,k}\,dx=\int_{\mathbb R^n}\left(\tilde{m}_{\delta,k}\ast\eta_{k}\right)^{\alpha}\left(\bar{m}_{\delta,k}\ast \eta_{k}\right)\,dx\\
	    &\leq \left( \int_{\mathbb R^n}\left(\tilde{m}_{\delta,k}\ast\eta_{k}\right)^{1+\alpha}\,dx \right)^{\frac{\alpha}{1+\alpha}} \left( \int_{\mathbb R^n}\left(\bar{m}_{\delta,k}\ast \eta_{k}\right)^{1+\alpha}\,dx \right)^{\frac{1}{1+\alpha}}\leq 1.
	\end{aligned}
\end{equation}
Therefore, ‘‘=’’ holds in \eqref{rela_energy_3}. As a consequence, we deduce that $\bar{m}_{\delta,k}\ast \eta_{k}= C\tilde{m}_{\delta,k}\ast \eta_{k}\,\,\text{a.e.\,in}\,\, \mathbb{R}^n$ for some $ C> 0$, which, together with the fact $ \int_{\mathbb R^n}(\tilde{m}_{\delta,k}\ast\eta_{k})^{1+\alpha}\,dx=1$ and \eqref{rela_energy_3}, implies that $ C=1$. Hence, we obtain that $\bar{m}_{\delta,k}\ast \eta_{k}=\tilde{m}_{\delta,k}\ast \eta_{k}\,\text{a.e.\,in}\,\,\mathbb{R}^n$. Finally, collect \eqref{rela_energy}, \eqref{rela_energy_1} and \eqref{rela_energy_3}, we conclude that $\tilde{e}_{\delta, k}=\bar{e}_{\delta, k}$,  which completes the proof. 
\end{proof}    

Proposition \ref{exi_min_appro_re} indicates that solving the minimization problem \eqref{delta_problem_tilde} is equivalent to solving the minimization problem \eqref{delta_problem_bar}. In other words, $\tilde{e}_{\delta, k}$ and $\bar {e}_{\delta, k}$ share the same minimizers and  $\tilde{e}_{\delta, k}=\bar {e}_{\delta, k}$.  Next, we directly analyze the minimizers of problem \eqref{delta_problem_bar} and show that there is  a solution to the corresponding auxiliary MFG system. Owing to the form of linearized constraint in \eqref{delta_problem_bar}, it is necessary to apply the linearization argument once again for finding the value function $u$ to exchange the minimum and maximum, which, whereas, leads to the appearance of an extra Lagrange multiplier $\lambda$.  Fortunately, by using the scaling argument, we can further show that $\lambda$ vanishes.  We refer to this procedure as the {\it{two-stage linearization}} argument, and by applying it, we obtain the following result. 
\begin{proposition}\label{classcal_sol_mfg}
	Let $(\bar m_{\delta,k},\bar w_{\delta,k})\in \mathcal{A}\cap \left\{m: \left\Vert \left[\left(\tilde{m}_{\delta,k}\ast\eta_{k}\right)^{\alpha}\ast\eta_{k}\right]m\right\Vert_{L^{1}(\mathbb R^n)}=1 \right\}$ be a minimizer of $\bar e_{\delta, k}$ defined in \eqref{delta_problem_bar}.  Then, there exists a solution $(\bar u_{\delta, k},\bar m_{\delta, k},\bar \mu_{\delta, k})$ to system
    \begin{equation}\label{MFG_mass_supercritical_linear}
         \begin{aligned}  
		\left\{\begin{array}{ll}
			-\Delta u_{\delta,k}+H(\nabla u_{\delta,k})+\mu \left(m_{\delta,k}\ast\eta_k\right)^{\alpha}\ast\eta_{k}=\delta V(x)+1,&x\in\mathbb R^n,\\
			\Delta m_{\delta,k}+\nabla\cdot (m_{\delta,k}\nabla H(\nabla u_{\delta,k}))=0,&x\in\mathbb R^n,\\
			w_{\delta,k}=-m_{\delta,k}\nabla H(\nabla u_{\delta,k}),\, \int_{\mathbb R^n}\left(m_{\delta,k}\ast\eta_{k}\right)^{1+\alpha}\,dx=1,
		\end{array}
		\right.
	\end{aligned}
\end{equation} 
where $\bar \mu_{\delta,k}=\bar e_{\delta,k}$ is uniformly bounded with respect to $k>0$ and $\delta\in(0,1)$.   Moreover, there exists $C_{\delta,k}>0$ such that 
\begin{equation}\label{estimate_gra}
	|\nabla \bar u_{\delta, k}(x)|\leq C_{\delta,k}(1+V(x))^{\frac{1}{\gamma}} \,\,\text{and}\,\, \bar u_{\delta, k}(x)\geq C_{\delta,k}(1+V(x))^{\frac{1}{\gamma}}-C_{\delta,k}^{-1}.
	\end{equation}
\end{proposition}
\begin{proof}
The approach what we shall employ is based on Proposition 3.4 in \cite{cesaroni2018concentration}. More specifically, we define the space of test functions
\begin{equation*}
\mathcal{B}= \left\{\psi\in C^{2}(\mathbb{R}^n):\limsup_{|x|\to\infty}\frac{|\nabla \psi(x)|}{V^{\frac{1}{\gamma}}(x)} <\infty,\,\, \limsup_{|x|\to\infty}\frac{|\Delta \psi(x)|}{V(x)} <\infty \right\}.
\end{equation*}
Then, similar as in \cite[Proposition 3.4]{cesaroni2018concentration}, we conclude that
\begin{equation}\label{test_2}
	-\int_{\mathbb{R}^n}\bar m_{\delta,k} \Delta\psi\,dx=\int_{\mathbb{R}^n}\bar w_{\delta, k}\cdot\nabla\psi\,dx,\,\, \text{for all}\, \psi\in \mathcal{B}.
\end{equation}
In light of the integrability properties of $\bar m_{\delta, k}$ and $\bar w_{\delta, k}$ obtained in Lemma \ref{lemma21-crucial} and Proposition \ref{exi_min_appro_re} , we have the minimization problem \eqref{delta_problem_bar} is equivalent to minimizing $\mathcal E_{\delta}$ on $\mathcal{H}$, where $\mathcal{H}$ is given by

\begin{equation*}
	\begin{split}
		\mathcal{H} := \Bigg\{ 
		&(m,w) \in \big(L^{1+\alpha}(\mathbb{R}^{n}) \cap W^{1,\hat{q}}(\mathbb{R}^{n})\big) \times L^{\hat q}(\mathbb{R}^{n}): (m,w) \mbox{ satisfies }\eqref{min_delta_pro_4} \mbox{ and }\eqref{test_2},  m \geq, \not\equiv 0 \text{ a.e. in } \mathbb{R}^{n}, \\
		& \int_{\mathbb R^n}\left[\left(\tilde{m}_{\delta,k}\ast\eta_{k}\right)^{\alpha}\ast\eta_{k}\right]m\,dx=1, \text{ where } \tilde{m}_{\delta,k} \text{ is the minimizer of } \tilde{e}_{\delta, k} \text{ obtained in Proposition } \ref{exi_min_original} \Bigg\}\\
	= \Bigg\{&(m,w) \in \left(L^{1+\alpha}(\mathbb{R}^{n}) \cap W^{1,\hat{q}}(\mathbb{R}^{n})\right) \times L^{\hat q}(\mathbb{R}^{n}): (m,w) \text{ satisfies \eqref{min_delta_pro_4}  and \eqref{test_2}, } m \geq, \not\equiv 0 \text{ a.e. in } \mathbb{R}^{n}, \\
		& \int_{\mathbb R^n}\left[\left(\bar{m}_{\delta,k}\ast\eta_{k}\right)^{\alpha}\ast\eta_{k}\right]m\,dx=1, \text{ where } \bar{m}_{\delta,k} \text{ is the minimizer of } \bar{e}_{\delta, k} \text{ obtained in Proposition } \ref{exi_min_appro_re} \Bigg\},
	\end{split}
\end{equation*}
where we have used 
	$\tilde{m}_{\delta,k}\ast \eta_{k}(x)=\bar{m}_{\delta,k}\ast\eta_{k}(x)\,\,\text{a.e. in}\,\,\mathbb{R}^{n}$, which is shown in Proposition \ref{exi_min_appro_re}.
 
Now, define
\begin{equation}\label{sup_inf_exch_0}
		\bar\mu_{\delta,k}:=\sup\left\{\mu:-\Delta\psi+H(\nabla \psi)\leq -\mu \left(\bar{m}_{\delta,k}\ast\eta_{k}\right)^{\alpha}\ast\eta_{k}+\delta V(x)+1\text{ on }\mathbb{R}^n\,\text{for some}\,\psi\in\mathcal{B}\right\},
\end{equation}
then we intend to prove that 
\begin{equation}\label{sup_inf_exch}
	\bar\mu_{\delta,k}=\min_{(m,w)\in\mathcal{H}}\mathcal{E}_{\delta}(m,w)=\mathcal E_{\delta}(\bar m_{\delta,k},\bar w_{\delta,k})=\bar e_{\delta, k}.
\end{equation}
Indeed, let
 \begin{equation*}
 	\mathcal{L}(m,w,\mu,\psi):=\mathcal{E}_{\delta}(m,w)+\int_{\mathbb{R}^n}(m\Delta \psi+w\cdot\nabla \psi)\,dx-\mu\int_{\mathbb{R}^n}\left[\left(\bar{m}_{\delta,k}\ast\eta_{k}\right)^{\alpha}\ast\eta_{k}\right]m\,dx+\mu.
 \end{equation*}
Then, one has
\begin{equation}\label{sup_inf_exch_2}
\min_{(m,w)\in\mathcal{H}}\mathcal{E}_{\delta}(m,w)=	\inf_{(m,w)}\sup_{(\mu,\psi)\in \mathbb{R}\times\mathcal{B}}\mathcal{L}(m,w,\mu,\psi),
\end{equation}
where the infimum in the right-hand side is taken over all pairs 
\begin{equation*}
(m,w) \in \left(L^{1+\alpha}(\mathbb{R}^{n}) \cap W^{1,\hat{q}}(\mathbb{R}^{n})\right) \times L^{\hat q}(\mathbb{R}^{n})
\end{equation*}
 satisfying 
 \begin{equation}\label{eq-integrable}
 \tilde{m}(1+V)\in L^{1}(\mathbb{R}^{n}),\  {w}(1+V)^{\frac{1}{\gamma}}\in L^{1}(\mathbb{R}^{n}).\end{equation}
  Thus, the infimum in right-hand side of \eqref{sup_inf_exch_2} is a minimum.  One can also verify that $\mathcal{L}(\cdot,\cdot,\mu,\psi)$ is convex and weakly lower semicontinuous, and $\mathcal{L}(m,w,\cdot,\cdot)$ is linear. Therefore, by applying min-max theorem \cite[Theorem 2.3.7]{Borwein2010Vanderwerff}, we get
\begin{equation}\label{sup_inf_exch_4}
	\begin{aligned}
		&\min_{(m,w)}\sup_{(\mu,\psi)\in \mathbb{R}\times\mathcal{B}}\mathcal{L}(m,w,\mu,\psi)
		=\sup_{(\mu,\psi)\in \mathbb{R}\times\mathcal{B}}\min_{(m,w)}\mathcal{L}(m,w,\mu,\psi)\\
		=&\sup_{(\mu,\psi)\in \mathbb{R}\times\mathcal{B}}\min_{(m,w)}\int_{\mathbb R^n}\left[m L\left (-\frac{w}{m}\right ) + (\delta V(x) +1)m+m\Delta \psi+w\cdot\nabla \psi-\mu \left[\left(\bar{m}_{\delta,k}\ast\eta_{k}\right)^{\alpha}\ast\eta_{k}\right] m\right]\,dx+\mu\\
		=&\sup_{(\mu,\psi)\in \mathbb{R}\times\mathcal{B}}\int_{\mathbb R^n}\min_{(m,w)}\left[mL\left (-\frac{w}{m}\right ) + (\delta V(x) +1)m+m\Delta \psi+w\cdot\nabla \psi-\mu \left[\left(\bar{m}_{\delta,k}\ast\eta_{k}\right)^{\alpha}\ast\eta_{k}\right]m\right]\,dx+\mu\\
		=&\sup_{(\mu,\psi)\in \mathbb{R}\times\mathcal{B}}\int_{\mathbb R^n}\min_{(m,w)}\left[-H(\nabla \psi) + (\delta V(x)+1)+\Delta \psi-\mu \left[\left(\bar{m}_{\delta,k}\ast\eta_{k}\right)^{\alpha}\ast\eta_{k}\right]\right]m\,dx+\mu.
	\end{aligned}
\end{equation}
Note that, with
\[
\Phi := -H(\nabla \psi) + (\delta V(x) + 1) + \Delta \psi - \mu \left[\left(\bar{m}_{\delta,k}\ast\eta_{k}\right)^{\alpha}\ast\eta_{k}\right],
\]
we have
\begin{equation}\label{sup_inf_exch_5}
\min_{(m,w)} \Phi \cdot m = 
\begin{cases}
0, & \text{if } \Phi \geq 0, \\[4pt]
-\infty, & \text{if } \Phi < 0.
\end{cases}
\end{equation}
Then by using \eqref{sup_inf_exch_2}, \eqref{sup_inf_exch_4} and \eqref{sup_inf_exch_5}, one has \eqref{sup_inf_exch} holds. Moreover, from Proposition \ref{exi_min_original} and Proposition \ref{exi_min_appro_re},  we see that $\bar \mu_{\delta,k}=\bar e_{\delta, k}=\tilde e_{\delta, k} $. 

Next, for $\bar\mu_{\delta,k}\in \mathbb{R}$ defined by \eqref{sup_inf_exch_0}, we set 
\begin{equation}\label{new_constraint_tildeE}
\bar {\mathcal{E}}_{\delta}(m,w):=\mathcal{E}_{\delta}(m,w)-\bar\mu_{\delta,k}\int_{\mathbb R^n}\left[\left(\bar{m}_{\delta,k}\ast\eta_{k}\right)^{\alpha}\ast\eta_{k}\right]m\,dx,
\end{equation}
and show that
\begin{align}\label{reallyuseprop33dec}
\min_{(m,w)\in\bar{\mathcal  H}}\bar{ \mathcal E}_{\delta}(m,w)=\bar{\mathcal E}_{\delta}(\bar m_{\delta,k},\bar w_{\delta,k})=0,
\end{align}
where $\bar{\mathcal{H}}$ is given by
\begin{equation*}
	\begin{split}
		\bar{\mathcal{H}} := \Bigg\{&(m,w) \in \left(L^{1+\alpha}(\mathbb{R}^{n}) \cap W^{1,\hat{q}}(\mathbb{R}^{n})\right) \times L^{\hat q}(\mathbb{R}^{n}): (m,w) \text{ satisfies \eqref{min_delta_pro_4}  and \eqref{test_2}, } m \geq, \not\equiv  0 \text{ a.e. in } \mathbb{R}^{n}, \\
		& \int_{\mathbb R^n}m\,dx=\int_{\mathbb R^n}\bar{m}_{\delta,k}\,dx, \text{ where } \bar{m}_{\delta,k} \text{ is the minimizer of } \bar{e}_{\delta} \text{ obtained in Proposition } \ref{exi_min_appro_re} \Bigg\}.
	\end{split}
\end{equation*}
Indeed, for any $(m,w)\in{\mathcal {\bar H}}$,   let $\rho:=\frac{1}{\int_{\mathbb R^n}\left[(\bar m_{\delta,k}\ast\eta_k)^{\alpha}\ast \eta_k\right] m \,dx}$ and $(\hat  m,\hat w)=\rho(m,w)\in\mathcal H$,  then it follows from \eqref{sup_inf_exch} that 
\begin{align*}
\mathcal E_{\delta}(\hat m,\hat w)=\rho\mathcal E_{\delta}( m,w) \geq \mathcal E_{\delta}(\bar m_{\delta,k},\bar w_{\delta,k})=\bar\mu_{\delta,k}.
\end{align*}
Hence,
\begin{align*}
\mathcal E_{\delta}(m,w)\geq \bar \mu_{\delta,k}\int_{\mathbb R^n}\left[(\bar m_{\delta,k}\ast \eta_k)^{\alpha}\ast\eta_k\right]m\,dx,
\end{align*}
so $\bar {\mathcal E}_{\delta}(m,w)\geq 0$  for any $(m,w)\in \bar {\mathcal H}$.
On the other hand,  using $\mathcal E_{\delta}(\bar m_{\delta,k},\bar w_{\delta,k})=\bar \mu_{\delta,k}$ and (\ref{new_constraint_tildeE}), we find that $\inf\limits_{(m,w)\in\bar {\mathcal H}}\bar {\mathcal E}_{\delta}(m,w)=0$ is achieved by $(\bar m_{\delta,k},\bar w_{\delta,k}).$  Consequently, we obtain \eqref{reallyuseprop33dec}.

In what follows, we claim that
\begin{equation}\label{sup_inf_exch_bar}
	\begin{aligned}
&\bar \lambda_{\delta,k}\int_{\mathbb{R}^n}\bar{m}_{\delta,k}\,dx\\
&:=	\sup\left\{\lambda\int_{\mathbb{R}^n}\bar{m}_{\delta,k}\,dx :-\Delta\psi+H(\nabla \psi)+\lambda\leq -\bar \mu_{\delta,k} \left(\bar{m}_{\delta,k}\ast\eta_{k}\right)^{\alpha}\ast\eta_{k}+\delta V(x)+1\,\text{on}\,\mathbb{R}^n\,\text{for some}\,\psi\in\mathcal{B}\right\}\\
        &=\min_{(m,w)\in \bar{\mathcal{H}}}\bar {\mathcal{E}}_{\delta}(m,w).
    \end{aligned}
\end{equation}
To this end, we first define the following functional 
\begin{equation*}
 	\bar{\mathcal{L}}(m,w,\lambda,\psi;\mu):=\bar{\mathcal{E}}_{\delta}(m,w)+\int_{\mathbb{R}^n}(m\Delta \psi+w\cdot\nabla \psi)\,dx-\lambda\int_{\mathbb{R}^n} m \,dx+ \lambda\int_{\mathbb{R}^n} \bar{m}_{\delta,k}\,dx.
 \end{equation*}
Similarly to the above, we have
\begin{equation}\label{sup_inf_exch_bar_1}
		\min_{(m,w)\in \bar{\mathcal{H}}}\bar{\mathcal{E}}_{\delta}(m,w)=\min_{(m,w)}\sup_{(\lambda, \psi)\in \mathbb{R}\times \mathcal{B}}\bar{\mathcal{L}}(m,w,\lambda,\psi;\mu),
\end{equation}
where the minimum in the right-hand side is taken over all pairs
\begin{equation*}
(m,w) \in \left(L^{1+\alpha}(\mathbb{R}^{n}) \cap W^{1,\hat{q}}(\mathbb{R}^{n})\right) \times L^{\hat q}(\mathbb{R}^{n})
\end{equation*}
that satisfy \eqref{eq-integrable}.  Moreover, we proceed with the same argument as deriving (\ref{sup_inf_exch_4}) and obtain from (\ref{sup_inf_exch_bar_1}) that (\ref{sup_inf_exch_bar}) holds. 
Due to the boundedness of $\bar\mu_{\delta ,k}$ and smoothness of $\left(\bar{m}_{\delta,k}\ast\eta_{k}\right)^{\alpha}\ast\eta_{k}$, we obtain from $(i)$ and $(ii)$ of Lemma \ref{lemma22preliminary} with $f_{k}=-\bar \mu_{\delta,k} \left(\bar{m}_{\delta,k}\ast\eta_{k}\right)^{\alpha}\ast\eta_{k}$ and $V_{k}(x)=\delta V(x)+1$ that for $\bar \lambda_{\delta,k}$ given in (\ref{sup_inf_exch_bar}), there exists $\bar{u}_{\delta,k}\in C^{2}(\mathbb{R}^n)$ bounded from below (depending on $\delta$ and $k$) such that
\begin{equation}\label{sup_inf_exch_7}
	-\Delta\bar{u}_{\delta,k}+H(\nabla \bar{u}_{\delta,k})+ \bar \lambda_{\delta, k}=-\bar\mu_{\delta,k} \left(\bar{m}_{\delta,k}\ast\eta_{k}\right)^{\alpha}\ast\eta_{k}+\delta V(x)+1\,\,\text{on}\,\,\mathbb{R}^n.
\end{equation}
Moreover, invoking Lemma \ref{sect2-lemma21-gradientu} and Lemma \ref{lowerboundVkgenerallemma22}, one has the solution $\bar u_{\delta,k}$ to (\ref{sup_inf_exch_7}) satisfies (\ref{estimate_gra}). In addition, 
\begin{equation*}
	|\Delta\bar{u}_{\delta,k}(x)|\leq C_{H}|\nabla \bar{u}_{\delta,k}(x)|^{\gamma}+|\bar{\mu}_{\delta,k} \left(\bar{m}_{\delta,k}\ast\eta_{k}\right)^{\alpha}\ast\eta_{k}|+|\bar \lambda_{\delta, k}|+\delta V(x)+1\leq C_{k,\delta}(1+ V(x))\,\, \text{on}\,\, \mathbb{R}^n.
\end{equation*}
Therefore, we have $\bar{u}_{\delta,k}\in \mathcal{B}$. 
Moreover, combining \eqref{reallyuseprop33dec} with (\ref{sup_inf_exch_bar}), one has $\bar{\lambda}_{\delta,k}=0$. 
Now, with the aid of \eqref{sup_inf_exch_7}, we can use \eqref{test_2} with the test function $\psi=\bar{u}_{\delta,k}$ to obtain that
\begin{equation*}
	\begin{aligned}
	0&=\int_{\mathbb R^n}\bar{m}_{\delta,k}L\left (-\frac{\bar{w}_{\delta,k}}{\bar{m}_{\delta,k}}\right)\,dx +\int_{\mathbb R^n}\left((\delta V +1)\bar{m}_{\delta,k}\right)\,dx- \bar{\mu}_{\delta,k}\int_{\mathbb R^n}\left[\left(\bar{m}_{\delta,k}\ast\eta_{k}\right)^{\alpha}\ast\eta_{k}\right]\bar{m}_{\delta,k}\,dx\\
	&=\int_{\mathbb R^n}\bar{m}_{\delta,k}\left( L\left (-\frac{\bar{w}_{\delta,k}}{\bar{m}_{\delta,k}}\right)-\Delta\bar{u}_{\delta,k} + H(\nabla \bar{u}_{\delta,k}) \right)\,dx\\
	&=\int_{\mathbb R^n}\bar{m}_{\delta,k}\left(L\left (-\frac{\bar{w}_{\delta,k}}{\bar{m}_{\delta,k}}\right)-\nabla\bar{u}_{\delta,k}\cdot\frac{\bar{w}_{\delta,k}}{\bar{m}_{\delta,k}} + H(\nabla \bar{u}_{\delta,k}) \right)\,dx.	
	\end{aligned}
\end{equation*}
Using the properties of the Legendre transform and \eqref{general-Lagrangian}, we have
\begin{equation*}
\frac{\bar{w}_{\delta,k}(x)}{\bar{m}_{\delta,k}(x)}=-\nabla H(\nabla \bar{u}_{\delta,k}(x)),\,\,\text{for}\,\, \bar{m}_{\delta,k}(x)\neq0,\,\,x\in \mathbb{R}^n.
\end{equation*}
In addition, invoking the definition \eqref{general-Lagrangian}, we obtain that $\bar{w}_{\delta,k}(x)=0$ where $\bar{m}_{\delta,k}(x)=0$, $x\in\mathbb{R}^n$. As a consequence, we obtain the relationship between $\bar{m}_{\delta,k}$ and $\bar u_{\delta,k}$, which is
\begin{equation*}
	\Delta \bar{m}_{\delta,k}+\nabla \cdot \left(\bar{m}_{\delta,k}  \nabla H (\nabla \bar{u}_{\delta,k})\right)=0
\end{equation*}
in the weak sense.   This together with (\ref{sup_inf_exch_7}) and $\bar\lambda_{\delta,k}=0$ indicates that $(\bar{u}_{\delta,k},\bar{m}_{\delta,k},\bar{\mu}_{\delta,k})$ is a solution to \eqref{MFG_mass_supercritical_linear}.
\end{proof}
Concerning the regularity of 
$m$ constructed in Proposition \ref{classcal_sol_mfg}, we have  the following proposition
\begin{proposition}\label{classcal_sol_mfg_1}
Let $(\bar{u}_{\delta,k},\bar{m}_{\delta,k},\bar{\mu}_{\delta,k})$ be the solution of \eqref{MFG_mass_supercritical_linear}, then $\bar{m}_{\delta,k}\in W^{1,p}(\mathbb{R}^n)$ for all $p>1$.  
\end{proposition}
\begin{proof}
We skip the proof and refer the readers to the arguments shown in Proposition 3.6 of \cite{cesaroni2018concentration} or Proposition 5.2 of \cite{cirant2025critical}.
\end{proof}

\begin{remark}\label{rem2}
	In the case of $\gamma'>n$, since   $W^{1,\hat{q}}(\mathbb{R}^n)\hookrightarrow C^{0,\theta}(\mathbb{R}^n)$ for some $\theta\in(0,1)$, we can bypass the regularized problem \eqref{delta_problem_tilde}   and consider the minimization problem \eqref{delta_problem} directly. 
\end{remark}

With the aid of Lemma \ref{blowupanalysismlinfboundcritical}, we shall establish the uniform boundedness of $(\bar{u}_{\delta,k},\bar{m}_{\delta,k},\bar{\mu}_{\delta,k})$ obtained in Proposition \ref{classcal_sol_mfg} in some space and take the limit to obtain a solution of \eqref{MFG_mass_supercritical}, which proves Lemma \ref{potential_MFG}.  
 
\vspace{0.2cm}
\textbf{Proof of Lemma \ref{potential_MFG}:} We present the proof only for the case $\gamma'\leq n$, since the case $\gamma'> n$ can be treated in a similar manner.  First of all, it follows from Proposition \ref{exi_min_appro_re}, Proposition \ref{classcal_sol_mfg} and Proposition \ref{classcal_sol_mfg_1} that, for each  $k>0$ and $\delta\in(0,1)$, (\ref{MFG_mass_supercritical_linear}) has a solution $(\bar u_{\delta,k}, \bar m_{\delta,k}, \bar \mu_{\delta,k})\in C^{2}(\mathbb{R}^n)\times W^{1,p}(\mathbb{R}^n)\times \mathbb{R}$, $\forall p>1$,  satisfying (\ref{estimate_gra}),
 and $$(\bar m_{\delta,k}, \bar w_{\delta,k})\in \mathcal{A}\cap \left\{m: \left\Vert \left[\left(\tilde{m}_{\delta,k}\ast\eta_{k}\right)^{\alpha}\ast\eta_{k}\right]m\right\Vert_{L^{1}(\mathbb R^n)}=1 \right\}\cap \left\{m: \Vert m\ast \eta_{k}\Vert_{L^{1+\alpha}(\mathbb R^n)}=1 \right\}$$ 
 is a minimizer of $\bar{e}_{\delta, k}=\tilde e_{\delta,k}$.  
In addition, similar to the proof of \eqref{min_est_2}, we can obtain that 
\begin{equation}\label{MFG_mass_sup_class1}
	\int_{\mathbb{R}^n}\bar m_{\delta,k}L\left(-\frac{\bar w_{\delta,k}}{\bar m_{\delta,k}}\right)\,dx\leq C \,\,\text{and}\,\, \int_{\mathbb R^n}(V(x) +1)\bar m_{\delta,k}\,dx\leq C_\delta, 
\end{equation}
where the constants $C_{\delta} > 0$ and $C > 0$ do not depend on $k$. 


In what follows, we aim to find a minimizer $(m_{\delta}, w_{\delta})$ for minimization problem \eqref{delta_problem} by taking the limit as $k \to +\infty$.  Noting that $(\bar u_{\delta,k}, \bar m_{\delta,k}, \bar \mu_{\delta,k})$ satisfies \eqref{assumptionsincrucialemmanewest}.  Then we apply Young's inequality to get
\begin{align*}
	\sup_{k}\Vert \left(\bar m_{\delta,k}\ast\eta_k\right)^{\alpha}\ast\eta_{k}\Vert_{L^{1+\frac{1}{\alpha}}(\mathbb R^n)}\leq \sup_{k}\Vert  \bar m_{\delta,k}^{\alpha}\Vert_{L^{1+\frac{1}{\alpha}}(\mathbb R^n)}<\infty,
\end{align*}
and
\begin{align*}
	\sup_{k}\Vert \left(\bar m_{\delta,k}\ast\eta_k\right)^{\alpha}\ast\eta_{k}\Vert_{L^{1+\frac{1}{\alpha}}(B_R(x_0))}\leq \sup_{k}\Vert  \bar m_{\delta,k}^{\alpha}\Vert_{L^{1+\frac{1}{\alpha}}(B_{2R}(x_0))}<\infty,~\text{for~any~}x_0\in\mathbb R^n\text{ and }R\text{ large.}
\end{align*}
Then,  we utilize Lemma \ref{blowupanalysismlinfboundcritical} {with $y_k\equiv 0$} to obtain that
\begin{equation}\label{eq4.35}
	\limsup_{k\to +\infty}\|\bar m_{\delta,k}\|_{L^\infty(\mathbb R^n)}<\infty.
\end{equation}
Next, by using Lemma \ref{sect2-lemma21-gradientu}, we can further derive
\begin{equation}\label{eq4.36}
|\nabla \bar u_{\delta,k}(x)|\leq C_{\delta}(1+V(x))^{\frac{1}{\gamma}}, \text{ for all $x\in \mathbb{R}^n$, where $C_{\delta}>0$ is independent of $k$.}
\end{equation}
Noting that $\bar u_{\delta,k}$ is bounded from below,  without loss of generality, we assume that $\bar u_{\delta,k}(0)=0$.
Thanks to \eqref{29uklemma22}, we obtain that for some $C_{\delta, k}>0$ depending on $k$ and $\delta$, $\bar u_{\delta,k}(x) \geq C_{\delta,k} V(x)^{\frac{1}{\gamma}} - C_{\delta,k} \to +\infty$ as $|x|\to+\infty$. Consequently, for each fixed $\delta$ and $k$, the function $\bar u_{\delta,k} \in C^2(\mathbb R^n)$ attains its minimum at some finite point $x_{\delta,k}$. By using  \eqref{eq4.35}, the boundedness of $\bar \mu_{\delta, k}$ and the coercivity of $V$, we can deduce from the $u$-equation of \eqref{MFG_mass_supercritical_linear} that $x_{\delta,k}$ is uniformly bounded with respect to $k$. Then, from $\bar u_{\delta, k}(0)=0$ and \eqref{eq4.36} , we get that
there exists $C_{\delta}>0$ depending only on $\delta\in(0,1)$ such that
\begin{equation*}
-C_{\delta}\leq \bar u_{\delta,k}(x)\leq C_{\delta}|x|(1+V(x)^{\frac{1}{\gamma}}) \text{ for all } x\in\mathbb{R}^n,
\end{equation*}
which indicates $\bar u_{\delta,k}$ is bounded from below uniformly in $k$.  By employing \eqref{29uklemma22}, we conclude that
\begin{equation}\label{eq4.37}
C_{1,\delta}V(x)^{\frac{1}{\gamma}}-C_{1,\delta}\leq \bar u_{\delta,k}(x)\leq C_{2,\delta}|x|(1+V(x)^{\frac{1}{\gamma}}),\text{ for all } x\in\mathbb {R}^n,
\end{equation}
where $C_{1,\delta},C_{2,\delta}>0$ are independent of $k$.

{It follows from \eqref{eq4.35}, \eqref{eq4.36} and (H1)-(H2)} that, for any $R>1$ and $p>1$,
\begin{equation}\label{eq4.306}
\|\bar w_{\delta,k}\|_{L^p(B_{2R}(0))}=\|\bar m_{\delta,k}\nabla H(\nabla \bar u_{\delta,k})\|_{L^p(B_{2R}(0))}\leq C_{p,R}<\infty,
\end{equation}
where the constant  $C_{p, R}>0$ depending only on $p$, $R$ and independent of $k$.  By applying Lemma \ref{proposition-lemma21-FP}, we obtain from \eqref{eq4.306} that $\|\bar m_{\delta,k}\|_{W^{1,p}(B_{2R}(0))}\leq C_{p,R}<\infty$. Choosing $p>n$ large enough and employing Sobolev embedding theorem,  we get
\begin{equation}\label{eq4.38}
\|\bar m_{\delta,k}\|_{C^{0,\theta_1}(B_{2R}(0))}\leq C_{\theta_1,R}<\infty \text{ for some $\theta_1\in(0,1)$.}
\end{equation}

Now, we estimate $\bar u_{\delta,k}$. For this purpose, we rewrite $u$-equation of \eqref{MFG_mass_supercritical_linear} in the form of
\begin{equation}\label{eq4.39}
-\Delta \bar u_{\delta,k}=-H(\nabla \bar u_{\delta,k})+\bar f_{\delta,k}\ \text{ with }\bar f_{\delta,k}(x):=-\bar \mu_{\delta,k}\left( \bar m_{\delta,k}\ast\eta_{k}\right)^{\alpha}\ast \eta_{k}+\delta V(x)+1.
\end{equation}
    By using \eqref{MFG-H}, \eqref{eq4.35}, \eqref{eq4.36} and the boundedness of $\bar \mu_{\delta, k}$, we obtain that for any $p>1$,
$$\|\bar f_{\delta,k}\|_{L^{p}(B_{2R}(0))}+\|H(\nabla \bar u_{\delta,k})\|_{L^{p}(B_{2R}(0))}\leq C_{p,R}<\infty.$$
Then by the classical $W^{2,p}$ estimates, we deduce from \eqref{eq4.37} and \eqref{eq4.39} that
\begin{equation*}
\|\bar u_{\delta,k}\|_{W^{2,p}(B_{R+1})}\leq C_{p,R}\left(\|\bar u_{\delta,k}\|_{L^{p}(B_{2R}(0))}+\|\bar f_{\delta,k}\|_{L^{p}(B_{2R}(0))}+\| H(\nabla \bar u_{\delta,k})\|_{L^{p}(B_{2R}(0))}\right)
\leq \bar C_{p,R}<\infty.
\end{equation*}
We obtain by taking $p>n$ large enough and using Sobolev embedding theorem that
\begin{equation}\label{eq4.41}
\|\bar u_{\delta,k}\|_{C^{1,\theta_2}(B_{R+1}(0))}\leq C_{\theta_2,R}<\infty, \text{ for some $\theta_2\in(0,1)$.}
\end{equation}
Combining \eqref{eq4.38}, (\ref{eq4.41}) {with (H1)-(H2)}, one gets
$$\|H(\nabla \bar u_{\delta,k})\|_{C^{0,\theta_3}(B_{R+1}(0))} +\|\bar f_{\delta,k}\|_{C^{0,\theta_3}(B_{R+1}(0))}\leq C_{\theta_3,R}<\infty, \text{ for some $\theta_3\in(0,1)$.}$$
Then, by using Schauder's estimates in \eqref{eq4.39}, we infer that
\begin{equation}\label{eq4.401}
\|\bar u_{\delta,k}\|_{C^{2,\theta_4}(B_{R}(0))} \leq C_{\theta_4,R}<\infty, \text{ for some $\theta_4\in(0,1)$.}
\end{equation}
Now, taking the limit of $k\to\infty$ and then $R\to\infty$, we can apply Arzel\`{a}-Ascoli theorem and diagonal argument to obtain that there exists  $u_{\delta}\in C^2(\mathbb R^n)$ such that
\begin{equation}\label{eq4.42}
\bar u_{\delta,k}\overset{k\to+\infty}\longrightarrow u_{\delta} \text{ in }C^{2,\theta_5}_{\rm loc}(\mathbb {R}^n) \text{ for some $\theta_5\in(0,1)$.}
\end{equation}
On the other hand, it follows from Lemma \ref{lemma21-crucial}, (\ref{MFG-L}) and \eqref{MFG_mass_sup_class1} that, there exists $(m_{\delta},w_{\delta})\in W^{1,\hat q}(\mathbb R^n)\times \big(L^{1}(\mathbb R^n)\cap L^{\hat q}(\mathbb R^n) \big)$ such that, as $k\to\infty$,
\begin{align}\label{eq4.4404}
\bar m_{\delta,k}\to m_{\delta} \text{ a.e. in $\mathbb R^n$ \ and \ \ }(\bar m_{\delta,k}, \bar w_{\delta,k})\rightharpoonup (m_{\delta},w_{\delta}) \text{~weakly in~}  W^{1,\hat{q}}(\mathbb R^n)\times L^{\hat q}(\mathbb R^n).
\end{align}
In addition, invoking Lemma \ref{lem4.1}, we obtain that
\begin{align}\label{eq4.44}
\bar m_{\delta,k}\overset{k\to+\infty}\rightarrow m_{\delta} \text{~strongly in~} L^1(\mathbb R^n)\cap L^{1+\alpha}(\mathbb R^n) .
\end{align}

Passing to the limit as $k\to+\infty$ in \eqref{MFG_mass_supercritical_linear}, we then conclude from \eqref{eq4.42}-\eqref{eq4.44} and the boundedness of $\bar \mu_{\delta, k}$ that there exists $\mu_{\delta}\in \mathbb R$ such that
 $(u_{\delta},m_{\delta},\mu_{\delta})$ solves the auxiliary MFG system \eqref{MFG_mass_supercritical}. 
In particular, by \eqref{eq4.36} and \eqref{eq4.37}, we get
\begin{equation}\label{eq4.46}
|\nabla u_{\delta}(x)|\leq C_{\delta}(1+V(x))^{\frac{1}{\gamma}}\,\, \text{and}\,\, C_{1,\delta}V(x)^{\frac{1}{\gamma}}-C_{1,\delta}\leq u_{\delta}(x)\leq C_{2,\delta}|x|V(x)^{\frac{1}{\gamma}}+C_{2,\delta},\text{ for all } x\in\mathbb {R}^n.
\end{equation}
In view of \eqref{eq4.35} and \eqref{eq4.4404}, we have $m_{\delta}\in L^\infty(\mathbb R^n)$.  Then, using \eqref{MFG_mass_supercritical} and \eqref{eq4.46}, we can follow the approach outlined in \cite[Proposition. 5.1]{cirant2025critical} to obtain
\begin{equation*}
w_{\delta}=- m_{\delta}\nabla H(\nabla u_{\delta})\in L^p(\mathbb R^n)\,\,\text{and}\,\, m_{\delta}\in W^{1,p}(\mathbb R^n)
\,\,\text{for all}\,\, p>1.
\end{equation*}
In addition,  we deduce from  \eqref{MFG_mass_sup_class1} and Fatou's lemma that $\int_{\mathbb{R}^n}V(x) m_{\delta}\leq C_{\delta}$ for some $C_{\delta}>0$ depending on $\delta$.  Therefore, we conclude that $(m_{\delta},w_{\delta})\in\mathcal{A}\cap \left\{m: \Vert m\Vert_{L^{1+\alpha}(\mathbb R^n)}=1 \right\}$.

Finally, we prove that $(m_{\delta},w_{\delta})\in\mathcal{A}\cap \left\{m: \Vert m\Vert_{L^{1+\alpha}(\mathbb R^n)}=1 \right\}$ is a minimizer of $e_{\delta}$ defined by \eqref{delta_problem}.  Indeed,    a standard approximation argument shows  that $\lim_{k\to \infty}\tilde e_{\delta,k}=\lim_{k\to \infty}\bar e_{\delta,k}= e_{\delta}$. Hence, using \eqref{eq4.4404}, we obtain from Fatou's lemma that
 \begin{equation*}
 \begin{aligned}
 e_{\delta}=\lim_{k\to \infty}\bar e_{\delta,k}&=\lim_{k\to\infty}\left\{ \int_{\mathbb{R}^n}\bar m_{\delta,k} L\left(-\frac{\bar w_{\delta,k}}{\bar m_{\delta,k}}\right)\,dx+\int_{\mathbb R^n}\left(\delta V(x) +1\right)\bar m_{\delta,k}\,dx \right\}\\
 &\geq \int_{\mathbb R^n}m_{\delta} L\left(-\frac{w_{\delta}}{m_{\delta}}\right)\, dx+\int_{\mathbb R^n}\left(\delta V(x)+1\right)m_{\delta}\,dx=\mathcal{E}_{\delta}(m_{\delta},w_{\delta})\geq e_{\delta}, 
 \end{aligned}
 \end{equation*}
  which shows $(m_{\delta},w_{\delta})$ minimizes $e_{\delta}$. Moreover, since  $(u_{\delta},m_{\delta},\mu_{\delta})$ satisfies \eqref{MFG_mass_supercritical}, we test the $m$-equation in \eqref{MFG_mass_supercritical} against $u_\delta$ to obtain
 \begin{align*}
	\int_{\mathbb R^n}m_{\delta}\Delta u_{\delta}\,dx=\int_{\mathbb R^n}w_{\delta}\cdot\nabla u_{\delta}\,dx=-\int_{\mathbb R^n}m_{\delta}\nabla H(\nabla u_{\delta})\cdot \nabla u_{\delta}\,dx=-\gamma\int_{\mathbb R^n}m_{\delta}H(\nabla u_{\delta})\,dx.
 \end{align*}
 Similarly, multiplying the $u$-equation in \eqref{MFG_mass_supercritical} by $m_{\delta}$, we integrate it to get
 \begin{equation}\label{eq4.301}
	\begin{aligned}
		\mu_{\delta}&=-(1-\gamma)\int_{\mathbb R^n}m_{\delta}H(\nabla u_{\delta})\,dx+\int_{\mathbb R^n}(\delta V(x)+1)m_{\delta}\,dx\\
 		&=\int_{\mathbb R^n}m_{\delta} L\left(-\frac{w_{\delta}}{m_{\delta}}\right)\, dx+\int_{\mathbb R^n}(\delta V(x)+1)m_{\delta}\,dx=\mathcal{E}_{\delta}(m_{\delta},w_{\delta})= e_{\delta},
	\end{aligned}
 \end{equation}
{where we have used \eqref{finalpohomar18}, which follows from the strict convexity of $H$ and assumption \textup{(H1)}, in the second equality of \eqref{eq4.301}. }  This completes the proof of the lemma.	
$\hfill\qed$
\subsection{Auxiliary potential-free MFG systems }\label{potential_free_MFG_sys}

In this subsection, we analyze the limiting process of the triple $( u_{\delta},m_{\delta},\mu_{\delta})$ corresponding to a minimizer of $e_{\delta}$ obtained in Lemma \ref{potential_MFG} as $\delta\to 0^+$, which in turn proves the existence of solutions to an auxiliary MFG system without potential {in the mass-supercritical case}. To this end, we first recall from \eqref{e_0_min} that 
\begin{equation*}
	e_{0}=\inf_{(m,w)\in \mathcal{A}\cap \left\{m: \Vert m\Vert_{L^{1+\alpha}(\mathbb R^n)}=1 \right\}} \mathcal{E}_{0}(m,w),
\end{equation*}
where 
\begin{equation*}
\mathcal{E}_{0}(m,w):=\int_{\R^n} mL\left(-\frac{w}{m}\right) \, dx+\int_{\R^n}m\,dx. 
\end{equation*}
The main results can be  stated as follows.

\begin{lemma}\label{Existience_of_MFG_free_potential}
Let $(u_{\delta}, m_{\delta},\mu_{\delta})$ be as in Lemma \ref{potential_MFG}. Then, up to subsequences, $(u_{\delta}, m_{\delta},\mu_{\delta})$ converges in $C_{loc}^{2}(\mathbb{\R}^N)\times L^{p}(\mathbb{R}^n)\times \mathbb{\R}$, for all $p\in[1+\alpha, \hat q^{*})$, to a solution $(u_{0},m_{0}, \mu_{0})$ of 
\begin{align}\label{mass_supercritical_problem_free_potential}
	\left\{\begin{array}{ll}
		-\Delta u+H(\nabla u)+\mu m^{\alpha}=1,&x\in\mathbb R^n,\\
		\Delta m+\nabla\cdot (m\nabla H(\nabla u))=0,&x\in\mathbb R^n,\\
		w=-m\nabla H(\nabla u),\ \int_{\mathbb R^n}m^{1+\alpha}\,dx=1,
	\end{array}
	\right.
\end{align}
Moreover, $0<m_{0}\leq C\kappa_{1}e^{-\kappa_{2}|x|^{2}}$ for some $\kappa_1,~\kappa_2>0$, and $\mu_{0}=e_0>0$. Furthermore, set $w_{0}=-m_0\nabla H(\nabla u_0)$, then $(m_{0},w_{0})$ is a minimizer of \eqref{e_0_min}.
\end{lemma}

To show Lemma \ref{Existience_of_MFG_free_potential}, we first analyze the limiting process of  $e_{\delta}$ as $\delta\to 0^{+}$ and establish the following relationship between the energy $e_{0}$ and $e_{\delta}$.
\begin{lemma}\label{re_energy}
Let $(u_{\delta}, m_{\delta},\mu_{\delta})$ be as in Lemma \ref{potential_MFG}. Then, it holds that 
$$\lim\limits_{\delta\to 0^+}e_{\delta}=e_{0},$$
where $e_{\delta}$ and $e_0$ are defined in (\ref{delta_problem}) and \eqref{e_0_min}, respectively. In addition,
\begin{equation}\label{potential_free101}
	\lim_{\delta\to 0^+}\delta \int_{\mathbb{R}^{n}} V(x)m_{\delta}\,dx=0.
\end{equation}
\end{lemma}
\begin{proof}
 From the definition of functionals $\mathcal E_{\delta}(m,w)$ and $\mathcal E_0(m,w)$, we see that 
$e_{\delta}\geq e_{0}$ for any $\delta\in(0,1)$.
Moreover, by the definition of $e_{0}$, for any $\epsilon>0$, there exists $(m_{\epsilon},w_{\epsilon})\in \mathcal{A}\cap \left\{m: \Vert m\Vert_{L^{1+\alpha}(\mathbb R^n)}=1 \right\}$ such that
\begin{equation*}
	e_{0}\leq \mathcal{E}_{0}(m_{\epsilon},w_{\epsilon})\leq e_{0}+\epsilon.
\end{equation*}
Now, setting $M_{\epsilon}:=\int_{\mathbb R^n}V(x)m_{\epsilon}\,dx<\infty$,  we deduce  that
\begin{equation}\label{potential_free4}
	e_{\delta}\leq \mathcal{E}_{\delta}(m_{\epsilon},w_{\epsilon})=\mathcal E_0(m_{\epsilon},w_{\epsilon})+\delta \int_{\mathbb R^n} V(x)m_{\epsilon}\,dx \leq e_{0}+\epsilon+\delta M_{\epsilon}.
\end{equation}
Letting $\delta\to 0^+$ then $\epsilon\to 0^+$ in \eqref{potential_free4}, we obtain that
$\lim\limits_{\delta\to 0^+}e_{\delta}\leq e_{0}.$
In summary, we conclude that $\lim\limits_{\delta\to 0^+}e_{\delta}= e_{0}$.  Moreover, since $(m_{\delta}, w_{\delta})$ is the minimizer of $e_{\delta}$, then 
$$0\leq \delta \int_{\mathbb{R}^{n}} V(x)m_{\delta}\,dx=e_{\delta}-\mathcal{E}_{0}(m_{\delta}, w_{\delta})\leq e_{\delta}-e_{0}\to 0\,\text{as}\,\, \delta\to0^+,$$
which implies that \eqref{potential_free101} holds. This ends the proof of the lemma.
\end{proof}
\begin{lemma}\label{exi_potential_free_MFG}
Let $(u_{\delta}, m_{\delta},\mu_{\delta})$ be the solution of \eqref{MFG_mass_supercritical} obtained in Lemma \ref{potential_MFG}. 
Then there exists $\{y_{\delta}\}\in \mathbb{R}^n$ such that, up to a subsequence,
\begin{equation*}
      \left(m_{\delta}(x+y_{\delta}), w_{\delta}(x+y_{\delta})\right)\overset{\delta \to 0^+}\rightharpoonup(m_{0}(x),w_{0}(x))\,\, \text{weakly in}\,\,W^{1,\hat q}( \mathbb{R}^n)\times L^{\hat q}( \mathbb{R}^n),
\end{equation*}
\begin{equation*}
	m_{\delta}(x+y_{\delta})\overset{\delta \to 0^+}\to m_{0}(x) \,\,\text{strongly in }\,\, L^{p}(\mathbb{R}^n), \forall p\in[1+\alpha, \hat q^{*}),
\end{equation*}
and
\begin{equation*}
	u_{\delta}(x+y_{\delta})\overset{\delta \to 0^+}\to u_{0}(x) \,\,\text{strongly in }\,\, C_{loc}^{2,\theta}(\mathbb{R}^n)\,\,\text{for some}\,\, \theta\in(0,1),
\end{equation*}
 where $(m_{0},w_{0})\in\mathcal{A}\cap \left\{m: \Vert m\Vert_{L^{1+\alpha}(\mathbb R^n)}=1 \right\}$, with $w_{0}=- m_{0}\nabla H(\nabla u_0)$ is a minimizer of the energy $e_{0}$. Moreover, $u_{0}\in C^{2}(\mathbb{R}^n)$ is bounded from below, and $(u_{0},m_{0},\mu_{0})$ solves the auxiliary potential-free MFG system \eqref{mass_supercritical_problem_free_potential} for  $\mu_{0}=e_{0}$.	
\end{lemma}
\begin{proof}
Recalling that $(m_{\delta},w_{\delta}) \in \mathcal{A}\cap \left\{m: \Vert m\Vert_{L^{1+\alpha}(\mathbb R^n)}=1\right\}$, where $w_{\delta}=- m_{\delta}\nabla H(\nabla u_{\delta})$  is a minimizer of $e_{\delta}$, and the triple $(u_{\delta}, m_{\delta}, \mu_{\delta})$ solves the system \eqref{MFG_mass_supercritical}. Then, by Lemma \ref{potential_MFG} and Lemma \ref{re_energy}, we get
\begin{equation}\label{potential_free7_1}
	\begin{aligned}
		\mu_{0}:=\lim_{\delta\to0^{+}}\mu_{\delta}=\lim_{\delta\to0^{+}}e_{\delta}= e_{0}.
	\end{aligned}
\end{equation}
Hence, there exists $C>0$ independent of $\delta\in(0,1)$ such that
\begin{equation*}
|\mu_{\delta}|=|e_{\delta}|\leq C,\,\text{and thus} \,\, \int_{\mathbb R^n}m_{\delta}L\left(-\frac{w_{\delta}}{m_{\delta}}\right)\, dx+\int_{\mathbb R^n}(\delta V(x)+1)m_{\delta}\,dx\leq C.
\end{equation*}
As a consequence, we apply Lemma \ref{lemma21-crucial} and (\ref{MFG-L}) to obtain that 
\begin{equation}\label{potential_free7_3}
\Vert m_{\delta}\Vert_{W^{1,\hat q}(\mathbb R^n)},\, \Vert w_{\delta}\Vert_{L^{p}(\mathbb R^n)}\leq C\, \text{for any}\,\, p\in [1,\hat q],
\end{equation}
where $C>0$ is some constant independent of $\delta\in(0,1)$.

In order to complete the proof of this lemma, it is necessary to prove the following claims: 

\noindent\textbf{Claim A:} There exist $\zeta>0$, $R>0$ and $y_{\delta}\in\mathbb{R}^n$ such that
\begin{equation}\label{potential_free8}
\liminf_{\delta\to0^{+}}\int_{B_{R}(y_{\delta})} m_{\delta}(x)\,dx>\zeta>0.
\end{equation}
Indeed, if \eqref{potential_free8} is not true, then by the Lions' lemma    \cite[Lemma \uppercase\expandafter{\romannumeral 1}.1]{Lions1984Ann}, we obtain that
\begin{equation*}
\lim_{\delta\to 0^{+}}	\int_{\mathbb{R}^{n}}m_{\delta}^{1+\alpha}\,dx=0,
\end{equation*}
which contradicts the fact that $\int_{\mathbb{R}^{n}}m_{\delta}^{1+\alpha}\,dx=1$ for any $\delta\in(0,1)$. Therefore, we conclude that \textbf{Claim A} holds.  

\noindent\textbf{Claim B:} We claim that 
\begin{equation*}
	\lim_{\delta\to 0^{+}}\delta V(y_{\delta})=0.
\end{equation*}
In fact, if $\lim\limits_{\delta\to 0^{+}}|y_{\delta}|<+\infty$, then obviously \textbf{Claim B} holds. If $\lim\limits_{\delta\to 0^{+}}|y_{\delta}|=+\infty$, we then deduce from \eqref{potential_free101},  \eqref{potential_free8} and (\ref{V2mainassumption_2}) that there exists constant $C>0$ such that
\begin{equation*}
\zeta V\big({y_{\delta}}\big)\delta\leq V\big({y_{\delta}}\big)\delta\int_{B_{R}(y_{\delta})}m_{\delta}\,dx\leq \delta C\int_{B_{R}(y_{\delta})}V(x)m_{\delta}\,dx\leq \delta C\int_{\mathbb{R}^n} V(x)m_{\delta}\,dx\to 0\ \text{ as }\ \delta\to 0,
\end{equation*}
which indicates that  the \textbf{Claim B} holds. 

Now, we define  
\begin{equation}\label{potential_free14}
\bar{m}_{\delta}(x):=m_{\delta}(x+y_{\delta})\,,~  \bar{u}_{\delta}(x):=u_{\delta}(x+y_{\delta})\,\, \text{and}\,\, \bar w_{\delta}:=- \bar m_{\delta}H(\nabla\bar u_{\delta}),
\end{equation}
then it follows from \eqref{MFG_mass_supercritical} that $(\bar{m}_{\delta}(x),\bar{u}_{\delta}(x))$ satisfies
\begin{align}\label{MFG_mass_supercritical6_translated}
	\left\{\begin{array}{ll}
		-\Delta \bar u_{\delta}+H(\nabla \bar u_{\delta})+\mu_{\delta} \bar m_{\delta}^{\alpha}=\delta V(x+y_{\delta})+1,&x\in\mathbb R^n,\\
		\Delta \bar m_{\delta}+\nabla\cdot (\bar m_{\delta}\nabla H(\nabla\bar u_{\delta}))=0,&x\in\mathbb R^n,\\
		\int_{\mathbb R^n}\bar m_{\delta}^{1+\alpha}\,dx=1.
	\end{array}
	\right.
\end{align}

Similarly as shown in the proof of Lemma \ref{potential_MFG}, we now estimate $\nabla \bar u_{\delta}$ and rewrite $u$-equation of \eqref{MFG_mass_supercritical6_translated} as
\begin{equation*}
    -\Delta \bar u_{\delta}=-H(\nabla \bar u_{\delta})+\hat f_{\delta}(x),\,\text{where}\,\, \hat f_{\delta}(x):=-\mu_{\delta} \bar m_{\delta}^{\alpha}+\delta V( x+y_{\delta})+1.
\end{equation*}
We first obtain from \textbf{Claim B} that
\begin{equation}\label{potential_tra}
\lim\limits_{\delta\to0^+}\delta V(x+y_{\delta})=0\,\,\text{a.e. in }\, \mathbb{R}^n.
\end{equation}
{When $\gamma' \leq n$,} since $\bar{m}_{\delta}^{\alpha} \in L^{1+\frac{1}{\alpha}}(\mathbb{R}^n)$, by utilizing the maximal regularity shown in Lemma \ref{thmmaximalregularity},  we get $|\nabla \bar{u}_{\delta}|^{\gamma-1} \in L_{\mathrm{loc}}^{(1+\frac{1}{\alpha})\gamma'}(\mathbb{R}^n)$ with $(1+\frac{1}{\alpha})\gamma' > n$.  Thus, invoking Theorem 1.7.4 in \cite{bogachev2022fokker} and Morrey's embedding, we have from the $m$-equation in \eqref{MFG_mass_supercritical6_translated} that $\bar m_{\delta}\in C^{0,\theta}_{loc}(\mathbb{R}^n)$ for some $\theta\in(0,1)$.{ When $\gamma' > n$, the local H\"{o}lder regularity of $\bar m_{\delta}$  follows directly from the Sobolev embedding theorem, since $\bar m_{\delta}\in W^{1,\hat q}(\mathbb{R}^n)$ with $\hat q>n$.}  Applying Theorem A.1 in \cite{Lasry1989Lions} together with the argument used in the proof of Theorem 2.5 in \cite{cesaroni2018concentration}, we further obtain 
$|\nabla \bar u_{\delta}|\leq C_{R}<+\infty\,\,\text{for any}\,\,|x|<2R$. Therefore, one has
$$\Vert H(\nabla \bar u_{\delta})\Vert_{L^{p}(B_{2R}(0))}+\Vert \hat f_{\delta}\Vert_{L^{p}(B_{2R}(0))}\leq C_{p,R}<+\infty.$$ 
Then, proceeding the similar argument shown in the proof of \eqref{eq4.401}, we get that
\begin{equation*}
\|\bar u_{\delta}\|_{C^{2,\bar\theta_1}(B_{R}(0))} \leq C_{\bar\theta_1,R}<\infty \text{ for some $\bar \theta_1\in(0,1)$,}
\end{equation*}
which implies that there exists $u_{0}\in C^2(\mathbb R^n)$ such that
\begin{equation}\label{eq4.421}
\bar u_{\delta}\overset{\delta\to 0^+}\longrightarrow u_{0} \text{ in }C^{2,\bar\theta_1}_{\rm loc}(\mathbb{R}^n)\text{ for some $\bar\theta_1\in(0,1)$.}
\end{equation}
On the other hand, from \eqref{potential_free7_3} and \eqref{potential_free14}, we see that there exists $(m_{0},w_{0})\in W^{1,\hat q}(\mathbb R^n)\times \big(L^{1}(\mathbb R^n)\cap L^{\hat q}(\mathbb R^n) \big)$ such that
\begin{align}\label{eq4.431}
\bar m_{\delta}\overset{\delta\to0^+}\to m_{0} \text{ a.e. in $\mathbb R^n$,  \ and \ \ }(\bar m_{\delta},\bar w_{\delta})\overset{\delta\to 0^+}\rightharpoonup (m_{0},w_{0}) \text{~weakly in~}  W^{1,\hat{q}}(\mathbb R^n)\times L^{\hat q}(\mathbb R^n),
\end{align}
where $m_{0}\not\equiv  0$ due to \eqref{potential_free8}.  {
Then, applying Lemma \ref{blowupanalysismlinfboundcritical} to \eqref{MFG_mass_supercritical6_translated}, we obtain $\sup\limits_{\delta}\Vert \bar m_{\delta}\Vert_{L^\infty(\mathbb R^n)}\leq C$ for some constant $C>0$, which indicates $m_0 \in{L^\infty}(\mathbb R^n)$}.  Therefore, combining \eqref{potential_free7_1}, \eqref{MFG_mass_supercritical6_translated},  \eqref{potential_tra}, \eqref{eq4.421} and \eqref{eq4.431}, we apply a diagonal argument to conclude that $(u_{0},m_{0}, \mu_{0})\in C^{2}(\mathbb R^n)\times W^{1,\hat{q}}(\mathbb R^n)\times\mathbb{R} $ with $\mu_{0}=e_{0}$ satisfies
\begin{align}\label{MFG_mass_supercritical23}
\left\{\begin{array}{ll}
	-\Delta u_0+H(\nabla u_0)+\mu_{0} m_{0}^{\alpha}=1,&x\in\mathbb R^n,\\
	-\Delta m_{0}+\nabla\cdot w_{0}=0,&x\in\mathbb R^n,\\
    w_{0}=-m_{0}\nabla H(\nabla u_0).
\end{array}
\right.
\end{align}
In view of \textbf{Claim A} and Fatou's Lemma, one finds that
\begin{align}\label{exponentialdecaym123}
0<\int_{\mathbb R^n}m_{0}^{1+\alpha}\,dx\leq 1.
\end{align}

Next, we show that $(m_{0}, w_{0})\in\mathcal{A}\cap \left\{m: \Vert m\Vert_{L^{1+\alpha}(\mathbb R^n)}=1 \right\}$ is a minimizer of $e_{0}$. Indeed, invoking Lemma \ref{sect2-lemma21-gradientu} and $m_0\in L^\infty(\mathbb R^n)$, we have $\Vert\nabla u_0\Vert_{L^\infty(\mathbb R^n)}\leq C $ for some positive constant $C.$  Then, by the standard elliptic regularity, we get $(u_{0},m_{0})\in C^2(\mathbb R^n)\times W^{1,p}(\mathbb R^n)$ for any $p>1$. Moreover, from Lemma \ref{mdecaylemma} it follows that
\begin{align}\label{exponentialdecaym1234}
	m_{0}(x)\leq C_1 e^{-C_2|x|}  ~\text{ for all } x\in \mathbb R^n,
\end{align}
where $C_1,C_2>0$ are some constants. Then, we have
\begin{align*}
	\int_{\mathbb{R}^n} V(x)m_{0}\,dx< \infty . 
\end{align*}
Given this and that $(m_0, w_0)$ satisfies \eqref{MFG_mass_supercritical23}, we conclude that $(m_0, w_0) \in \mathcal{A}$. Moreover, thanks to \eqref{exponentialdecaym1234}, we apply Pohozaev identities given in Lemma \ref{poholemma} and \eqref{potential_free7_1} to obtain
\begin{equation*}
    \frac{ n \alpha}{(\alpha+1)\gamma'-n\alpha}\int_{\mathbb R^n}m_{0}\, dx=\int_{\mathbb R^n}m_{0}L\left(-\frac{w_{0}}{m_{0}}\right)\, dx=\frac{e_{0} n \alpha}{(\alpha+1)\gamma'}\int_{\mathbb R^n}m_{0}^{1+\alpha}\, dx.
\end{equation*}
 On the other hand,  since $(m_{0}, w_{0})\in\mathcal{A}$, one gets from the definition of $e_0$ and the above two estimates that  
 
\begin{equation*}
 e_{0}\leq\frac{\int_{\R^n} m_{0} L\left (-\frac{w_{0}}{m_{0}}\right) \, dx+\int_{\R^n}m_{0}\,dx}{      \left(\int_{\mathbb{R}^n}m_{0}^{1+\alpha}\,dx \right)^{\frac {1} {1+\alpha}  }  }=e_0\left(\int_{\mathbb{R}^n}m_{0}^{1+\alpha}\,dx \right)^{\frac {\alpha} {1+\alpha}  }.
 \end{equation*}
This indicates that  
 \begin{equation*}
 \int_{\mathbb{R}^n}m_{0}^{1+\alpha}\,dx\geq 1.
 \end{equation*}
From this and \eqref{exponentialdecaym123}, we get 
 $ \int_{\mathbb{R}^n}m_{0}^{1+\alpha}\,dx=1.$ Hence, the triple $(u_{0},m_{0},\mu_{0})$ solves \eqref{mass_supercritical_problem_free_potential}. In addition, we obtain $\bar m_{\delta}\to m_{0}$ in $L^{1+\alpha}(\mathbb{R}^n)$ as $\delta\to 0^{+}$. And by \eqref{eq4.431}, one has $\bar m_{\delta}\to m_{0} \,\,\text{strongly in }\,\, L^{p}(\mathbb{R}^n),\,\,\forall p\in[1+\alpha, \hat q^{*})$, as $\delta \to 0^+$.
{ Finally, due to \eqref{potential_free14} and \eqref{eq4.431}, we get from Lemma \ref{re_energy} and Fatou's lemma that
\begin{equation*}
 \begin{aligned}
	e_{0}=\lim_{\delta\to 0^{+}}e_{\delta}&=\lim_{\delta\to 0^{+}}\left\{\int_{\mathbb R^n} m_{\delta}L\left(-\frac{ w_{\delta}}{m_{\delta}}\right)\, dx+\int_{\mathbb R^n} \left(\delta V(x)+1\right)m_{\delta}\,dx\right\}\\
    &\geq\lim_{\delta\to 0^{+}}\left\{\int_{\mathbb R^n} m_{\delta}L\left(-\frac{ w_{\delta}}{m_{\delta}}\right)\, dx+\int_{\mathbb R^n} m_{\delta}\,dx\right\}\\
    &=\lim_{\delta\to 0^{+}}\left\{\int_{\mathbb R^n}\bar m_{\delta}L\left(-\frac{\bar w_{\delta}}{\bar m_{\delta}}\right)\, dx+\int_{\mathbb R^n}\bar m_{\delta}\,dx\right\}\\
    &\geq \int_{\mathbb R^n}m_{0}L\left(-\frac{w_{0}}{m_{0}}\right)\, dx+\int_{\mathbb R^n}m_{0}\,dx=\mathcal{E}_{0}(m_{0},w_{0})\geq e_{0},
 \end{aligned}
\end{equation*}
which indicates that $(m_{0}, w_{0})\in\mathcal{A}\cap \left\{m: \Vert m\Vert_{L^{1+\alpha}(\mathbb R^n)}=1 \right\}$ is a minimizer of $e_{0}$. }
\end{proof}

\textbf{Proof of Lemma \ref{Existience_of_MFG_free_potential}:}  With the aid of Lemma \ref{re_energy} and \ref{exi_potential_free_MFG}, we readily obtain the conclusion of Lemma \ref{Existience_of_MFG_free_potential}  { except for the positivity of  $m_{0}$ and $\mu_{0}$. In fact, by applying the classical maximum principle and the Pohozaev identity, we can prove that $m_{0}>0$ and $\mu_{0}>0$, respectively.}
$\hfill\qed$

\medskip

Lemma \ref{Existience_of_MFG_free_potential} implies that the potential-free auxiliary MFG system \eqref{mass_supercritical_problem_free_potential} admits a classical solution. Based on this and the previous results, we can now discuss the existence of mountain-pass solutions to system~\eqref{L^1_constrain_equation} and conclude the proof of Theorem~\ref{Existience_of_MPS_MFG} and Theorem \ref{GNinequalitythm}.
{

\section{Existence of Mountain-Pass Solutions to Potential-Free MFG Systems }\label{mountain_pass_solution_MFG}
In this section, we shall prove Theorem \ref{Existience_of_MPS_MFG} and Theorem \ref{GNinequalitythm}. In detail, we first aim to establish the Gagliardo-Nirenberg type inequality (\ref{Gn_Inequality}).  Then, we show that, the optimizers of (\ref{Gn_Inequality}) are, in fact, equivalent to those of the minimization problem~\eqref{e_0_min}.  Finally, in the mass-supercritical case, we prove that the optimizers of the Gagliardo-Nirenberg type inequality also attain the  minimax problem
\eqref{constrain_mountain_pass__problem}.  Combining these results with Lemma \ref{Existience_of_MFG_free_potential}, one obtains the conclusion in Theorem \ref{Existience_of_MPS_MFG} holds by using the scaling argument.  Finally, we show that (\ref{Gn_Inequality}) under the mass-supercritical regime is attained by the mountain-pass solution given in Theorem \ref{Existience_of_MPS_MFG}, which proves Theorem \ref{GNinequalitythm}.

\subsection{Gagliardo-Nirenberg type inequality: potential-free MFG systems}\label{sect3-optimal} 
This subsection is devoted to the existence of minimizers for problem~\eqref{Gn_Inequality} and their relationship with the minimizers of~\eqref{e_0_min}.  In section \ref{MFG_sys}, we have obtained that the minimization problem \eqref{e_0_min} is attained by some pair $(m_0,w_0)$  and there exist $u_0$ and $\mu_{0}$ such that $(u_{0}, m_{0},\mu_{0})$ solves \eqref{mass_supercritical_problem_free_potential}.

We next prove the equivalence between minimization problems \eqref{e_0_min} and problem \eqref{Gn_Inequality}, which is 
\begin{lemma}\label{sect3_lemm32}
	Assume that $\gamma'>1$ and $\alpha\in\left(0,\alpha^{*}\right)$. Let $(m_0,w_0)$ be the minimizer for $e_0$ given in \eqref{e_0_min}  obtained by Lemma \ref{Existience_of_MFG_free_potential}. Then $(m_0,w_0)$ also minimizes the problem \eqref{Gn_Inequality}.  Moreover, there holds
	\begin{align}\label{sect3-relation-Calpha-emalpha}
		\Gamma_{\alpha}=\left(\frac{n\alpha}{(1+\alpha)\gamma'-n\alpha}\right)^{\frac{n\alpha}{\gamma'}}\left(\frac{(1+\alpha)\gamma'-n\alpha}{(1+\alpha)\gamma'}  e_{0}\right)^{1+\alpha}.
	\end{align}
\end{lemma}
\begin{proof}
	Since problem \eqref{Gn_Inequality} and \eqref{Gn_Inequality_constrain} are equivalent, we just need to show that $(m_{0}, w_{0})$ optimizes problem \eqref{Gn_Inequality_constrain}.
    To begin with, we set
	\begin{align*}
		G_{\alpha}(m,w):=\left(\int_{\mathbb R^n}mL\left(-\frac{w}{m}\right)\, dx\right)^{\frac{n\alpha}{\gamma'}}\left(\int_{\mathbb R^n}m\,dx\right)^{\frac{(\alpha+1)\gamma'-n\alpha}{\gamma'}},
	\end{align*}
	then, one has
	\begin{align}\label{sect3-Calpha-equivalent}
		\Gamma_{\alpha}=\inf_{(m,w)\in\mathcal{A}\cap \left\{m: \Vert m\Vert_{L^{1+\alpha}(\mathbb R^n)}=1 \right\}}G_{\alpha}(m,w).
	\end{align}
	
	For any $(m,w)\in\mathcal{A}\cap \left\{m: \Vert m\Vert_{L^{1+\alpha}(\mathbb R^n)}=1 \right\}$ satisfying $\int_{\mathbb R^n}mL\left(-\frac{w}{m}\right)\, dx<\infty$.
	Let $(m^{t}(x),w^{t}(x)):=(t^{\frac{n}{1+\alpha}} m(tx),t^{1+\frac{n}{1+\alpha}}w(tx))$ for $t\in\mathbb R^+$, then $(m^{t}(x),w^{t}(x))\in \mathcal{A}\cap \left\{m: \Vert m\Vert_{L^{1+\alpha}(\mathbb R^n)}=1 \right\}$. By simple computations, we have 
	\begin{align}\label{sect3-311-mathcalE0}
		\mathcal E_0(m^{t},w^{t})=&t^{\frac{(1+\alpha)\gamma'-n\alpha}{1+\alpha}}\int_{\mathbb R^n}mL\left(-\frac{w}{m}\right)\, dx+t^{-\frac{n\alpha}{1+\alpha}}\int_{\mathbb R^n}m\, dx\nonumber\\
		\geq&\left(\frac{(1+\alpha)\gamma'}{(1+\alpha)\gamma'-n\alpha}\right) \left(\frac{(1+\alpha)\gamma'-n\alpha}{n\alpha}\right)^{\frac{n\alpha}{(1+\alpha)\gamma'}} \left(\int_{\mathbb R^n}m\,dx\right)^{\frac{(1+\alpha)\gamma'-n\alpha}{(1+\alpha)\gamma'}}\left(\int_{\mathbb R^n}mL\left(-\frac{w}{m}\right)\, dx\right)^{\frac{n\alpha}{(1+\alpha)\gamma'}},
	\end{align}
	where the ``=" holds if and only if
	\begin{align*}
{t}=\bar{t}:=\left(\frac{n\alpha\int_{\mathbb R^n}m\, dx}{\left((\alpha+1)\gamma'-n\alpha\right)\int_{\mathbb R^n}mL\left(-\frac{w}{m}\right)\,dx}\right)^{\frac{1}{\gamma'}}.
	\end{align*}
	Recalling the definition of $e_{0}$ defined in \eqref{e_0_min}, we find from (\ref{sect3-311-mathcalE0}) that
	\begin{align}\label{513-estimate-sect3}
		e_{0}\leq {\mathcal{E}}_{0}(m^{\bar t}, w^{\bar t}) =\left(\frac{(1+\alpha)\gamma'}{(1+\alpha)\gamma'-n\alpha}\right) \left(\frac{(1+\alpha)\gamma'-n\alpha}{n\alpha}\right)^{\frac{n\alpha}{(1+\alpha)\gamma'}} \left(G_{\alpha}(m,w)\right)^{\frac{1}{1+\alpha}}.
	\end{align}
	Taking the infimum over $ \mathcal{A}\cap \left\{m: \Vert m\Vert_{L^{1+\alpha}(\mathbb R^n)}=1 \right\} $ in \eqref{513-estimate-sect3} and by using the fact $(m_{0},w_{0})\in \mathcal{A}\cap \left\{m: \Vert m\Vert_{L^{1+\alpha}(\mathbb R^n)}=1 \right\}$, we deduce that
	\begin{align}\label{313-estimate-sect3}
		e_{0}&\leq \left(\frac{(1+\alpha)\gamma'}{(1+\alpha)\gamma'-n\alpha}\right) \left(\frac{(1+\alpha)\gamma'-n\alpha}{n\alpha}\right)^{\frac{n\alpha}{(1+\alpha)\gamma'}} \left(\Gamma_{\alpha}\right)^{\frac{1}{1+\alpha}}\nonumber\\
        &\leq \left(\frac{(1+\alpha)\gamma'}{(1+\alpha)\gamma'-n\alpha}\right) \left(\frac{(1+\alpha)\gamma'-n\alpha}{n\alpha}\right)^{\frac{n\alpha}{(1+\alpha)\gamma'}} \left(G_{\alpha}(m_{0},w_{0})\right)^{\frac{1}{1+\alpha}}.
	\end{align}
    On the other hand, by taking $(m,w)=(m_{0},w_{0})$ and $t=1$ in \eqref{sect3-311-mathcalE0}, we get
	\begin{align}\label{514-estimate-sect3}
		e_{0}=\mathcal E_0(m_{0},w_{0})\geq\left(\frac{(1+\alpha)\gamma'}{(1+\alpha)\gamma'-n\alpha}\right) \left(\frac{(1+\alpha)\gamma'-n\alpha}{n\alpha}\right)^{\frac{n\alpha}{(1+\alpha)\gamma'}} \left(G_{\alpha}(m_{0},w_{0})\right)^{\frac{1}{1+\alpha}}.
	\end{align}
    As a result, we obtain that 
\begin{align}\label{314-estimate-sect3}
		e_{0}\geq\left(\frac{(1+\alpha)\gamma'}{(1+\alpha)\gamma'-n\alpha}\right) \left(\frac{(1+\alpha)\gamma'-n\alpha}{n\alpha}\right)^{\frac{n\alpha}{(1+\alpha)\gamma'}} \left(\Gamma_{\alpha}\right)^{\frac{1}{1+\alpha}}.
	\end{align}
	Collecting \eqref{513-estimate-sect3}-\eqref{314-estimate-sect3}, we conclude that \eqref{sect3-Calpha-equivalent} is attained by $(m_{0},w_{0})$ and \eqref{sect3-relation-Calpha-emalpha} holds.  
\end{proof}

Based on Lemma \ref{sect3_lemm32}, we next investigate the existence of mountain-pass solution to MFG system \eqref{L^1_constrain_equation} and prove   Theorem \ref{Existience_of_MPS_MFG}.

\subsection{Mountain-pass solutions under the mass-supercritical regime}\label{mountain_pass_geometryMFG}
{Before proving Theorem \ref{Existience_of_MPS_MFG} via Lemma \ref{sect3_lemm32}, we first establish the following key lemma on the existence of solutions to MFG system \eqref{L^1_constrain_equation}.  Indeed, with the aid of Lemma  \eqref{Existience_of_MFG_free_potential}, we apply the scaling argument and obtain   }
\begin{lemma}\label{resolMFG}
For every $p>1$, there exists a classical solution $(\bar u_0,\bar m_0,\bar\lambda_0)\in C^2(\mathbb R^n)\times W^{1,p}(\mathbb R^n)\times\mathbb R$ to system \eqref{L^1_constrain_equation}.
    In addition, the pair  $(\bar m_{0},\bar w_{0})\in \mathcal{A}\cap \left\{m: \Vert m\Vert_{L^{1}(\mathbb R^n)}=M \right\}$ with $\bar w_{0}=- \bar m_{0}\nabla H(\nabla \bar u_{0})$ is a minimizer of \eqref{Gn_Inequality}.
\end{lemma}
\begin{proof}

Let $\big(u_{0},m_{0},\mu_{0}\big)\in C^2(\mathbb R^n)\times W^{1,p}(\mathbb R^n)\times \mathbb R$, $\forall p>1$, be a solution of system \eqref{mass_supercritical_problem_free_potential} obtained in Lemma \ref{Existience_of_MFG_free_potential}. Correspondingly, $(m_{0},w_{0})\in \mathcal{A}\cap \left\{m: \Vert m\Vert_{L^{1+\alpha}(\mathbb R^n)}=1 \right\}$ is the minimizer of the Gagliardo-Nirenberg type inequality \eqref{Gn_Inequality} with $w_{0}=-m_{0}\nabla H(\nabla u_0)$.
By using the Pohozaev identities \eqref{eq2.49}, we obtain that $(u_{0},m_{0},\mu_{0})$ satisfies
	\begin{align*}
	\left\{\begin{array}{ll}
		-\Delta u_{0}+H(\nabla u_{0})+\mu_{0} m_{0}^{\alpha}=1,&x\in\mathbb R^n,\\
		\Delta m_{0}+\nabla\cdot (m_{0}\nabla H(\nabla u_0))=0,&x\in\mathbb R^n,\\
		 \frac{(\alpha+1)\gamma'}{\mu_{0}\left[(\alpha+1)\gamma'- n\alpha\right]}\int_{\mathbb R^n}m_{0}\,dx=1.
	\end{array}
	\right.
\end{align*}
Let
\begin{equation}\label{eqs1}
	\bar m_{0}(x):= s_{0}m_{0}(r_{0}x), \, \bar u_{0}(x):=t_{0}u_{0}(r_{0}x)\,\, \text{and}\,\, \bar\lambda_{0}:=-t_{0}r_{0}^{2},
\end{equation}
 where 
 \begin{equation}\label{eqs2}
 s_{0}=\mu_{0}^{\frac{1}{\alpha}}\mathcal{T}^\frac{\gamma'}{n\alpha-\gamma'} ,\,  t_{0}=\mathcal{T}^\frac{\alpha(\gamma'-2)}{n\alpha-\gamma'}, \,\,    r_{0}=\mathcal{T}^\frac{\alpha}{n\alpha-\gamma'}\,\,\text{with} \,\, \mathcal{T}:=\left(1-\frac{n\alpha}{(1+\alpha)\gamma'}\right)M^{-1}\mu_{0}^{1+\frac{1}{\alpha}}.
 \end{equation}
Then, the remaining proof of this lemma follows easily from the scaling \eqref{eqs1} and \eqref{eqs2}.
\end{proof}

We are now ready to show that the solution of the MFG system \eqref{L^1_constrain_equation} given by Lemma \ref{resolMFG} is a mountain-pass solution  when $\alpha\in(\alpha_{*},\alpha^{*})$.
To begin with, we note that the solution to equations \eqref{L^1_constrain_equation} formally corresponds to a critical point of the functional $\mathcal{J}_{0}$ given in \eqref{L^1_constrain_functional}.  {Therefore, we now turn to the study of the connection between the solutions of \eqref{L^1_constrain_equation} obtained in Lemma \ref{resolMFG} and the minimax values of the functional $\mathcal{J}_{0}$ subject to the constraint set $\mathcal{A}\cap \left\{m: \Vert m\Vert_{L^{1}(\mathbb R^n)}=M \right\}$.}  We first show that the functional  $\mathcal{J}_{0}$ possesses a mountain pass geometry on $\mathcal{A}\cap \left\{m: \Vert m\Vert_{L^{1}(\mathbb R^n)}=M \right\}$.  To see this, we establish a preliminary lemma as follows:
\begin{lemma}\label{MP_1}
{Assume that $\alpha\in(\alpha_{*},\alpha^{*})$}. Then there exists $R_{0}>0$ such that  
   \begin{equation*}
       0<\sup_{(m,w)\in B_{R_{0}}}\mathcal{J}_{0}(m,w)<\inf_{(m,w)\in \partial B_{2R_{0}}}\mathcal{J}_{0}(m,w),
   \end{equation*}
   where 
   \begin{equation}\label{set_Mp}
       B_{R_{0}}:= \left\{(m,w)\in \mathcal{A}\cap \left\{m: \Vert m\Vert_{L^{1}(\mathbb R^n)}=M \right\}: \int_{\mathbb{R}^{n}}L\bigg(-\frac{w}{m}\bigg)m\, dx \leq R_{0} \right\}.
\end{equation}
\end{lemma}
\begin{proof}
 Suppose that $(m,w), (\tilde{m},\tilde{w}) \in \mathcal{A}\cap \left\{m: \Vert m\Vert_{L^{1}(\mathbb R^n)}=M \right\}$ satisfying
\begin{equation*}
\int_{\mathbb{R}^{n}}mL\left(-\frac{w}{m}\right)\,dx\leq R_{0}~\text{and}~ \int_{\mathbb{R}^{n}}\tilde{m}L\left(-\frac{\tilde{w}}{\tilde{m}}\right)\,dx =2 R_{0}, 
\end{equation*}
where $R_{0}>0$ is arbitrary but fixed. Then, by (\ref{Gn_Inequality}) and $\alpha>\alpha_{*}=\frac{\gamma'}{n}$, we obtain that for some $C=C(M,R_0)>0,$
\begin{align*}
   \mathcal{J}_{0}(\tilde{m},\tilde{w})-\mathcal{J}_{0}(m,w)&\geq \int_{\mathbb{R}^{n}}L\bigg(-\frac{\tilde{w}}{\tilde{m}}\bigg)\tilde{m}\,dx-\int_{\mathbb{R}^{n}}L\bigg(-\frac{w}{m}\bigg)m\,dx-\frac{1}{1+\alpha}\int_{\mathbb{R}^{n}} {\tilde{m}}^{1+\alpha}\,dx \\
   &\geq R_{0}-\frac{1}{(1+\alpha){ \Gamma_{\alpha}}}\left(\int_{\mathbb R^n}\tilde mL\left(-\frac{\tilde w}{\tilde m}\right)\,dx\right)^{\frac{n\alpha}{\gamma'}}    \\
   &\geq R_{0}-CR_{0}^{\frac{n\alpha}{\gamma'}}\geq \frac{R_{0}}{2}>0,
\end{align*}
provided that $R_{0}>0$ is chosen sufficiently small. In addition, for $(m,w)$ and $R_{0}$ given above, we further have
\begin{align*}
   \mathcal{J}_{0}(m,w)\geq C_L\int_{\mathbb{R}^{n}}L\bigg(-\frac{w}{m}\bigg)m\,dx-C{R_{0}}^{\frac{n\alpha}{\gamma'}}\geq \frac{R_{0}}{2}>0.
\end{align*}
Therefore, the desired conclusion holds.
\end{proof}
Now,  let $(m_{1},w_{1})\in B_{R_{0}}$ and  $(m_{2},w_{2})\in B_{2R_{0}}^{\, c}$, $B_{R_{0}}$ with $R_{0}>0$ is given by lemma \ref{MP_1}, be such that 
  \begin{equation}\label{mountain_pass_level_00}
\mathcal{J}_{0}(m_{1},w_{1})>0> \mathcal{J}_{0}(m_{2},w_{2}).
\end{equation}
Following the ideas outlined in \cite{Jeanjean1997}, we can define the mountain pass level
\begin{equation}\label{mountain_pass_level_0}
    e_{MP}:=\inf_{h\in \Pi}\max_{t\in[0,1]}\mathcal{J}_{0}(h(t)),
\end{equation}
where 
\begin{equation}\label{mounpass_path}
    \Pi:=\bigg\{  h\in C\left([0,1],\mathcal{A}\cap \left\{m: \Vert m\Vert_{L^{1}(\mathbb R^n)}=M \right\}\right): h(0)=(m_{1},w_{1})\,\,\text{and}\,\, h(1)=(m_{2},w_{2})\bigg\}.
\end{equation}
\begin{remark}\label{mountain_pass_geometry_1}
The existence of $(m_{i},w_{i})$, $i=1,2,$ mentioned above should be addressed. Indeed, let $(m,w)\in\mathcal{A}\cap \left\{m: \Vert m\Vert_{L^{1}(\mathbb R^n)}=M \right\}$ be arbitrary but fixed, one can easily check that
\begin{equation}\label{scaling_mw}
   (m_{t}(x),w_{t}(x)):=\big(t^{n}m(tx),t^{n+1}m(tx)\big)\in \mathcal{A}\cap \left\{m: \Vert m\Vert_{L^{1}(\mathbb R^n)}=M \right\} \text{ for}\,\,t\in\mathbb{R}^{+}.
\end{equation}
Moreover, there exist $t_{1}>0$ small enough and $t_{2}>0$ large enough such that $(m_{i},w_{i}):=(  m_{t_{i}}, w_{t_{i}})\in \mathcal{A}\cap \left\{m: \Vert m\Vert_{L^{1}(\mathbb R^n)}=M \right\}$, $i=1,2$, satisfying
\begin{equation*}
     \int_{\mathbb{R}^{n}}m_{1}L\left(-\frac{w_{1}}{m_{1}}\right)\,dx\leq R_{0}\,\, \text{and} \,\,\int_{\mathbb{R}^{n}}m_{2}L\left(-\frac{w_{2}}{m_{2}}\right)\,dx> 2R_{0}
     \end{equation*}
     for $R_{0}>0$ given by Lemma \ref{MP_1}.
\end{remark}

We emphasize that, in the definition of the mountain-pass level given in \eqref{mountain_pass_level_0}, the two endpoints $(m_{1},w_{1})$ and $(m_{2}, w_{2})$ are  fixed. In fact, this requirement may be relaxed, and an equivalent definition will be established below.  We remark that the argument is inspired by the result in \cite{Jeanjean1997}.
\begin{lemma}\label{mp_equi}
   Let $B_{R_{0}}$ be given in \eqref{set_Mp}, then $B_{R_{0}}$ and  $\overline{B_{2R_{0}}^{\,c}}\, \cap \left\{(m,w):\mathcal{J}_{0}(m,w)\leq 0\right \}$ are arc-connected. Moreover, it holds 
    \begin{equation}\label{mountain_pass_level_001}
        e_{MP}= \inf_{g\in\widetilde \Pi}\max_{s\in[0,1]} \mathcal{J}_{0}(g(s)),   
\end{equation}
    where $e_{MP}$ is given in \eqref{mountain_pass_level_0} and
\begin{equation}\label{mounpass_path_tilde}
        \widetilde \Pi:=\bigg\{  g\in C\left([0,1],\mathcal{A}\cap \left\{m: \Vert m\Vert_{L^{1}(\mathbb R^n)}=M \right\}\right): g(0)\in B_{R_{0}}\,\,\text{and}\,\, g(1)\in \overline{B_{2R_{0}}^{\,c}}\, \cap \left\{(m,w):\mathcal{J}_{0}(m,w)\leq 0\right \}\bigg\}. 
\end{equation} 
\end{lemma}
\begin{proof}
To begin with, 
let $(m_i,w_i)\in \mathcal{A}\cap \left\{m:\int_{\mathbb{R}^n}m\,dx=M\right \}$, $i=1,2$, be two distinct pairs. We define $ h\Big((m_1, w_1), (m_2, w_2), s, t \Big):\bigg \{\mathcal{A}\cap \left\{m:\int_{\mathbb{R}^n}m\,dx=M\right \} \bigg\}^{2}\times[0,1] \times \mathbb{R}^{+}\to \mathcal{A} \cap \left\{m:\int_{\mathbb{R}^n}m\,dx=M\right\}$ as
     \begin{equation*}
       h\Big((m_1, w_1), (m_2, w_2), s, t \Big):=\Big((1-s) t^{n}m_{1}(tx)+st^{n} m_{2}(tx),~ (1-s)t^{n+1}w_{1}(tx)+s t^{n+1}w_{2}(tx) \Big).
    \end{equation*}
Then, it can also be verified that, for all  $t\in \mathbb{R}^+$,
$$h\Big((m_{1}, w_{1}), (0, 0), 0, t \Big), h\Big((0, 0), (m_{2}, w_{2}), 1, t \Big)  \in \mathcal{A}\cap \left\{m:\int_{\mathbb{R}^n}m\,dx=M\right\}.$$
We further assume that 
    \begin{equation}\label{norm_mass_path}
     a:=\int_{\mathbb{R}^{n}}m_{1}L\left(-\frac{w_{1}}{m_{1}}\right)\,dx=\int_{\mathbb{R}^{n}}m_{2}L\left(-\frac{w_{2}}{m_{2}}\right)\,dx>0,
     \end{equation}
    and that $w_1$ and $w_2$ are linearly independent.  Let  
\begin{equation*}
\begin{aligned}
 I\Big((m_{1},w_{1}), (m_{2},w_{2}),s,t\Big):&= \int_{\mathbb{R}^{n}}\left((1-s)t^{n}m_{1}(tx) +st^{n}m_{2}(tx)\right)L\left(-\frac{(1-s)t^{n+1}w_{1}(tx) +st^{n+1}w_{2}(tx)}{(1-s)t^{n}m_{1}(tx)+st^{n}m_{2}(tx)}\right)\,dx\\
 &=t^{\gamma'}\int_{\mathbb{R}^{n}}\left((1-s)m_{1} +s m_{2}\right)L\left(-\frac{(1-s)w_{1} +sw_{2}}{(1-s)m_{1}+sm_{2}}\right)\,dx,
 \end{aligned}
\end{equation*}
where $s\in[0,1]$ and $t\in \mathbb{R}^+$. On the one hand, using the convexity of $\int_{\mathbb{R}^n}mL\left(-\frac{w}{m}\right)\,dx$ and \eqref{norm_mass_path}, we obtain
\begin{equation*}
\begin{aligned}
 I\Big((m_{1},w_{1}), (m_{2},w_{2}),s,t\Big)&\leq t^{\gamma' }(1-s)\int_{\mathbb{R}^n}m_{1}  L\left(- \frac{w_{1}}  {m_{1}} \right)\,dx+t^{\gamma' }s\int_{\mathbb{R}^n} m_{2}L \left(-\frac{w_{2}}  {m_{2}} \right)\,dx=at^{\gamma'},
 \end{aligned}
\end{equation*}
for any $t\in \mathbb{R}^+$ and all $s\in[0,1]$. On the other hand, due to the linear independence of $w_{1}$ and $w_{2}$, we deduce that for all $t\in \mathbb{R}^{+}$,
\begin{equation*}
    I\Big((m_{1},w_{1}), (m_{2},w_{2}),s,t\Big) \geq 
     b t^{\gamma'},
\end{equation*}
where $$b:=b\Big((m_{1},w_{1}),(m_{2},w_{2})\Big):=\min_{s\in[0,1]}\int_{\mathbb{R}^{n}}\left((1-s)m_{1} +s m_{2}\right)L\left(-\frac{(1-s)w_{1} +sw_{2}}{(1-s)m_{1}+sm_{2}}\right)\,dx>0.$$
Therefore, we conclude that 
\begin{equation}\label{norm_mass_path_7}
    bt^{\gamma'}\leq I\Big((m_{1},w_{1}), (m_{2},w_{2}),s,t\Big) \leq at^{\gamma'}, \,\,\text{for any}\,\, s\in[0,1]\,\,\text{and}\,\, t\in \mathbb{R}^{+}. 
\end{equation}

In addition, noting that $(m_{i},w_{i})\in \mathcal{A}\cap \left\{m:\int_{\mathbb{R}^n}m\,dx=M\right \}$, $i=1,2$, we have for any $t\in\mathbb{R}^{+}$ and any $s\in[0,1]$, 
\begin{equation}\label{norm_mass_path_8}
\begin{aligned}
   & \int_{\mathbb{R}^n} \Big((1-s) t^{n}m_{1}(tx) +st^{n}m_{2}(tx) \Big)^{1+\alpha}\,dx\\
   = & t^{n\alpha} \int_{\mathbb{R}^n} \Big((1-s) m_{1}(x) +sm_{2}(x) \Big)^{1+\alpha}\,dx\\
   \geq & t^{n\alpha} \left\{ (1-s)^{1+\alpha}\int_{\mathbb{R}^n} m_{1}^{1+\alpha}\, dx +s^{1+\alpha}\int_{\mathbb{R}^n} m_{2}^{1+\alpha}\, dx  \right\}
   \geq  c t^{n\alpha}, 
\end{aligned}
\end{equation}
where $$c:=c(m_{1},m_{2})=\min_{s\in [0,1]}\left\{ (1-s)^{1+\alpha}\int_{\mathbb{R}^n} m_{1}^{1+\alpha}\, dx +s^{1+\alpha}\int_{\mathbb{R}^n} m_{2}^{1+\alpha}\, dx  \right\}>0. $$

Now, we prove that $ B_{R_{0}}$ is arc-connected. Let $(\bar m_{i},\bar w_{i})\in B_{R_{0}}$, $i=1,2$, be two distinct arbitrary points. We finish the proof under the assumption that
\begin{equation}\label{const_topo_009}
     \int_{\mathbb{R}^{n}}\bar m_{1}L\left(-\frac{\bar w_{1}}{\bar m_{1}}\right)\,dx=\int_{\mathbb{R}^{n}}\bar m_{2}L\left(-\frac{\bar w_{2}}{\bar m_{2}}\right)\,dx=a,
 \end{equation}
and that $\bar w_1$  and $\bar w_2 $  are linearly independent, clearly $a\leq R_{0}$ here.  We note that if $\int_{\mathbb{R}^{n}}\bar m_{1}L\left(-\frac{\bar w_{1}}{\bar m_{1}}\right)\,dx\neq\int_{\mathbb{R}^{n}}\bar m_{2}L\left(-\frac{\bar w_{2}}{\bar m_{2}}\right)\,dx$, then we can apply the continuous transformation \eqref{scaling_mw} 
to obtain that, for some $\hat t>0$, $(\hat m_{1},\hat w_{1})=(\bar m_{1,t}, \bar w_{1,t})$ and $\int_{\mathbb{R}^{n}}\hat m_{1}L\left(-\frac{\hat w_{1}}{\hat m_{1}}\right)\,dx=\int_{\mathbb{R}^{n}}\bar m_{2}L\left(-\frac{\bar w_{2}}{\bar m_{2}}\right)\,dx$ .  In addition, if  $\bar w_1$  and $\bar w_2 $  are linearly dependent, we first introduce an intermediate pair $(\bar m_{3}, \bar w_{3})$ such that $\bar w_{3}$ is linearly independent of both $\bar w_{1}$ and $\bar w_{2}$, and then construct a path connecting $(\bar m_{1},\bar w_{1})$ to $(\bar m_{2},\bar w_{2})$. This path is obtained by concatenating the path from $(\bar m_{1},\bar w_{2})$ to $(\bar m_{3},\bar w_{3})$ with the path from $(\bar m_{3},\bar w_{3})$ to $(\bar m_{2},\bar w_{2})$. Under these assumptions, we now proceed with the proof. In view of \eqref{norm_mass_path_7}, we can choose $$\bar g(s)=h\Big( (\bar m_{1}, \bar w_{1}),(\bar m_{2}, \bar w_{2}), s, 1 \Big),\, \forall s\in[0,1]$$ as a path connecting $(\bar m_{1},\bar w_{1})$ and $(\bar m_{2},\bar w_{2})$, which satisfies $$\bar g(0)=(\bar m_{1},\bar w_{1}),\bar g(1)=(\bar m_{2},\bar w_{2}) \text{ and } \bar g(r)\in B_{R_{0}}.$$


Next, we show that $\overline{B_{2R_{0}}^{\,c}}\, \cap \left\{(m,w):\mathcal{J}_{0}(m,w)\leq 0\right \}$ is arc-connected. Indeed, suppose that $(\tilde m_{i},\tilde w_{i})\in\overline{B_{2R_{0}}^{\,c}}\, \cap \left\{(m,w):\mathcal{J}_{0}(m,w)\leq 0\right \}$, $i=1,2$, are arbitrary pairs, which satisfy the same condition as \eqref{const_topo_009},
    and that $\tilde w_1$  and $\tilde w_2 $  are linearly independent. Then, \eqref{norm_mass_path_7}, \eqref{norm_mass_path_8} and the fact $n\alpha>\gamma'$ yield that there exists $\tilde t_{0}=\tilde t_{0}\Big((\tilde m_1,\tilde w_1), (\tilde m_2,\tilde w_2)\Big)>0$ sufficiently large such that
$$I\Big((\tilde m_{1},\tilde w_{1}), (\tilde m_{2},\tilde w_{2}),s,\tilde t_{0}\Big)\geq 2 R_{0}\,\,\text{and}\,\, \mathcal{J}_{0}\left(h\Big((\tilde m_{1},\tilde w_{1}),(\tilde m_{2},\tilde w_{2}), s , \tilde t_{0}\Big)\right)\leq 0,\,\forall s\in [0,1].$$
Then, we construct a continuous path $\tilde g:[0, 3]\to \mathcal{A}\cap \left\{m:\int_{\mathbb{R}^n}m\,dx=M\right \}$ to connect $(\tilde m_{1},\tilde w_{1})$ and $(\tilde m_{2},\tilde w_{2})$ by 
\begin{equation*}
\tilde g(r):=\begin{cases} h\Big( (\tilde m_{1}, \tilde w_{1}),(0, 0),0 , 1+r(\tilde t_{0}-1)\Big) , \ &\text{ if } 0\leq r\leq 1,\\
h\Big( (\tilde m_{1}, \tilde w_{1}),(\tilde m_{2}, \tilde w_{2}), r-1, \tilde t_{0} \Big),\ &\text{ if }1\leq r\leq 2,\\
h\Big( (0,0),(\tilde m_{2}, \tilde w_{2}), 1, \tilde t_{0}+(r-2)(1-\tilde t_{0}) \Big),&\text{ if }2\leq r\leq 3.
\end{cases}
\end{equation*}
One can verify that $\tilde g(0)=(\tilde m_{1},\tilde w_{1})$, $\bar g(3)=(\tilde m_{2},\tilde w_{2})$ and $\tilde g\in\overline{B_{2R_{0}}^{\,c}}\, \cap \left\{(m,w):\mathcal{J}_{0}(m,w)\leq 0\right \}$.
 
Finally, from Lemma \eqref{MP_1} and \eqref{mountain_pass_level_0}, we see that \eqref{mountain_pass_level_001} holds. 
\end{proof}

{Invoking Lemma~\ref{mp_equi}, we have the mountain pass level $e_{MP}$
  defined in \eqref{mountain_pass_level_0} admits the equivalent characterization, which is shown as follows}
\begin{lemma}\label{Equivalent_mountain_level}
    Let $e_{MP}$ be the minimax value defined in \eqref{mountain_pass_level_0}. Then 
\begin{equation*}
    e_{MP}=e_{IM}:=\inf_{(m,w)\in \mathcal{A}\cap \left\{m: \Vert m\Vert_{L^{1}(\mathbb R^n)}=M \right\}}\max_{t\in \mathbb{R}^{+}}\mathcal{J}_{0}(m_{t},w_{t}),
\end{equation*}
where $(m_{t},w_{t})$ is given by \eqref{scaling_mw}.
\end{lemma}
\begin{proof}
    We first claim that
    \begin{equation}\label{eq-eP}
e_{IM}=e_{P}:=\inf_{(m,w)\in\mathcal{P}}\mathcal{J}_{0}(m,w),
  \end{equation}        
   where 
    \begin{equation*}
    \mathcal{P}:=\left\{(m,w)\in\mathcal{A}\cap \left\{m: \Vert m\Vert_{L^{1}(\mathbb R^n)}=M \right\}: P(m,w)=0\right\},
    \end{equation*}
with $ P(m,w)$ being defined as 
\begin{equation}\label{poho_equality}
   P(m,w):=\int_{\mathbb{R}^n}mL\left(-\frac{w}{m}\right)\,dx-\frac{n\alpha}{(1+\alpha)\gamma'}\int_{\mathbb{R}^n}m^{1+\alpha}\,dx.
\end{equation}
Indeed, by the scaling \eqref{scaling_mw}, one can obtain that $(m_{t},w_{t})\in \mathcal{A}\cap \left\{m: \Vert m\Vert_{L^{1}(\mathbb R^n)}=M \right\}$ for any $t\in \mathbb{R}^{+}$. Moreover, since $\alpha>\alpha^{*}=\frac{\gamma'}{n}$, it is straightforward to verify that the polynomial $\mathcal{J}_{0}(m_{t},w_{t})$ defined (for $t > 0$) by
\begin{equation}\label{deduce_GN}
		\mathcal{J}_{0}(m_{t},w_{t})=t^{\gamma'}\int_{\mathbb{R}^n}mL\left(-\frac{w}{m}\right)\,dx-\frac{t^{n\alpha}}{1+\alpha}\int_{\mathbb{R}^n}m^{1+\alpha}\,dx
	\end{equation}
possesses a unique maximum point $t_{m,w}\in\mathbb{R}^{+}$, and that $P(m_{t_{m,w}}, w_{t_{m,w}})=0$. Therefore, one can easy to obtain that  $e_{IM}=e_{P}$.

Then, we prove $e_{MP}= e_{P}$. On the one hand, we shall show that $e_{MP}\geq e_{P}$. For this purpose,  we separate $\mathcal{A}\cap \left\{m: \Vert m\Vert_{L^{1}(\mathbb R^n)}=M \right\}$ into two parts, more specifically,  
\begin{equation*}
   \mathcal{P}^{+}=\Big\{ (m,w)\in \mathcal{A}\cap \left\{m: \Vert m\Vert_{L^{1}(\mathbb R^n)}=M \right\}:  P(m,w)>0 \Big\},\,\,
   \mathcal{P}^{-}=\Big\{ (m,w)\in \mathcal{A}\cap \left\{m: \Vert m\Vert_{L^{1}(\mathbb R^n)}=M \right\}:  P(m,w)<0 \Big\}.
\end{equation*}
For any $h(t)\in \mathcal{A}\cap \left\{m: \Vert m\Vert_{L^{1}(\mathbb R^n)}=M \right\}$, $t\in[0,1]$, given in \eqref{mounpass_path}, we get by using \eqref{Gn_Inequality} and \eqref{poho_equality} that
\begin{align*}
  P(h(0))&= P(m_{1},w_{1})=\int_{\mathbb{R}^n}m_{1}L\left(-\frac{w_{1}}{m_{1}}\right)\,dx-\frac{n\alpha}{(1+\alpha)\gamma'}\int_{\mathbb{R}^n}{m_{1}}^{1+\alpha}\,dx\\
  &\geq \int_{\mathbb{R}^n}m_{1}L\left(-\frac{w_{1}}{m_{1}}\right)\,dx-\frac{n\alpha}{(1+\alpha)\gamma'{\Gamma_{\alpha}}}\left(\int_{\mathbb{R}^n}m_{1}L\left(-\frac{w_{1}}{m_{1}}\right)\,dx\right)^{\frac{n\alpha}{\gamma'}}> 0.
\end{align*}
The last inequality follows from the facts that $\alpha>\alpha_{*}=\frac{\gamma'}{n}$ and $(m_{1},w_{1})\in B_{R_{0}}$, where $R_{0}>0$ is chosen small enough. Hence $h(0)\in \mathcal{P}^{+}$. Meanwhile, by using \eqref{mountain_pass_level_00} and applying $\alpha>\alpha_{*}=\frac{\gamma'}{n}$ again, one has
\begin{align*}
  P(h(1))&= P(m_{2},w_{2})=\int_{\mathbb{R}^n}m_{2}L\left(-\frac{w_{2}}{m_{2}}\right)\,dx-\frac{n\alpha}{(1+\alpha)\gamma'}\int_{\mathbb{R}^n}{m_{2}}^{1+\alpha}\,dx\\
  &\leq \int_{\mathbb{R}^n}m_{2}L\left(-\frac{w_{2}}{m_{2}}\right)\,dx-\frac{1}{(1+\alpha)
}\int_{\mathbb{R}^n}{m_{2}}^{1+\alpha}\,dx=\mathcal{J}_{0}(m_{2},w_{2})<0.
\end{align*}
Thus, $h(1)\in \mathcal{P}^{-}$.  Consequently, using the continuity of $h(t)$ on $t\in[0,1]$, we find that there exists $t_{0}\in(0,1)$ such that $P(h(t_{0}))=0$, i.e., $h(t_{0})\in \mathcal{P}$. Therefore, we get
\begin{equation}\label{mountapass_level}
  \max_{t\in[0,1]}\mathcal{J}_{0}(h(t))\geq \mathcal{J}_{0}(h(t_{0}))\geq \inf_{(m,w)\in \mathcal{P}}\mathcal{J}_{0}(m,w)=e_{P}.
\end{equation}
Taking the infimum over $h\in\Pi$ in \eqref{mountapass_level}, we have
 $e_{MP}\geq e_{P}$. On the other hand,  we claim that the inequality  $e_{MP}\leq e_{P}$ holds. Indeed, arguing by contradiction, assume that there is $(\bar m, \bar w)\in \mathcal{P}$ such that $\mathcal{J}_{0}(\bar m, \bar w)<e_{MP}$. From Remark \ref{mountain_pass_geometry_1}, we know that there exist $0<\bar t_{1}<\bar t_{2}$ such that $(\bar m_{\bar t_{1}}, \bar w_{\bar t_{1}})\in B_{R_{0}}$ and $(\bar m_{\bar t_{2}}, \bar w_{\bar t_{2}})\in \overline{B_{2R_{0}}^{\,c}}\, \cap \left\{(m,w):\mathcal{J}_{0}(m,w)\leq 0\right \}$. Now, we set the path $\bar g(t)=( \bar m_{(1-t)\bar t_{1}+t\bar t_{2}}, \bar w_{(1-t)\bar t_{1}+t\bar t_{2}})\in \widetilde \Pi$, where $\widetilde \Pi$ is defined by \eqref{mounpass_path_tilde}. Then, from the equivalent definition of $e_{MP}$ shown in \eqref{mountain_pass_level_001}, one gets
 \begin{equation*}
	e_{MP}\leq \max_{t\in[0,1]}\mathcal{J}_{0}(\bar g(t))=\max_{t\in[0,1]}\mathcal{J}_{0}( \bar m_{(1-t)t_{1}+tt_{2}}, \bar w_{(1-t)t_{1}+tt_{2}})=\mathcal{J}_{0}(\bar m, \bar w)
 \end{equation*}
 which is a contradiction. Therefore, we obtain  $e_{MP}= e_{P}$.  This together with \eqref{eq-eP}  finishes the proof of this lemma.
\end{proof}

\begin{lemma}\label{Gn_Re_Mp}
	  Let $\alpha\in(\alpha_{*},\alpha^{*})$ and 
    $(\bar{m}_{0},\bar{w}_{0})\in \mathcal{A}\cap \left\{m: \Vert m\Vert_{L^{1}(\mathbb R^n)}=M \right\}$ be  the optimizer of\, $\Gamma_{\alpha}$ established in Lemma \ref{resolMFG}.  Then, 
	\begin{align*}
		\mathcal{J}_{0}(\bar m_{0},\bar w_{0})=\max_{t\in\mathbb{R}^{+}}\mathcal{J}_{0}(\bar m_{0t}, \bar w_{0t})= e_{MP}.
	\end{align*}
\end{lemma}
\begin{proof}
From equation \eqref{L^1_constrain_equation} and the Pohozaev identity established in \cite[Lemma 3.7]{cirant2025critical}, one see that
	\begin{equation*}
	\int_{\mathbb{R}^{n}}\bar m_{0}L\left(-\frac{\bar w_{0}}{\bar m_{0}} \right)\,dx=\frac{n\alpha}{(1+\alpha)\gamma'}\int_{\mathbb{R}^n}{\bar m_{0}}^{1+\alpha}\,dx.
    \end{equation*}
 This means $P(\bar m_{0},\bar w_{0} )=0$, where $P(\cdot,\cdot)$ is defined in \eqref{poho_equality}. Therefore, we get
\begin{equation*}
 \max_{t\in \mathbb{R}^{+}}\mathcal{J}_{0}(\bar m_{0,t},\bar w_{0, t})=\mathcal{J}_{0}(\bar m_{0},\bar w_{0}),
\end{equation*}
which, together with Lemma \ref{Equivalent_mountain_level}, implies
\begin{equation}\label{min_max_J}
	e_{MP}=\inf_{(m,w)\in \mathcal{A}\cap \left\{m: \Vert m\Vert_{L^{1}(\mathbb R^n)}=M \right\}}\max_{t\in \mathbb{R}^{+}}\mathcal{J}(m_{t},w_{t})\leq\mathcal{J}_{0}(\bar m_{0},\bar w_{0}).
\end{equation}
On the other hand, it follows from \eqref{deduce_GN} that, for any $(m,w)\in\mathcal{A}\cap \left\{m: \Vert m\Vert_{L^{1}(\mathbb R^n)}=M \right\}$, there is a unique $t_{m,w}>0$ such that $P(m_{t_{m,w}}, w_{t_{m,w}})=0$ and
\begin{equation}\label{min_max_fun}
 \max_{t\in \mathbb{R}^{+}}\mathcal{J}_{0}( m_{t}, w_{t})=\mathcal{J}_{0}( m_{t_{m,w}},w_{t_{m,w}})=\frac{n\alpha-\gamma'}{(1+\alpha)\gamma'}\int_{\mathbb{R}^n}{m_{t_{m,w}}}^{1+\alpha}\,dx.
\end{equation}
By applying \eqref{Gn_Inequality} and $P(m_{t_{m,w}}, w_{t_{m,w}})=0$, we have
\begin{equation}\label{min_max_fun_2}
	\int_{\mathbb{R}^n}{m_{t_{m,w}}}^{1+\alpha}\,dx\geq { \Gamma_{\alpha}}^{\frac{\gamma'}{n\alpha-\gamma'}}	\left(\frac{(1+\alpha)\gamma'}{n\alpha}\right)^{\frac{n\alpha}{n\alpha-\gamma'}}.
\end{equation}
However, using the fact that $(\bar{m}_{0},\bar{w}_{0})\in \mathcal{A}\cap \left\{m: \Vert m\Vert_{L^{1}(\mathbb R^n)}=M \right\}$ minimizes \eqref{Gn_Inequality} and $P(\bar{m}_{0},\bar{w}_{0})=0$, one has
\begin{equation}\label{min_max_fun_5}
	\mathcal{J}_{0}( \bar m_{0},\bar w_{0})=\frac{n\alpha-\gamma'}{(1+\alpha)\gamma'}\int_{\mathbb{R}^n}{\bar m_{0}}^{1+\alpha}\,dx
\,\,\text{and}\,\,
	\int_{\mathbb{R}^n}{\bar m_{0}}^{1+\alpha}\,dx= { \Gamma_{\alpha}}^{\frac{\gamma'}{n\alpha-\gamma'}}	\left(\frac{(1+\alpha)\gamma'}{n\alpha}\right)^{\frac{n\alpha}{n\alpha-\gamma'}}.
\end{equation}
Therefore, it follows from \eqref{min_max_fun}, \eqref{min_max_fun_2} and \eqref{min_max_fun_5} that
 \begin{equation*}
	\mathcal{J}_{0}( \bar m_{0},\bar w_{0}) \leq  	\mathcal{J}_{0}( m_{t_{m,w}},w_{t_{m,w}}) =\max_{t\in \mathbb{R}^{+}}\mathcal{J}_{0}( m_{t}, w_{t}),\,\,\text{for any}\,\, (m,w)\in\mathcal{A}\cap \left\{m: \Vert m\Vert_{L^{1}(\mathbb R^n)}=M \right\},
\end{equation*}
which yields
\begin{equation}\label{min_max_fun_7}
	\mathcal{J}_{0}( \bar m_{0},\bar w_{0}) \leq  \inf_{(m,w)\in\mathcal{A}\cap \left\{m: \Vert m\Vert_{L^{1}(\mathbb R^n)}=M \right\}}\max_{t\in \mathbb{R}^{+}}\mathcal{J}_{0}( m_{t}, w_{t})=e_{MP}.
\end{equation}
Consequently, we conclude from \eqref{min_max_J} and \eqref{min_max_fun_7} that $\mathcal{J}_{0}( \bar m_{0},\bar w_{0}) =e_{MP},$ and the desired conclusion  follows.
\end{proof}

We are now in a position to prove Theorem \ref{Existience_of_MPS_MFG} by combining the results established above. The theorem is stated as follows:

\vspace{0.2cm}

\textbf{Proof of Theorem \ref{Existience_of_MPS_MFG}:} First of all, by Lemma \ref{resolMFG}, we obtain that there exists $\big(\bar u_{0},\bar m_{0},\bar \lambda_{0}\big)\in C^2(\mathbb R^n)\times  W^{1,p}(\mathbb R^n)\times  \mathbb R$, $\forall p>1$, which solves MFG system \eqref{L^1_constrain_equation}, where the pair  $(\bar m_{0},\bar w_{0})\in \mathcal{A}\cap \left\{m: \Vert m\Vert_{L^{1}(\mathbb R^n)}=M \right\}$, and $\bar w_{0}=- \bar m_{0}\nabla H(\nabla {\bar u}_0)$, is a minimizer of $\Gamma_{\alpha}$ defined in \eqref{Gn_Inequality}. Next, it follows from Lemma \ref{MP_1} that functional $\mathcal{J}_{0}$ associated with MFG system \eqref{L^1_constrain_equation} admits a mountain-pass geometry on the admissible set $\mathcal{A}\cap \left\{m: \Vert m\Vert_{L^{1}(\mathbb R^n)}=M \right\}$. Moreover, invoking Lemma \ref{Equivalent_mountain_level}, we find that the mountain-pass level $e_{MP}$ defined on $\mathcal{A}\cap \left\{m: \Vert m\Vert_{L^{1}(\mathbb R^n)}=M \right\}$ admits several equivalent characterizations. Finally, by using the equivalent characterizations of $e_{MP}$, we obtain from Lemma \ref{Gn_Re_Mp} that 
\begin{equation*}
e_{MP}=\inf_{(m,w)\in \mathcal{A}\cap \left\{m: \Vert m\Vert_{L^{1}(\mathbb R^n)}=M \right\}}\max_{t\in \mathbb{R}^{+}}\mathcal{J}(m_{t},w_{t})=\mathcal{J}(\bar m_{0},\bar w_{0}).
\end{equation*}
This indicates that the minimizer $(\bar m_{0},\bar w_{0})$ corresponding to ${\Gamma}_{\alpha}$ attains the mountain pass level $e_{MP}$ of $\mathcal{J}_{0}$ in $\mathcal{A}\cap \left\{m: \Vert m\Vert_{L^{1}(\mathbb R^n)}=M \right\}$. Therefore, we conclude that $\big(\bar u_{0},\bar m_{0},\bar \lambda_{0}\big)$ is a mountain-pass solution to the MFG system \eqref{L^1_constrain_equation}, which completes the proof of Theorem \ref{Existience_of_MPS_MFG}.
$\hfill\qed$
{
 \begin{remark}\label{MP_Ground}
 By Lemma~\ref{exi_potential_free_MFG}, Lemma~\ref{resolMFG}, and Lemma~\ref{Gn_Re_Mp}, we obtain 
 $$e_{MP}=\frac{n\alpha-\gamma'}{(1+\alpha)\gamma'}
 \left[\left( 1 - \frac{n \alpha}{(1+\alpha) \gamma'} \right)M^{-1}\right]^{\frac{(1 + \alpha )\gamma' -n\alpha }{n \alpha - \gamma'}}
 \; 
 e_0^{\frac{ (1+\alpha)\gamma'}{n \alpha- \gamma'}}.$$  Combining this with \eqref{e_0_min} and $\alpha\in(\alpha_{*},\alpha^{*})$, we conclude that the solution to
 \eqref{L^1_constrain_equation} given by Theorem~\ref{Existience_of_MPS_MFG}
  is a ground state. 
 \end{remark}     }

Finally, using the results established in Lemma \ref{sect3_lemm32}, Lemma \ref{resolMFG}, and Theorem \ref{Existience_of_MPS_MFG}, we prove Theorem \ref{GNinequalitythm}, which is

\vspace{0.2cm}

 \textbf{Proof of Theorem \ref{GNinequalitythm}:}  Noting that  $(m_{0}, w_{0})$ is  an optimizer of $\Gamma_{\alpha}$ obtained in Lemma \ref{sect3_lemm32}.   In addition, under the mass-supercritical regime, it follows from Lemma \ref{resolMFG} that the infimum $\Gamma_{\alpha}$ is attained at a pair $(\bar m_{0}, \bar w_{0}) \in \mathcal{A}\cap \left\{m: \Vert m\Vert_{L^{1}(\mathbb R^n)}=M \right\}$, which is obtained via a scaling of $(m_{0}, w_{0})\in \mathcal{A}\cap \left\{m: \Vert m\Vert_{L^{1+\alpha}(\mathbb R^n)}=1 \right\}$. Moreover, $\big(\bar u_{0},\bar m_{0},\bar \lambda_{0}\big)\in C^2(\mathbb R^n)\times  W^{1,p}(\mathbb R^n)\times\mathbb R$, $\forall p>1$, is the solution to \eqref{L^1_constrain_equation}.   On the other hand, in light of Theorem \ref{Existience_of_MPS_MFG}, the pair $(\bar m_{0}, \bar w_{0})$ achieves the mountain-pass level $e_{MP}$ of the functional $\mathcal{J}_{0}$, thereby establishing Theorem \ref{GNinequalitythm}.
$\hfill\qed$


\section{Conclusions}
 
In this paper, we study mountain-pass solutions to a viscous ergodic MFG system with aggregating coupling.  Using a variational approach, we consider the $L^{1+\alpha}$-constrained problem and establish the existence of minimizers via the direct method.  Technically, we first consider the coercive potential case and then take the vanishing potential limit to establish the existence of solutions for the auxiliary potential-free MFG system. Moreover, using a scaling argument, we show that the potential-free MFG with an $L^1$
  constraint admits a classical solution.  In addition, the equivalence of energies implies that this solution is of mountain-pass type. The key ingredients are the two-stage linearization of the $L^{1+\alpha}$
  constraint while showing the existence of value function $u$ and verifying the relationship between $w$ and $u$.  An additional Lagrange multiplier arises in this procedure and is shown to vanish by means of a scaling argument.

It is worthy mentioning that  there are several interesting questions that deserve future investigation. Firstly, on a technical side, we reformulate the original problem as an $L^{1+\alpha}$-constrained problem and restrict our attention to polynomial-growth coupling costs, i.e., $f(m)=-m^{\alpha}$ for some $\alpha>0.$  Relaxing the assumption on the coupling cost is an intriguing problem, but it poses significant technical challenges, as the $L^{1+\alpha}$ constraint is no longer applicable. Another interesting but challenging direction is the investigation of the properties of solutions, such as radial symmetry, uniqueness and positivity.  Since the MFG system is coupled, novel techniques are required to classify its solutions. In addition, the stability of the mountain-pass solutions obtained here remains an open problem. In contrast to the existence of local minimizers established in \cite{cirant2023ergodic}, our construction yields solutions of mountain-pass type. Investigating and comparing the stability properties of these two distinct classes of solutions would be an interesting direction for future research.


\section*{Acknowledgments}
We thank Professor M. Cirant for many stimulating discussions and insightful suggestions. Xiaoyu Zeng is supported by NSFC (Grant Nos. 12322106, 12171379, 12271417) .

\begin{appendices}
\setcounter{equation}{0}
\renewcommand\theequation{A.\arabic{equation}}

\section{Proof of Proposition \ref{cover_0_super}}\label{cover_0_superappen}
In this appendix, we provide the proof of Proposition~\ref{cover_0_super}, which concerns the classification of solutions to the MFG system \eqref{L^1_constrain_equation} arising from minimizers of \eqref{Gn_Inequality} obtained by $L^{1+\alpha}$-constraint minimization in terms of $\alpha \in (0,\alpha^{*})$.

\medskip 

\textbf{Proof of Proposition \ref{cover_0_super}:} First, Conclusion~$(iii)$ of Proposition~\ref{cover_0_super} follows directly from
Theorem~\ref{Existience_of_MPS_MFG} and Remark~\ref{MP_Ground}.

To prove Conclusion~$(i)$, note that $(m_{0}, u_{0}, \mu_{0})$ is a solution of
\eqref{mass_supercritical_problem_free_potential} and that $(m_{0}, w_{0})$ is an
optimizer of \eqref{Gn_Inequality}. Applying the scaling defined in
\eqref{eqs1}--\eqref{eqs2} to $(m_{0}, w_{0}, u_{0}, \mu_{0})$, we obtain the
existence of a pair $(\hat m, \hat w) \in \mathcal{A} \cap \left\{ m : \|m\|_{L^{1}(\mathbb{R}^n)} = M \right\}$ that minimizes $\Gamma_{\alpha}$ for $\alpha \in (0,\alpha_{*})$.
Correspondingly, $(\hat m, \hat u, \hat \lambda)$ is a solution of
\eqref{L^1_constrain_equation}.

Let
\[
\bar e_{0,\alpha,M}
:= \inf_{(m,w)\in \mathcal{A}\cap \left\{m: \|m\|_{L^{1}(\mathbb{R}^n)}=M \right\}}
\mathcal{J}_{0}(m,w).
\]
From \cite{cesaroni2018concentration}, we know that $\bar e_{0,\alpha,M}$ is
well defined for any $\alpha \in (0,\alpha_{*})$. Hence, it suffices to show that
$\mathcal{J}_{0}(\hat m, \hat w) = \bar e_{0,\alpha,M}$.  Suppose, by contradiction, that there exists a solution
$(\hat m_{1}, \hat u_{1}, \hat \lambda_{1})$ to \eqref{L^1_constrain_equation}
such that
$(\hat m_{1}, \hat w_{1}) \in \mathcal{A} \cap \left\{ m : \|m\|_{L^{1}(\mathbb{R}^n)} = M \right\}$,
where $\hat w_{1} = - \hat m_{1} \nabla H(\nabla \hat u_{1})$, and $\mathcal{J}_{0}(\hat m_{1}, \hat w_{1})
< \mathcal{J}_{0}(\hat m, \hat w)$.  Then, by a direct computation and applying the   scaling to
$(\hat m_{1}, \hat w_{1})$, we can construct a pair $(m_{1}, w_{1}) \in \mathcal{A} \cap \left\{ m : \|m\|_{L^{1+\alpha}(\mathbb{R}^n)} = 1 \right\}$
such that $\mathcal{E}_{0}(m_{1}, w_{1}) < e_{0}$,
which contradicts the minimality of $e_{0}$. This completes the proof of
Conclusion~$(i)$.

Finally, we prove Conclusion $(ii)$. 
Let $\hat m(x):=sm_{0}(rx)$, $\hat u(x):=tu_{0}(rx)$, then by Pohozaev identities \eqref{eq2.49}, one has 
\begin{align}\label{eqrt1}
	\left\{\begin{array}{ll}
		-\Delta \hat u+t^{1-\gamma}r^{2-\gamma} H(\nabla  {\hat u})+ s^{-\frac{\gamma'}{n}} tr^{2}\mu_{0} \hat m^{\alpha^{*}}=tr^{2},&x\in\mathbb R^n,\\
		\Delta \hat m+t^{1-\gamma}r^{2-\gamma} \nabla\cdot (\hat m\nabla H(\nabla \hat u))=0,&x\in\mathbb R^n,\\
		\hat w=- \hat  m\nabla H(\nabla \hat u),\ \frac{s^{-1}r^{n}(n+\gamma')}{\mu_{0}\gamma'}\int_{\mathbb R^n}\hat m\,dx=1.
	\end{array}
	\right.
\end{align}
Then, setting $t^{1-\gamma} r^{2-\gamma} = 1,\,
s^{-\frac{\gamma'}{n}} t r^{2} \mu_{0} = 1 \, \text{and} \,
\frac{s^{-1} r^{n} (n+\gamma')}{\mu_{0} \gamma'} = \frac{1}{M}$,
we deduce that $M = M^{*} := \Big(\Gamma_{\alpha_{*}} (1+\alpha_{*}) \Big)^{\frac{1}{\alpha_{*}}}$
must hold, with \(r\) remaining a free parameter. Using a limiting argument and (4.15) in \cite{cirant2025critical}, it follows that when \(\alpha = \alpha_*\) and \(M = M^*\), the Lagrange multiplier \(\lambda\) in system \eqref{L^1_constrain_equation} is uniquely determined as $
\lambda = - \frac{\gamma'}{n M^*}.$
Therefore, $t r^{2} = \frac{\gamma'}{n M^*}.$
Consequently, combining the above equalities and the fact $\mu_0=e_{0}=\left(\frac{n+\gamma'}{\gamma'}M^{*}\right)^{\frac{\gamma'}{n+\gamma'}}$, we have
$$
s_1:=\left(\frac{\gamma'}{n}\right)^{\frac{n}{\gamma'}}\left(\frac{n+\gamma'}{\gamma'}\right)^{\frac{n}{n+\gamma'}}{M^{*}}^{-\frac{n^2}{\gamma'(n+\gamma')}}, \,t_{1}:=\left(\frac{\gamma'}{nM^{*}}\right)^{\frac{\gamma'-2}{\gamma'}}\,\text{and}\,\,r_{1}:=\left(\frac{\gamma'}{nM^{*}}\right)^{\frac{1}{\gamma'}}.
$$
Thus, we obtain a triple $(\hat u (x),\hat m(x),\hat \lambda)=(t_1u_{0}(r_1 x), s_1m_{0}(r_1 x),-\frac{\gamma'}{n M^*})$   solving \eqref{L^1_constrain_equation} with $\alpha=\alpha_{*}$ and $M=M^{*}$.  It is straightforward to verify that $(\hat m,\hat w)\in \mathcal{A}\cap \left\{m: \Vert m\Vert_{L^{1}(\mathbb R^n)}=M^{*} \right\}$, $\mathcal{J}_{0}(\hat m,\hat w)=0$ and   $\mathcal{J}_{0}(m,w)\geq 0$ for any  $(m,w)\in \mathcal{A}\cap \left\{m: \Vert m\Vert_{L^{1}(\mathbb R^n)}=M^{*} \right\}$.  Consequently, $(\hat  m,\hat w)$ is a global minimizer of $\mathcal{J}_{0}$ in $\mathcal{A}\cap \left\{m: \Vert m\Vert_{L^{1}(\mathbb R^n)}=M^{*} \right\}$. In addition, if $M<M^{*}$, we claim that there is no classical solution of  \eqref{L^1_constrain_equation}.  Indeed, if there exists $\big(u_{1}, m_{1}, \lambda_1\big)\in C^2(\mathbb R^n)\times  W^{1,p}(\mathbb R^n)\times  \mathbb R$, $\forall p>1$, solving MFG system \eqref{L^1_constrain_equation} with $M<M*$. Then by the Gagliardo-Nirenberg inequality \eqref{Gn_Inequality} with $\alpha=\alpha_{*}$ and Pohozaev identities, we have
\begin{equation*}
\int_{\mathbb{R}^n}m_{1}^{1+\frac{\gamma'}{n}}\,dx\leq \Gamma_{\alpha_{*}}^{-1}\left(\int_{\mathbb{R}^n}m_{1}L\left(-\frac{w_1}{m_1}\right)dx\right)M^{\frac{\gamma'}{n}}= \Gamma_{\alpha_{*}}^{-1}\left(\frac{n}{n+\gamma'}\int_{\mathbb{R}^n}m_{1}^{1+\frac{\gamma'}{n}}\,dx\right)M^{\frac{\gamma'}{n}},
\end{equation*}
which yields 
\begin{equation*}
\Gamma_{\alpha_{*}}\leq \frac{n}{n+\gamma'}M^{\frac{\gamma'}{n}}.
\end{equation*}
Combining this with $M^{*}=\Big(\Gamma_{\alpha_{*}}(1+\alpha_{*})\Big)^{\frac{1}{\alpha_{*}}}$, we obtain $M\geq M^{*}$, which is a contradiction. Thus, the claim mentioned above holds.
Finally, if $M>M^{*}$, assume that $(m,w)\in \mathcal{A}\cap \left\{m: \Vert m\Vert_{L^{1}(\mathbb R^n)}=M \right\}$ is an optimizer of $\Gamma_{\alpha_{*}}$, then we have $\mathcal{J}_{0}(m,w)<0$ when $\alpha=\alpha^*$, where $\mathcal{J}_{0}(m,w)$ is given in \eqref{L^1_constrain_functional}.
Hence, $$\mathcal{J}_0\big(t^n m(tx),\, t^{n+1} w(tx)\big) = t^{\gamma'} \mathcal{J}_0(m, w) < 0,$$
which decreases monotonically as $t$ increases. This indicates that  \(\mathcal{J}_0\) does not admit a global minimizer, and that  the mountain pass geometry condition fails for $\mathcal{J}_0$ on
$\mathcal{A} \cap \big\{ m : \|m\|_{L^1(\mathbb{R}^n)} = M > M^* \big\}.$
\end{appendices}

\bibliographystyle{abbrv}
\bibliography{ref}

\end{document}